\documentclass[11pt]{article}

\usepackage[margin=1.15in]{geometry}
\usepackage{amsmath,amssymb,amsthm,mathtools}
\usepackage{mathrsfs}
\usepackage{enumitem}
\usepackage{hyperref}
\usepackage{stmaryrd}
\usepackage{xcolor}

\numberwithin{equation}{section}

\newtheorem{theorem}{Theorem}[section]
\newtheorem{proposition}[theorem]{Proposition}
\newtheorem{lemma}[theorem]{Lemma}
\newtheorem{corollary}[theorem]{Corollary}
\newtheorem{definition}[theorem]{Definition}
\newtheorem{remark}[theorem]{Remark}

\newcommand{\R}{\mathbb R}

\newcommand{\spt}{\operatorname{spt}}
\newcommand{\mres}{\mathbin{\llcorner}}

\title{Perturbative Schauder Estimates for \(Q\)-Valued Quasilinear Elliptic Systems and a Sharp Dimension Bound for Branch Sets of Stationary Graphs}

\author{Mattia Luchese}

\begin{document}

\date{}
\maketitle

\begin{abstract}
We establish a priori interior \(C^{1,\alpha}\) and \(C^{2,\alpha}\) Schauder estimates for a class of \(Q\)-valued quasilinear elliptic systems which are perturbations of the Laplace system, for arbitrary multiplicity \(Q\), domain dimension \(n\geq2\), and target dimension \(k\). The \(C^{1,\alpha}\) estimate generalises the work of Simon and Wickramasekera in \cite{SW16} on \(2\)-valued solutions of linear systems and concerns weak solutions of divergence-form systems, while the \(C^{2,\alpha}\) estimate concerns strong solutions. In dimension \(n=2\), both estimates hold for every \(0<\alpha<1/Q\). In dimensions \(n\geq3\), there exists \(\delta=\delta(n,k,Q)>0\) such that both estimates hold for every \(0<\alpha<\delta\).

As applications, we obtain a small-slope Schauder estimate and a small-slope Bernstein theorem for \(Q\)-valued maps whose graph varifolds are stationary. Combining the Schauder estimate with the recent branch-set stratification theory of Krummel--Minter--Wickramasekera in \cite{KMW26}, we further prove that, for every \(\gamma>0\), the branch set \(\mathcal B_u\) of any \(C^{1,\gamma}\) \(Q\)-valued map \(u\) whose graph varifold is stationary satisfies $\dim_{\mathcal H}\mathcal B_u\leq n-2$.
This bound is sharp already in codimension one.
\end{abstract}

\tableofcontents

\section{Introduction}

We develop an interior Schauder theory for \(Q\)-valued quasilinear elliptic systems which are perturbations of the Laplace system. The first result concerns \(C^{1,\alpha}(B^n_1,\mathcal A_Q(\R^k))\) weak solutions of divergence-form systems
\[
\Delta u+\operatorname{div}_x\mathscr E(x,u,Du)+\mathscr B(x,u,Du)=f,
\]
where \(f\in L^\infty(B^n_1,\R^k)\) is single-valued and \(\mathscr B(x,0,0)=0\).
Under natural regularity assumptions on the coefficients, and provided the principal perturbation \(\mathscr E\) is sufficiently small (see hypothesis \((\star)\) in Section~\ref{sec:first-order-schauder}), Theorem~\ref{thm:quasilinear-Schauder} gives, for solutions satisfying
\(\sup_{B_1^n}|Du|\leq L\), the estimate
\[
\|u\|_{C^{1,\alpha}(B_{1/2}^n)}
\leq
C\left(
\|u\|_{L^2(B_1^n)}
+
\|f\|_{L^\infty(B_1^n)}
\right),
\]
where the constant \(C\) also depends on \(L\).
The admissible range is
\[
0<\alpha<\frac1Q
\quad\text{if }n=2,
\qquad
0<\alpha<\delta(n,k,Q)
\quad\text{if }n\geq3,
\]
where \(\delta(n,k,Q)>0\) is the homogeneity gap established in Theorem~\ref{thm:homogeneity-one-gap}.

We also prove a second-order analogue for \(C^{2,\alpha}(B^n_1,\mathcal A_Q(\R^k))\) strong solutions of systems
\[
\Delta u+\sum_{i,j=1}^n\mathscr A^{ij}(x,u,Du)D_{ij}u+\mathscr C(x,u,Du)=f,
\]
where \(f\in C^{0,\alpha}(B^n_1,\R^k)\) is single-valued and \(\mathscr C(x,0,0)=0\). Under analogous regularity assumptions on the coefficients, and provided the principal perturbation \(\mathscr A\) is sufficiently small (see hypothesis \((\star\star)\) in Section~\ref{sec:second-order-schauder}), Theorem~\ref{thm:quasilinear-C2-Schauder} gives, for solutions satisfying \(\sup_{B_1^n}|Du|\leq L\), the estimate
\[
\|u\|_{C^{2,\alpha}(B_{1/2}^n)}
\leq
C\left(
\|u\|_{L^2(B_1^n)}
+
\|f\|_{C^{0,\alpha}(B_1^n)}
\right),
\]
where the constant \(C\) also depends on \(L\).
The admissible range is \(0<\alpha<1/Q\) when \(n=2\), and \(0<\alpha<\delta(n,kn,Q)\) when \(n\geq3\). The target dimension \(kn\) appears because the limiting Liouville argument is applied to the differential, which takes values in \(\operatorname{Hom}(\mathbb R^n,\mathbb R^k)\).

\(C^{1,\alpha}\) Schauder estimates in the \(Q\)-valued setting were first established by Simon and Wickramasekera in \cite{SW16} for \(2\)-valued symmetric solutions of linear systems. The general scheme used here, in which a homogeneity gap yields a Liouville theorem and the Liouville theorem yields Schauder estimates, is the same as in \cite{SW16}. Krummel obtained analogous estimates for arbitrary multiplicity in a setting with a fixed codimension-two branching axis \cite{Kru19}. It should be remarked that in these earlier works the formulation of the equations is considerably simpler than in the present setting. In \cite{SW16}, the symmetric part of a \(2\)-valued solution of the Euler--Lagrange system of the area functional already satisfies a linear system with coefficients depending only on \(x\). In the general setting of the present paper, by contrast, to include the area system one must consider genuine quasilinear systems (see Remark~\ref{rem:Q-valued-weak-solutions-examples}\textup{(iv)}). In \cite{SW16}, the interpretation of weak solutions is also more straightforward, since away from the branch set the solution is locally plus or minus a single-valued \(C^{1,\alpha}\) branch. Similarly, in \cite{Kru19}, a global selection of single-valued \(C^{1,\alpha}\) branches is possible after cutting the domain along a half-hyperplane bounded by the branching axis, so the weak equations can be interpreted sheet by sheet. Here, by contrast, care must be taken even in defining an appropriate notion of \(Q\)-valued weak solution.

Following this approach, we prove both estimates by contradiction. A failure of either estimate produces, after rescaling, an entire Dirichlet-stationary map whose differential retains a nonzero oscillation. This map is the blow-up itself in the \(C^{1,\alpha}\) case and the differential of the blow-up in the \(C^{2,\alpha}\) case. The contradiction then follows from the following Liouville theorem (Theorem~\ref{thm:liouville}): if
\[
v\in C^{1,\alpha}\bigl(\mathbb R^n,\mathcal A_Q(\mathbb R^k)\bigr)
\]
is Dirichlet-stationary and satisfies
\[
[Dv]_{\alpha;\mathbb R^n}<\infty,
\]
then \(v\) is affine, provided
\[
0<\alpha<\frac1Q
\quad\text{if }n=2,
\qquad
0<\alpha<\delta(n,k,Q)
\quad\text{if }n\geq3.
\]

Section~\ref{sec:preliminaries} is devoted to the foundational \(C^1\), Sobolev, compactness, and local splitting theory for \(Q\)-valued functions needed in the proofs. The results collected there are natural analogues of familiar facts from the single-valued theory. Their multiple-valued formulations nevertheless require some care, and the corresponding arguments are rather technical. A parallel calculus for second jets is developed separately later in the paper, since it is needed only for the \(C^{2,\alpha}\) theory.

Beyond this foundational material, the paper contains three new ideas. These ideas are largely modular. The first is used in every dimension and underlies both Schauder estimates. The remaining two are used principally to establish the higher-dimensional homogeneity gap which supplies the Liouville theorem needed in the blow-up arguments.

The first new idea is a clustered blow-up argument for multiple-valued Schauder estimates. Once a suitable Liouville theorem is available, Simon's scaling method reduces the estimate to the analysis of a normalized blow-up sequence. The blueprint for the first-order argument is the proof presented in \cite{SW16} for the case \(Q=2\). For general \(Q\), however, different groups of sheets may separate at different rates in the relevant jet space. To deal with this situation, we introduce a finer clustering procedure which partitions the sheets according to their relative (first or second order) jet separation at the blow-up scale.

The second new idea is a Hodge-type structure for the differential of Dirichlet-stationary \(Q\)-valued maps. We first prove an inductive Bochner inequality which implies that every \(C^1\) map which is stationary with respect to outer variations belongs locally to \(W^{2,2}\). We then show that, if
\[
v\in W_{\mathrm{loc}}^{2,2}\bigl(\Omega,\mathcal A_Q(\mathbb R^k)\bigr)
\]
is stationary with respect to outer variations, then both \(v\) and its differential
\[
Dv\in W_{\mathrm{loc}}^{1,2}\bigl(\Omega,\mathcal A_Q(\mathbb R^{k\times n})\bigr)
\]
are Dirichlet-stationary. Moreover, if
\[
v=\sum_{\ell=1}^Q\llbracket v^\ell\rrbracket,
\]
then, for almost every point, every sheet index \(\ell\in\{1,\ldots,Q\}\), and every target component \(a\in\{1,\ldots,k\}\), the matrix
\[
D^2v^{\ell,a}\in\mathbb R^{n\times n}
\]
is symmetric and trace-free. This Hodge-type structure yields, sheet by sheet, the proper aperture estimate
\[
D^2v^\ell(D^2v^\ell)^*
\leq \frac{n-1}{n}|D^2v^\ell|^2I.
\]
The resulting aperture bound is the structural input for the higher-dimensional homogeneity gap, while the same Hodge structure also reappears in the second-order blow-up argument.

The third new idea is a target-stress argument for the higher-dimensional homogeneity gap, based on the following monotonicity formula for positive-semidefinite matrix-valued measures. Let
\[
\mathsf S\in\mathcal M\bigl(B^m_R;\operatorname{Sym}^+(\mathbb R^m)\bigr),
\qquad
\mu:=\operatorname{tr}\mathsf S,
\]
and write
\[
d\mathsf S(y)=P(y)\,d\mu(y),
\qquad
\operatorname{tr}P(y)=1
\quad\text{for \(\mu\)-a.e. }y.
\]
Suppose that, for some \(c\in(0,1]\),
\[
P(y)\leq c I
\]
for \(\mu\)-almost every \(y\). If
\[
-\operatorname{div}\mathsf S=a\delta_0
\]
for some \(a\in\mathbb R^m\), including the divergence-free case \(a=0\), then
\[
r\longmapsto\frac{\mu(B^m_r(0))}{r^{1/c}}
\]
is nondecreasing on $(0,R)$. The atomic term does not contribute to the radial variations underlying the formula. If \(c<1\), the resulting upper growth estimate is incompatible with the linear lower growth forced by a nonzero atomic load. In particular, necessarily \(a=0\).

The divergence-free case of this monotonicity formula first appeared in \cite[Lemma~4.11]{Huang26} in the context of conductance maximisation for singularly distributed media, and may be viewed as a generalisation of the classical monotonicity formula for stationary varifolds. In the present paper, the version with atomic divergence is applied with \(m=kn\) in a contradiction argument proving the homogeneity gap at \(1\). Suppose that
\[
v_h\in W_{\mathrm{loc}}^{2,2}\bigl(\mathbb R^n,\mathcal A_Q(\mathbb R^k)\bigr)
\]
is a sequence of homogeneous Dirichlet-stationary maps of degrees \(\sigma_h>1\) satisfying \(\sigma_h\to1\), normalized so that
\[
\int_{\mathbb S^{n-1}}|Dv_h|^2=1.
\]
We consider what we call the normalized target-stress measures associated with \(Dv_h\):
\[
\mathsf S_h(dy)
:=
\frac1{\sigma_h-1}
\int_{B_1^n}
\sum_{\ell=1}^Q
D^2v_h^\ell(x)
\bigl(D^2v_h^\ell(x)\bigr)^*
\,\delta_{Dv_h^\ell(x)}(dy)\,dx.
\]
After passing to a subsequence, these measures converge weakly-* to a matrix-valued measure \(\mathsf S\). The Hodge structure of \(Dv_h\) gives the required aperture condition with \(c=\frac{n-1}{n}<1\), while the outer stationarity and homogeneity of \(Dv_h\) produce a nonzero atomic divergence for \(\mathsf S\). This contradicts the preceding exclusion principle coming from the monotonicity formula, yielding the gap
\[
\sigma=1
\qquad\text{or}\qquad
\sigma\geq1+\delta(n,k,Q)
\]
among homogeneities \(\sigma\geq1\).

In dimension \(2\), the higher-dimensional target-stress argument is unnecessary. The explicit homogeneity gap
\[
\sigma=1
\qquad\text{or}\qquad
\sigma\geq1+\frac1Q
\]
is obtained directly from the angular ordinary differential equation and the monodromy of its solutions.

The homogeneity gaps yield the Liouville theorem by considering the frequency at infinity and taking a blow-down. The global Hölder bound implies that the limiting homogeneity is at most \(1+\alpha\), while the gap excludes every value strictly between \(1\) and the relevant threshold. After establishing a global bound for the differential, its Dirichlet stationarity implies that it is constant. The clustered blow-up argument then converts the Liouville theorem into the Schauder estimates.

As an application of the first-order theorem, we obtain small-slope Schauder estimates for \(Q\)-valued maps with stationary graph varifold. Indeed, if
\[
u\in C^{1,\alpha}\bigl(B^n_1,\mathcal A_Q(\mathbb R^k)\bigr)
\]
has stationary graph varifold, then \(u\) is a weak solution of
\[
\Delta u+\operatorname{div}\mathscr E_{\mathrm{area}}(Du)=0,
\qquad
\mathscr E_{\mathrm{area}}(P):=D_P\sqrt{\det(I+P^TP)}-P.
\]
Since \(D_P\mathscr E_{\mathrm{area}}(0)=0\), the first-order theorem gives, for the admissible values of \(\alpha\),
\[
\|u\|_{C^{1,\alpha}(B^n_{1/2})}
\leq C\left(\int_{B^n_1}|u|^2\right)^{1/2},
\]
provided that \(\sup_{B^n_1}|Du|\leq\varepsilon_0\). Applying this estimate to rescalings of an entire graph also gives a small-slope Bernstein theorem: if \(u\in C_{\mathrm{loc}}^{1,\alpha}\bigl(\mathbb R^n,\mathcal A_Q(\mathbb R^k)\bigr)\), its graph varifold is stationary, and \(\sup_{\mathbb R^n}|Du|\) is sufficiently small, then \(u\) is affine.

Finally, we use the first-order estimate to obtain a dimension bound for branch sets of stationary multiple-valued graphs. Theorem~\ref{thm:branch-dimension} states that, if \(\Omega\subset\mathbb R^n\) is open, \(u\in C_{\mathrm{loc}}^{1,\alpha}\bigl(\Omega,\mathcal A_Q(\mathbb R^k)\bigr)\) for some \(\alpha>0\), and the graph varifold associated with \(u\) is stationary, then
\[
\dim_{\mathcal H}\mathcal B_u\leq n-2.
\]
No smallness assumption on \(Du\) is required here. The proof combines the Schauder estimate with the abstract branch-set stratification theorem of Krummel, Minter, and Wickramasekera in \cite{KMW26}. The main point is to verify the required \(\varepsilon\)-regularity property for the class of stationary varifolds which, after rotation, are locally represented by small-slope \(Q\)-valued graphs. This follows from the \(Q\)-valued analogue of the Arzelà--Ascoli theorem, a \(C^{1,\alpha}\) reparametrization argument, and the Schauder estimate.

The dimension bound is sharp. For \(k\geq 2\), this follows from a simple example: for \(Q\geq 2\), the complex curve
\[
\bigl\{(z,w)\in\mathbb C^2:w^Q=z^{Q+1}\bigr\}
\]
is locally the graph over the \(z\)-plane of a \(C^{1,1/Q}\) \(Q\)-valued map with a branch point at the origin. With its natural orientation, the curve defines an integral current which is locally area-minimizing by calibration. Thus its product with \(\mathbb R^{n-2}\) defines an \(n\)-dimensional stationary varifold in \(\mathbb R^{n+2}\) with branch set of dimension \(n-2\). The bound is also sharp in codimension \(k=1\), as shown by the branched minimal graphs constructed by Simon and Wickramasekera in \cite{SW07}, Rosales in \cite{Ros10}, and Krummel in \cite{Kru19}.

\subsection*{Acknowledgements}

This work was supported by the UK Engineering and Physical Sciences Research Council (EPSRC) Studentship EP/T517847/1 and by the Cambridge Trust, and was carried out while the author was a PhD student at the University of Cambridge under the supervision of Professor Neshan Wickramasekera.

The author is grateful to Paul Minter and Neshan Wickramasekera for many helpful conversations surrounding this work, as well as for their personal support.

\section{Preliminaries on $Q$-valued functions}
\label{sec:preliminaries}

\subsection{Notation and basic definitions}
\label{subsec:notation}

Throughout the paper, we use the following notation:
\begin{itemize}
    \item   We fix $\Omega\subset \mathbb{R}^n$ to be an open set, $n\geq 2$. For $y\in \R^n$, $d(y,\partial\Omega)$ denotes the distance from $y$ to $\partial\Omega$.

    \item We denote by \(G(n,m)\) the Grassmannian of unoriented \(n\)-dimensional linear subspaces of \(\mathbb R^m\). For \(T\in G(n,m)\), \(\pi_T\) denotes the orthogonal projection onto \(T\).
    
    \item \(E\) will always denote a finite-dimensional Euclidean space. If \(v,w\in E\), \(v\cdot w\) denotes their Euclidean inner product. If \(F\) is another finite-dimensional Euclidean space and \(A,B\in\operatorname{Hom}(E,F)\), we write \(A:B:=\operatorname{tr}(A^*B)\) for their Hilbert--Schmidt inner product. Thus \(|A|^2=A:A\).
    
    \item \(\mathcal A_Q(E)\) is the space of unordered \(Q\)-tuples of points in \(E\), counted with multiplicity, in the sense of Almgren \cite[Section~1.1]{Alm00} (see also \cite{DLS11} for a modern treatment). If \(T=\sum_{\ell=1}^Q\llbracket T_\ell\rrbracket\) and \(S=\sum_{\ell=1}^Q\llbracket S_\ell\rrbracket\) belong to \(\mathcal A_Q(E)\), we write \(\mathcal G(T,S):=\min_{\sigma\in\mathfrak S_Q}\left(\sum_{\ell=1}^Q|T_\ell-S_{\sigma(\ell)}|^2\right)^{1/2}\), where \(\mathfrak S_Q\) denotes the group of permutations of \(\{1,\ldots,Q\}\). We also write \(|T|^2:=\mathcal G\bigl(T,Q\llbracket0\rrbracket\bigr)^2=\sum_{\ell=1}^Q|T_\ell|^2\), and for a measurable map \(u:\Omega\to\mathcal A_Q(E)\), we set \(\|u\|_{L^2(\Omega)}:=\left(\int_\Omega|u(x)|^2 \,dx\right)^{1/2}\).
    
    \item For $T=\sum_{i=1}^Q\llbracket T_i\rrbracket\in\mathcal A_Q(E)$, we denote its average by $\eta\circ T:=\frac1Q\sum_{i=1}^Q T_i$.
    
    \item We equip the first-jet space \(\mathbb R^k\times\operatorname{Hom}(\mathbb R^n,\mathbb R^k)\) with the Euclidean norm $|(a,A)|^2:=|a|^2+|A|^2$.

    \item If \(T=\sum_{i=1}^Q\llbracket T_i\rrbracket\in\mathcal A_Q(E)\), \(b\in E\), and \(\lambda\in\mathbb R\), we write $T\oplus b:=\sum_{i=1}^Q\llbracket T_i+b\rrbracket$, $T\ominus b:=T\oplus(-b)$, $\lambda T:=\sum_{i=1}^Q\llbracket\lambda T_i\rrbracket$.
    We use the same notation pointwise for multiple-valued maps. Thus, if \(u:\Omega\to\mathcal A_Q(E)\) and \(\ell:\Omega\to E\) is single-valued, then $(u\ominus\ell)(x):=u(x)\ominus\ell(x)$.

    \item We call \(u:\Omega\to\mathcal A_Q(\mathbb R^k)\) affine if there are \(a_i\in\mathbb R^k\) and \(A_i\in\operatorname{Hom}(\mathbb R^n,\mathbb R^k)\) such that $u(x)=\sum_{i=1}^Q\llbracket a_i+A_ix\rrbracket$ for every \(x\in\Omega\).

    \item If \(u\in C^1(\Omega,\mathcal A_Q(\mathbb R^k))\), we denote by \(\mathbf v(u)\) its graph varifold. More precisely, if $j_1u(x)=\sum_{i=1}^Q\llbracket(u_i(x),A_i(x))\rrbracket$
    is any measurable selection of the first jet (see Definition \ref{def:C1-Q-valued} for the definition of first jet), then
    \[
    \mathbf v(u)(\varphi)
    :=
    \int_\Omega\sum_{i=1}^Q
    \varphi\bigl((x,u_i(x)),P_{A_i(x)}\bigr)
    \sqrt{\det\bigl(I+A_i(x)^TA_i(x)\bigr)}\,dx
    \]
    for every \(\varphi\in C_c\bigl(\Omega\times\mathbb R^k\times G(n,n+k)\bigr)\), where $P_A:=(\operatorname{Id},A)[\mathbb R^n]$.

    \item For \(u\in C^1(\Omega,\mathcal A_Q(\mathbb R^k))\), we call \(x\in\Omega\) a collision point of \(u\) if at least two sheets of \(u(x)\) have the same value. We denote by $C_u:=\left\{x\in\Omega:u(x)=Q\llbracket a\rrbracket\text{ for some }a\in\mathbb R^k\right\}$ its complete value-collision set, and by $K_u:=\left\{x\in\Omega:j_1u(x)=Q\llbracket(a,A)\rrbracket\text{ for some }(a,A)\right\}$ its complete first-jet collision set.

    \item We define
    \[
    \operatorname{reg}u
    :=
    \left\{
    x\in\Omega:
    \begin{array}{l}
    \text{there are a neighbourhood \(W\) of \(x\) and maps}\\
    u_1,\ldots,u_Q\in C^1(W,\mathbb R^k)
    \text{ such that }
    u|_W=\sum_{i=1}^Q\llbracket u_i\rrbracket
    \end{array}
    \right\},
    \]
    where repetitions among the \(u_i\) are allowed, and call $\mathcal B_u:=\Omega\setminus\operatorname{reg}u$ the branch set of \(u\).
\end{itemize}

\subsection{\(C^1\) multiple-valued maps and first jets}

\begin{definition}
\label{def:C1-Q-valued}
Let $u:\Omega\to \mathcal{A}_Q(\mathbb{R}^k)$ be a $Q$-valued map. We say that $u$ is differentiable at $x\in \Omega$ if there are linear maps $A_1,...,A_Q\in \operatorname{Hom}(\mathbb{R}^n, \mathbb{R}^k)$ such that
$$
\mathcal G(u(x+h),\sum_{\ell=1}^Q \llbracket a_\ell + A_\ell h\rrbracket )=o(|h|)\qquad \text{as $|h|\to 0$},
$$
where $u(x)=\sum_{\ell=1}^Q \llbracket a_\ell\rrbracket $. We call $Du(x):=\sum_{\ell=1}^Q \llbracket A_\ell\rrbracket $ the differential of $u$ at $x$ and $j_1u(x):=\sum_{\ell=1}^Q \llbracket (a_\ell, A_\ell)\rrbracket \in \mathcal{A}_Q(\R^k\times \operatorname{Hom}(\R^n,\R^k))$ the first-order jet of $u$ at $x$.

We say that \(u\in C^1(\Omega)\) if \(u\) is differentiable at every \(x\in\Omega\) and the \(Q\)-valued map \(x\mapsto j_1u(x)\) is continuous. We set \(\|u\|_{C^1(\Omega)} := \sup_{\Omega}|u| + \sup_{\Omega}|Du|\).

We say that $u\in C^{1,\alpha}(\Omega)$ if $u\in C^1(\Omega)$ and there is $C>0$ such that, for all $x,y\in \Omega$, $\mathcal G (Du(x), Du(y))\leq C|x-y|^{\alpha}$. We indicate the least such constant \(C\) by \([Du]_{\alpha;\Omega}\), and set \(\|u\|_{C^{1,\alpha}(\Omega)} := \sup_{\Omega}|u| + \sup_{\Omega}|Du| + [Du]_{\alpha;\Omega}\).

Finally, we say that \(u_j\to u\) locally in \(C^1(\Omega,\mathcal A_Q(\mathbb R^k))\) if $j_1u_j\longrightarrow j_1u$ locally uniformly in $\mathcal A_Q\bigl(\mathbb R^k\times\operatorname{Hom}(\mathbb R^n,\mathbb R^k)\bigr)$.
\end{definition}

\begin{remark}
Let \(\Omega\subset\mathbb R^n\) be a convex and bounded domain and let \(u\in C^1(\Omega,\mathcal A_Q(\mathbb R^k))\). It can be proven (using Lemma \ref{lem:one-dimensional-jet-selection} as in the proof of Lemma \ref{lem:C1-compactness} below) that \(u\in C^{1,\alpha}(\Omega)\) if and only if there is \(M<\infty\) such that $\mathcal G\bigl(u(y),j_1u(z)[y-z]\bigr)\le M|y-z|^{1+\alpha}$ for all $y,z\in\Omega$, where $j_1u(z)[h] :=\sum_{i=1}^Q\llbracket u_i(z)+Du_i(z)h\rrbracket$.
Moreover, one may take $M = \frac{C(Q)}{1+\alpha}[Du]_{\alpha,\Omega}$.
We will not use this characterization below; it is included only to motivate the preceding definition.
\end{remark}

\begin{lemma}[Elementary properties of first jets]
\label{lem:elementary-first-jet-calculus}
Let \(u\in C^1(\Omega,\mathcal A_Q(\mathbb R^k))\). Then:
\begin{enumerate}[label=\textup{(\roman*)}]
\item the first jet \(j_1u(x)\) is uniquely determined at every \(x\in\Omega\);
\item if \(\gamma:I\to\Omega\) is \(C^1\), then \(u\circ\gamma\in C^1(I)\) and, whenever $j_1u(\gamma(t))=\sum_{i=1}^Q\llbracket(a_i,A_i)\rrbracket$, one has $j_1(u\circ\gamma)(t) = \sum_{i=1}^Q\llbracket(a_i,A_i\dot\gamma(t))\rrbracket$;
\item \(u\) is locally Lipschitz, and its first jet agrees almost everywhere with its approximate first-order jet in the sense of \cite{DLS11}.
\end{enumerate}
\end{lemma}

\begin{proof}
For uniqueness, separate first the distinct values of \(u(x)\). Within a cluster having common value, suppose that \(\{A_i\}\) and \(\{B_i\}\) give two first-order approximations. Evaluating at \(h=te\) and letting \(t\to0\) gives $\sum_i\llbracket A_i e\rrbracket = \sum_i\llbracket B_i e\rrbracket$ for every $e\in\mathbb R^n$.
After grouping repeated matrices, choose \(e_0\) outside the kernels of all their nonzero pairwise differences. Near \(e_0\) the matching is fixed, and hence \(A_i e=B_{\pi(i)}e\) on an open set of \(e\)'s for some permutation \(\pi\). Thus \(A_i=B_{\pi(i)}\), proving \textup{(i)}.

Assertion \textup{(ii)} follows directly by inserting \(\gamma(t+s)-\gamma(t)=\dot\gamma(t)s+o(s)\) into the defining first-order expansion.

Since \(|Du|\) is locally bounded, the defining first-order expansion and a simple continuation argument along line segments show that \(u\) is locally Lipschitz. The almost-everywhere identification then follows from the Rademacher theorem \cite[Theorem~1.13]{DLS11} and uniqueness.
\end{proof}

\begin{lemma}[One-dimensional selection of first jets]
\label{lem:one-dimensional-jet-selection}
Let \(u\in C^1(\Omega,\mathcal A_Q(\mathbb R^k))\) and let \(\gamma:I\to\Omega\) be a \(C^1\) curve, where \(I\) is an interval. Then there are maps \(a_i\in C^1(I,\mathbb R^k)\) and \(A_i\in C^0(I,\operatorname{Hom}(\mathbb R^n,\mathbb R^k))\) such that $j_1u (\gamma(t)) = \sum_{i=1}^Q\llbracket (a_i(t),A_i(t)) \rrbracket$ and $\dot a_i(t)=A_i(t)\dot\gamma(t)$ for every \(t\in I\). Moreover, if \(j_1u \circ\gamma\) is given as a continuous sum of sub-jets, the selection can be chosen to respect this decomposition.
\end{lemma}

\begin{proof}
Since \(I\) is an interval, every continuous \(Q\)-valued map on \(I\) admits continuous selections; see \cite[Remark~1.3]{DLS11}. Applying this to \(j_1u\circ\gamma\), and separately to each prescribed continuous sub-jet, we obtain continuous maps \(a_\ell\) and \(A_\ell\) such that \(j_1u(\gamma(t))=\sum_{\ell=1}^Q\llbracket(a_\ell(t),A_\ell(t))\rrbracket\), with the selections respecting the prescribed decomposition.

Set \(v:=u\circ\gamma\). Since the regularity of the selected maps can be checked on compact subintervals, we may assume without loss of generality that \(I\) is compact. By Lemma~\ref{lem:elementary-first-jet-calculus}\textup{(iii)} and the compactness of \(I\), the map \(v\) is then Lipschitz. Fix \(s<t\) in \(I\). Since \(\mathcal G(v(\tau),v(s))\leq\operatorname{Lip}(v)(t-s)\) for every \(\tau\in[s,t]\), the path \(a_\ell([s,t])\) is contained in the union of the \(Q\) closed balls of radius \(\operatorname{Lip}(v)(t-s)\) centred at the values of \(v(s)\). By continuity, it remains in the connected component containing \(a_\ell(s)\). A chain of at most \(Q\) such balls therefore gives
\[
|a_\ell(t)-a_\ell(s)|\leq(2Q-1)\operatorname{Lip}(v)|t-s|.
\]
Thus each \(a_\ell\) is Lipschitz.

At almost every \(t\), all the \(a_\ell\) are differentiable, and uniqueness of the first jet, together with the chain rule, gives
\[
\sum_{\ell=1}^Q\llbracket(a_\ell(t),\dot a_\ell(t))\rrbracket=j_1v(t)=\sum_{\ell=1}^Q\llbracket(a_\ell(t),A_\ell(t)\dot\gamma(t))\rrbracket.
\]
By Sobolev locality, \(\dot a_\ell=\dot a_{\ell'}\) almost everywhere on \(\{a_\ell=a_{\ell'}\}\). Hence the preceding identity implies \(\dot a_\ell=A_\ell\dot\gamma\) almost everywhere on \(I\), for every \(\ell\). Since the right-hand side is continuous, \(a_\ell\in C^1(I,\mathbb R^k)\) and the identity holds everywhere.
\end{proof}

\begin{lemma}[\(C^1\)-compactness]
\label{lem:C1-compactness}
Let \(u_j\in C^{1,\alpha}(\Omega,\mathcal A_Q(\mathbb R^k))\). Suppose that, for every \(\Omega'\Subset\Omega\), $\sup_j\left(\sup_{\Omega'}|u_j| +\sup_{\Omega'}|Du_j| +[Du_j]_{\alpha,\Omega'} \right)<\infty$.
Then, after passing to a subsequence, there exists
\(u\in C^{1,\alpha}_{\mathrm{loc}}(\Omega,\mathcal A_Q(\mathbb R^k))\) such that $u_j\longrightarrow u$ locally in $C^1(\Omega)$.
\end{lemma}

\begin{proof}
Fix \(B_{2r}(x_0)\Subset\Omega\), and let \(M<\infty\) be such that $\sup_j\sup_{B_{2r}(x_0)}\bigl(|u_j|+|Du_j|\bigr)\le M$, $\sup_j[Du_j]_{\alpha,B_{2r}(x_0)}\le M$.
Fix \(x,y\in B_r(x_0)\), set \(h:=y-x\), and let \(\gamma(t):=x+th\). By Lemma \ref{lem:one-dimensional-jet-selection}, we may write $j_1u_j(\gamma(t)) = \sum_{i=1}^Q\llbracket(a_i(t),A_i(t))\rrbracket$, $\dot a_i(t)=A_i(t)h$.
For every \(t\in[0,1]\), $\mathcal G\bigl(Du_j(\gamma(t)),Du_j(x)\bigr)
\le Mt^\alpha|h|^\alpha$.
Since \(A_i\) is continuous, the usual chain-of-balls argument gives $|A_i(t)-A_i(0)| \le C(Q)Mt^\alpha|h|^\alpha$.
Indeed, \(A_i([0,t])\) is contained in the union of the \(Q\) balls of radius \(Mt^\alpha|h|^\alpha\) centred at the differential values occurring in \(Du_j(x)\), and the component containing \(A_i(0)\) consists of a chain of at most \(Q\) such balls.

Moreover, $|a_i(1)-a_i(0)| \le \int_0^1|A_i(t)|\,|h|\,dt \le M|h|$.
Using the matching supplied by these selections at \(t=0\) and \(t=1\), we obtain $\mathcal G\bigl(j_1u_j(x),j_1u_j(y)\bigr) \le C(Q)M\bigl(|x-y|+|x-y|^\alpha\bigr)$.
Thus \(j_1u_j\) is uniformly bounded and equicontinuous on \(B_r(x_0)\).

As each closed and bounded subset of $\mathcal A_Q\bigl(\mathbb R^k\times\operatorname{Hom}(\mathbb R^n,\mathbb R^k)\bigr)$ is compact, the metric Arzelà--Ascoli theorem and a diagonal argument give, after passing to a subsequence, $j_1u_j\longrightarrow J$ locally uniformly in $\Omega$ for some continuous \(Q\)-valued map \(J\). Let \(u\) be the value projection of \(J\).

It remains to show that \(J=j_1u\). Fix \(x\in\Omega\) and \(h\) sufficiently small. Applying the preceding argument along \(\gamma(t):=x+th\), we have $a_i(1)-a_i(0)-A_i(0)h = \int_0^1\bigl(A_i(t)-A_i(0)\bigr)h\,dt$, and therefore $\mathcal G\bigl(u_j(x+h),j_1u_j(x)[h]\bigr) \le C(Q)M|h|^{1+\alpha}$.
Passing to the limit as \(j\to\infty\) gives $\mathcal G\bigl(u(x+h),J(x)[h]\bigr) \le C(Q)M|h|^{1+\alpha} = o(|h|)$.
Thus \(u\) is differentiable at \(x\) and \(J(x)=j_1u(x)\). Since \(x\) was arbitrary and \(J\) is continuous, \(u\in C^1(\Omega,\mathcal A_Q(\mathbb R^k))\), and \(u_j\to u\) locally in \(C^1\).

Finally, passing to the limit in $\mathcal G\bigl(Du_j(x),Du_j(y)\bigr)\le M|x-y|^\alpha$ shows that \([Du]_{\alpha,B_r(x_0)}\le M\). Since \(B_{2r}(x_0)\Subset\Omega\) was arbitrary, \(u\in C^{1,\alpha}_{\mathrm{loc}}(\Omega,\mathcal A_Q(\mathbb R^k))\).
\end{proof}

\begin{proposition}[Local splitting under separation of the first jets]
\label{prop:local-splitting}
Let \(u\in C^1(\Omega,\mathcal A_Q(\mathbb R^k))\).
Assume that, at \(x_0\in\Omega\), $j_1u(x_0)=\sum_{\alpha=1}^NQ_\alpha\llbracket(a^\alpha,A^\alpha)\rrbracket$, where \((a^\alpha,A^\alpha)\neq(a^\beta,A^\beta)\) for \(\alpha\neq\beta\) and \(\sum_\alpha Q_\alpha=Q\).
Then there are a neighbourhood \(U\) of \(x_0\) and maps \(u^\alpha\in C^1(U,\mathcal A_{Q_\alpha}(\mathbb R^k))\) such that $u=\sum_{\alpha=1}^Nu^\alpha$, $u^\alpha(x_0)=Q_\alpha\llbracket a^\alpha\rrbracket$, $Du^\alpha(x_0)=Q_\alpha\llbracket A^\alpha\rrbracket$.
\end{proposition}

\begin{proof}
Choose \(\rho>0\) so that the closed balls \(\overline B_\rho((a^\alpha,A^\alpha))\) in the first-jet space are pairwise disjoint. By continuity of \(j_1u\), after shrinking to a neighbourhood \(U\) of \(x_0\), one has
\[
j_1u=\sum_{\alpha=1}^NJ^\alpha,
\qquad
J^\alpha(x)=\sum_{\ell=1}^{Q_\alpha}\llbracket(a^\alpha_\ell(x),A^\alpha_\ell(x))\rrbracket,
\]
where \(J^\alpha\) is the continuous sub-jet whose components lie in \(B_\rho((a^\alpha, A^\alpha))\).

Let \(u^\alpha\) and \(D^\alpha\) be the value and differential projections of \(J^\alpha\). Then \(u=\sum_\alpha u^\alpha\), and both \(u^\alpha\) and \(D^\alpha\) are continuous. We show that \(u^\alpha\) is differentiable with differential \(D^\alpha\).

Fix $\alpha$, \(x\in U\) and \(h\) sufficiently small, and set \(\gamma(t):=x+th\), \(0\leq t\leq1\). Applying Lemma \ref{lem:one-dimensional-jet-selection} to \(u\) along
\(\gamma\), we obtain
\[
j_1u(\gamma(t))
=
\sum_{i=1}^Q\llbracket(a_i(t),A_i(t))\rrbracket,
\qquad
\dot a_i(t)=A_i(t)h,
\]
where each \(t\mapsto (a_i(t),A_i(t))\) is continuous. For every \(t\in[0,1]\), each \((a_i(t),A_i(t))\) belongs to one of the pairwise disjoint balls \(B_\rho((a^\beta,A^\beta))\).
Since \([0,1]\) is connected and \((a_i,A_i)\) is continuous, for each \(i\) there is a unique \(\beta(i)\) such that $(a_i(t),A_i(t))\in B_\rho((a^{\beta(i)},A^{\beta(i)}))$ for every $t\in[0,1]$.
After relabelling, the indices for which \(\beta(i)=\alpha\) are
\(1,\ldots,Q_\alpha\). By the definition of \(J^\alpha\), it follows that
\[
J^\alpha(\gamma(t))
=
\sum_{\ell=1}^{Q_\alpha}
\llbracket(a_\ell(t),A_\ell(t))\rrbracket,
\qquad
\dot a_\ell(t)=A_\ell(t)h.
\]
By continuity of \(D^\alpha\), $\omega_\alpha(x,h):=\max_{\substack{0\leq t\leq1\\1\leq\ell\leq Q_\alpha}}|A_\ell(t)-A_\ell(0)|\longrightarrow0$ as $h\to0$.
As $a_\ell(1)-a_\ell(0)-A_\ell(0)h=\int_0^1(A_\ell(t)-A_\ell(0))h\,dt$, it follows that $\mathcal G\left(u^\alpha(x+h),\sum_{\ell=1}^{Q_\alpha}\llbracket a_\ell(0)+A_\ell(0)h\rrbracket\right)\le\sqrt{Q_\alpha}\,\omega_\alpha(x,h)|h|=o(|h|)$.
Thus \(u^\alpha\) is differentiable at \(x\), with \(Du^\alpha(x)=D^\alpha(x)\). Since \(x\) was arbitrary and \(J^\alpha\) is continuous, $u^\alpha\in C^1(U,\mathcal A_{Q_\alpha}(\mathbb R^k))$, $j_1u^\alpha=J^\alpha$.
The assertions at \(x_0\) follow directly from the construction.
\end{proof}

\begin{corollary}[Quantitative persistence of separated first-jet clusters]
\label{cor:persistence-jet-clusters}
Let \(0<\alpha<1\), \(R<\infty\), and \(u\in C^{1,\alpha}\bigl(B_R,\mathcal A_Q(\mathbb R^k)\bigr)\). Write \(j_1u(0)=\sum_{\ell=1}^Q\llbracket(a_\ell,A_\ell)\rrbracket\), and let \(\{1,\ldots,Q\}=I_1\cup\cdots\cup I_N\), \(N\geq2\), be a partition. Set
\[
d:=\min_{\substack{\gamma\neq\gamma'\\ \ell\in I_\gamma,\;\ell'\in I_{\gamma'}}}|(a_\ell,A_\ell)-(a_{\ell'},A_{\ell'})|.
\]
For every \(M<\infty\), there exists \(d_0=d_0(Q,\alpha,M,R)>0\) such that, if \([Du]_{\alpha;B_R}\leq M\) and \(d\geq d_0\), then there are uniquely determined maps \(u^\gamma\in C^{1,\alpha}\bigl(B_R,\mathcal A_{|I_\gamma|}(\mathbb R^k)\bigr)\), \(\gamma=1,\ldots,N\), such that \(u=\sum_{\gamma=1}^N u^\gamma\), \(j_1u^\gamma(0)=\sum_{\ell\in I_\gamma}\llbracket(a_\ell,A_\ell)\rrbracket\), and \([Du^\gamma]_{\alpha;B_R}\leq C_Q[Du]_{\alpha;B_R}\).
\end{corollary}

\begin{proof}
For \(x\in B_R\), let \(K(x)\) be the image of \(j_1u(x)\) under the map
\[
(a,A)\longmapsto(a-Ax,A),
\]
which takes the first jet at \(x\) of an affine map to its first jet at the origin. Then \(K\) is continuous and \(K(0)=j_1u(0)\).

Fix \(x\in B_R\). By Lemma~\ref{lem:one-dimensional-jet-selection}, we may write
\[
j_1u(tx)=\sum_{\ell=1}^Q\llbracket(a_\ell(t),A_\ell(t))\rrbracket,\qquad \dot a_\ell(t)=A_\ell(t)x,
\]
with \((a_\ell(0),A_\ell(0))=(a_\ell,A_\ell)\). The Hölder bound for \(Du\) and the usual chain-of-balls argument, now based at \(t=1\), give
\(|A_\ell(t)-A_\ell(1)|\leq C_Q[Du]_{\alpha;B_R}(1-t)^\alpha|x|^\alpha\). Hence
\[
|a_\ell(1)-A_\ell(1)x-a_\ell|
=
\left|\int_0^1(A_\ell(t)-A_\ell(1))x\,dt\right|
\leq
C_Q[Du]_{\alpha;B_R}|x|^{1+\alpha}.
\]
Together with the preceding estimate at \(t=0\), this yields
\[
\mathcal G(K(x),j_1u(0))\leq C_Q[Du]_{\alpha;B_R}|x|^\alpha(1+|x|).
\]

Choose \(d_0>0\) such that \(C_QMR^\alpha(1+R)<d_0/3\). Since \([Du]_{\alpha;B_R}\leq M\) and \(d\geq d_0\), the preceding estimate shows that every atom of \(K(x)\) belongs to the union of the fixed, pairwise disjoint sets
\[
\mathcal N_\gamma:=\bigcup_{\ell\in I_\gamma}B_{d/3}((a_\ell,A_\ell)),\qquad \gamma=1,\ldots,N.
\]
By continuity, each \(\mathcal N_\gamma\) contains precisely \(|I_\gamma|\) atoms of \(K(x)\), counted with multiplicity, for every \(x\in B_R\). Applying the inverse transformation \((b,B)\mapsto(b+Bx,B)\) to each group gives a continuous decomposition \(j_1u=J^1+\cdots+J^N\). Since this transformation is injective, the groups remain pairwise disjoint in the first-jet space.

Applying Proposition~\ref{prop:local-splitting} locally to the distinct first-jet atoms and grouping the resulting components according to \(J^\gamma\) shows that their value projections \(u^\gamma\) belong to \(C^1(B_R)\) and satisfy \(j_1u^\gamma=J^\gamma\). The decomposition is unique, since along any segment from the origin the transported jet sheets cannot pass between the disjoint sets \(\mathcal N_\gamma\).

Finally, apply Lemma~\ref{lem:one-dimensional-jet-selection} along an arbitrary segment joining \(x,y\in B_R\), with selections respecting the sub-jet decomposition. The same chain-of-balls argument gives \(\mathcal G(Du^\gamma(x),Du^\gamma(y))\leq C_Q[Du]_{\alpha;B_R}|x-y|^\alpha\), and hence \([Du^\gamma]_{\alpha;B_R}\leq C_Q[Du]_{\alpha;B_R}\), completing the proof.
\end{proof}

\subsection{Sobolev spaces for multifunctions}

We first introduce the Sobolev framework in which Dirichlet stationarity will be formulated. We use the notions of approximate differentiability and approximate differential for multiple-valued maps introduced by Almgren \cite[Section~1.4]{Alm00}; see \cite[Section~2.2.1]{DLS11} for a modern exposition.

\begin{definition}[Sobolev spaces for multifunctions]
\label{def:sobolev-spaces}
Let \(\xi:\mathcal A_Q(\mathbb R^k)\to\mathbb R^N\) be an Almgren embedding, normalized so that \(\xi(Q\llbracket0\rrbracket)=0\). We say that \(u\in W^{1,2}\bigl(\Omega,\mathcal A_Q(\mathbb R^k)\bigr)\) if \(\xi\circ u\in W^{1,2}(\Omega,\mathbb R^N)\). Writing \(j_1u=\sum_{\ell=1}^Q\llbracket(a_\ell,A_\ell)\rrbracket\) for the approximate first-order jet of \(u\), we define its approximate differential by \(Du:=\sum_{\ell=1}^Q\llbracket A_\ell\rrbracket\) and set
\[
\|u\|_{W^{1,2}(\Omega)}^2:=\|j_1u\|_{L^2(\Omega)}^2=\int_\Omega\left\{|u|^2+|Du|^2\right\}.
\]
We say that \(u_j\to u\) strongly in \(W^{1,2}(\Omega,\mathcal A_Q(\mathbb R^k))\) if \(\int_\Omega\mathcal G(j_1u_j,j_1u)^2\longrightarrow0\).

We say that \(u\in W^{2,2}\bigl(\Omega,\mathcal A_Q(\mathbb R^k)\bigr)\) if \(u\in W^{1,2}\) and \(j_1u\in W^{1,2}\). Writing the approximate first-order jet of \(j_1u\) as \(j_1(j_1u)=\sum_{\ell=1}^Q\llbracket((a_\ell,A_\ell),(B_\ell,S_\ell))\rrbracket\), we define the approximate second jet and the second differential of \(u\) by
\[
j_2u:=\sum_{\ell=1}^Q\llbracket(a_\ell,A_\ell,S_\ell)\rrbracket,
\qquad
D^2u:=\sum_{\ell=1}^Q\llbracket S_\ell\rrbracket.
\]
We set
\[
\|u\|_{W^{2,2}(\Omega)}^2:=\|j_2u\|_{L^2(\Omega)}^2=\int_\Omega\left\{|u|^2+|Du|^2+|D^2u|^2\right\}.
\]
We say that \(u_j\to u\) strongly in \(W^{2,2}(\Omega,\mathcal A_Q(\mathbb R^k))\) if \(\int_\Omega\mathcal G(j_2u_j,j_2u)^2\longrightarrow0\).

The local spaces \(W^{1,2}_{\mathrm{loc}}\) and \(W^{2,2}_{\mathrm{loc}}\) and the corresponding local notions of strong convergence are defined analogously.
\end{definition}

\begin{remark}
\label{rem:sobolev-equivalences}
With the notation of Definition~\ref{def:sobolev-spaces}, the following hold.
\begin{enumerate}[label=\textup{(\roman*)}]
\item The spaces \(W^{1,2}\) and \(W^{2,2}\) and the corresponding notions of strong convergence are independent of the choice of Almgren embeddings.

\item One has \(A_\ell=B_\ell\) almost everywhere for every \(\ell\). Indeed, the chain rule applied to the projection \((a,A)\mapsto a\) gives \(\sum_{\ell=1}^Q\llbracket(a_\ell,B_\ell)\rrbracket=j_1u=\sum_{\ell=1}^Q\llbracket(a_\ell,A_\ell)\rrbracket\), and the conclusion follows because \(A_\ell=A_m\) almost everywhere on \(\{a_\ell=a_m\}\).

Hence \(|D(j_1u)|^2=|Du|^2+|D^2u|^2\) a.e. in \(\Omega\). Consequently,
\[
\int_\Omega\left\{|j_1u|^2+|D(j_1u)|^2\right\}=\|u\|_{W^{2,2}(\Omega)}^2+\|Du\|_{L^2(\Omega)}^2\simeq\|u\|_{W^{2,2}(\Omega)}^2.
\]
Moreover, if \(\Xi\) is an Almgren embedding of the first-jet space, normalized so that \(\Xi(Q\llbracket(0,0)\rrbracket)=0\), then \(\|u\|_{W^{2,2}(\Omega)}\simeq\|\Xi\circ j_1u\|_{W^{1,2}(\Omega)}\).

\item It is possible to show that strong \(W^{1,2}\)-convergence is equivalently characterized either by \(u_j\to u\) in \(L^2(\Omega,\mathcal A_Q(\mathbb R^k))\) and \(\int_\Omega|Du_j|^2\to\int_\Omega|Du|^2\) or by the strong \(W^{1,2}\)-convergence of \(\xi\circ u_j\) to \(\xi\circ u\) for any Almgren embedding \(\xi\). However, we will not need these characterisations in the rest of the paper.
\end{enumerate}
\end{remark}

\begin{lemma}[Rellich compactness in \(W^{2,2}\)]
\label{lem:rellich-W22}
Let \(\Omega\subset\mathbb R^n\) be a bounded Lipschitz domain and let \(u_j\in W^{2,2}(\Omega,\mathcal A_Q(\mathbb R^k))\) satisfy $\sup_j\|u_j\|_{W^{2,2}(\Omega)}<\infty$. Then, after passing to a subsequence, there exists \(u\in W^{2,2}(\Omega,\mathcal A_Q(\mathbb R^k))\) such that $u_j\longrightarrow u$ strongly in $W^{1,2}(\Omega,\mathcal A_Q(\mathbb R^k))$.
\end{lemma}

\begin{proof}
Set \(E:=\mathbb R^k\times\operatorname{Hom}(\mathbb R^n,\mathbb R^k)\) and \(J_j:=j_1u_j\). Fix an Almgren embedding \(\Xi:\mathcal A_Q(E)\to\mathbb R^N\), normalized at the origin. By Remark~\ref{rem:sobolev-equivalences}\textup{(ii)}, the maps \(\Xi\circ J_j\) are uniformly bounded in \(W^{1,2}(\Omega,\mathbb R^N)\). Classical Rellich compactness gives, after passing to a subsequence, \(\Xi\circ J_j\rightharpoonup F\) weakly in \(W^{1,2}\) and \(\Xi\circ J_j\to F\) strongly in \(L^2\) and almost everywhere. Since the image of \(\Xi\) is closed and its inverse is Lipschitz, \(J:=\Xi^{-1}\circ F\) belongs to \(W^{1,2}(\Omega,\mathcal A_Q(E))\) and \(J_j\to J\) strongly in \(L^2\). Let \(\pi:E\to\mathbb R^k\) be the first projection and set \(u:=\pi_\#J\in W^{1,2}(\Omega,\mathcal A_Q(\mathbb R^k))\). It remains to identify \(J\) with \(j_1u\).

Since this is a local assertion, we may work in an arbitrary ball \(B\Subset\Omega\). For \(i=1,\ldots,n\), define \(P_i:E\to\mathbb R^k\times\mathbb R^k\) by \(P_i(a,A):=(a,Ae_i)\), where \(e_1,\ldots,e_n\) denote the standard basis of \(\mathbb R^n\). Fix \(i\), and set \(I_y:=\{t:y+te_i\in B\}\) for \(y\in e_i^\perp\). By Fubini's theorem, after passing to a subsequence, \(J_j(y+\cdot\,e_i)\to J(y+\cdot\,e_i)\) strongly in \(L^2(I_y)\) for almost every \(y\) with \(I_y\neq\varnothing\). The Sobolev bound and Fatou's lemma then give, for almost every such \(y\), a further subsequence (depending on \(y\)) along which these restrictions are uniformly bounded in \(W^{1,2}(I_y)\).

On each such line, the chain rule gives \((P_i)_\#J_j(y+\cdot\,e_i)=j_1[u_j(y+\cdot\,e_i)]\). We obtain two descriptions of the limit of this identity. The left-hand side converges strongly in \(L^2(I_y)\) to \((P_i)_\#J(y+\cdot\,e_i)\), since \((P_i)_\#\) is Lipschitz.

For the right-hand side, the one-dimensional Sobolev selection theorem gives uniformly \(W^{1,2}\)-bounded selections \((a_{j,\ell},b_{j,\ell})\), \(\ell=1,\ldots,Q\). Since these select the first jet of \(u_j(y+\cdot\,e_i)\), Sobolev locality on the sets where their first components coincide gives \(a_{j,\ell}'=b_{j,\ell}\) almost everywhere. After passing to a further subsequence, the selections converge weakly in \(W^{1,2}\) and strongly in \(L^2\) to \((a_\ell,b_\ell)\), with \(a_\ell'=b_\ell\). Their first components select \(u(y+\cdot\,e_i)\), so the right-hand side converges to \(\sum_{\ell=1}^Q\llbracket(a_\ell,a_\ell')\rrbracket=j_1[u(y+\cdot\,e_i)]\). Since we have already proven that \(u\in W^{1,2}\), by the chain rule for \(u\) this equals \((P_i)_\#j_1u(y+\cdot\,e_i)\).

So the complete chain of equalities becomes, writing \(J(y+te_i)=\sum_{\ell=1}^Q\llbracket(a_\ell(t),B_\ell(t))\rrbracket\) and \(j_1u(y+te_i)=\sum_{\ell=1}^Q\llbracket(a_\ell(t),A_\ell(t))\rrbracket\),
\[
\begin{aligned}
\sum_{\ell=1}^Q\llbracket(a_\ell,B_\ell e_i)\rrbracket
&=(P_i)_\#J(y+\cdot\,e_i)\\
&=\lim_j(P_i)_\#J_j(y+\cdot\,e_i)\\
&=\lim_j\sum_{\ell=1}^Q\llbracket(a_{j,\ell},a_{j,\ell}')\rrbracket\\
&=\sum_{\ell=1}^Q\llbracket(a_\ell,a_\ell')\rrbracket\\
&=(P_i)_\#j_1u(y+\cdot\,e_i)
=\sum_{\ell=1}^Q\llbracket(a_\ell,A_\ell e_i)\rrbracket.
\end{aligned}
\]

Since the matrices \(A_\ell\) agree almost everywhere whenever their values \(a_\ell\) coincide, the preceding identities imply \(B_\ell e_i=A_\ell e_i\) for every \(i\) and \(\ell\), under any common labelling of the value components, almost everywhere. Thus \(B_\ell=A_\ell\), and \(J=j_1u\).

Consequently, \(u\in W^{2,2}(\Omega,\mathcal A_Q(\mathbb R^k))\), and \(\int_\Omega\mathcal G(j_1u_j,j_1u)^2=\int_\Omega\mathcal G(J_j,J)^2\to0\), which is precisely the required strong \(W^{1,2}\)-convergence.
\end{proof}

\section{The Dirichlet-stationary problem}

\subsection{Stationarity and frequency}

\begin{definition}[Dirichlet stationarity]
Let \(v\in W_{\mathrm{loc}}^{1,2}(\Omega,\mathcal A_Q(\mathbb R^k))\).
Given \(\psi\in C_c^1(\Omega\times\mathbb R^k,\mathbb R^k)\), the outer variation generated by \(\psi\) is $v_t(x):=\sum_{\ell=1}^Q\llbracket v_\ell(x)+t\psi(x,v_\ell(x))\rrbracket$, where \(v=\sum_{\ell=1}^Q\llbracket v_\ell\rrbracket\).
Given \(\zeta\in C_c^1(\Omega,\mathbb R^n)\), set \(\Phi_t(x):=x+t\zeta(x)\); the inner variation generated by \(\zeta\) is $v_t:=v\circ\Phi_t$.

We say that \(v\) is Dirichlet-stationary if its Dirichlet energy has zero first variation at \(t=0\) with respect to every outer and inner variation. Equivalently, $\int_\Omega\sum_{\ell=1}^Q Dv_\ell:D\bigl(\psi(x,v_\ell)\bigr)=0$ for every \(\psi\in C_c^1(\Omega\times\mathbb R^k,\mathbb R^k)\), and $\int_\Omega\sum_{\ell=1}^Q\left\{\frac12|Dv_\ell|^2\operatorname{div}\zeta-Dv_\ell:\bigl(Dv_\ell D\zeta\bigr)\right\}=0$ for every \(\zeta\in C_c^1(\Omega,\mathbb R^n)\).
\end{definition}

\begin{lemma}[Outer stationarity of differential blocks]
\label{lem:stationarity-differential-blocks}
Let \(U\subset\mathbb R^n\) be a neighbourhood of \(x_0\), and let $v=\sum_{\alpha=1}^N v^\alpha$, $v^\alpha\in C^1\bigl(U,\mathcal A_{Q_\alpha}(\mathbb R^k)\bigr)$, $\sum_{\alpha=1}^NQ_\alpha=Q$, be a \(C^1\) decomposition. Suppose that, for some \(a\in\mathbb R^k\) and pairwise distinct $A^1,\ldots,A^N\in\operatorname{Hom}(\mathbb R^n,\mathbb R^k)$, one has $j_1v^\alpha(x_0) = Q_\alpha\llbracket(a,A^\alpha)\rrbracket$ for every \(\alpha=1,\ldots,N\).

If \(v\) is stationary with respect to outer variations for the Dirichlet energy, then, after possibly shrinking \(U\) around \(x_0\), every \(v^\alpha\) is stationary with respect to outer variations.
\end{lemma}

\begin{proof}
This is a special case of the localisation argument proved in greater detail and generality in Lemma~\ref{lem:weak-solution-differential-blocks}. We briefly indicate the argument here as well.

We argue by induction on \(Q\). The conclusion is immediate if \(Q=1\) or \(N=1\). Suppose that \(N\geq2\) and that the result has been proved at every multiplicity strictly smaller than \(Q\). For each \(\alpha\), set \(m^\alpha:=\eta\circ v^\alpha\). Choose \(\varrho>0\) so that the closed balls \(\overline B_\varrho(A^\alpha)\subset\operatorname{Hom}(\mathbb R^n,\mathbb R^k)\) are pairwise disjoint. Since \(Dv^\alpha(x_0)=Q_\alpha\llbracket A^\alpha\rrbracket\) and \(Dv^\alpha\) is continuous, after possibly shrinking \(U\) we may assume that \(\operatorname{spt}Dv^\alpha(x)\subset B_\varrho(A^\alpha)\) for every \(x\in U\) and every \(\alpha\). In particular, \(Dm^\alpha(x)\in B_\varrho(A^\alpha)\) for every \(x\in U\), and hence \(Dm^\alpha(x)\neq Dm^\beta(x)\) whenever \(\alpha\neq\beta\).
Choosing two blocks, there is \(e\in\mathbb R^k\) such that \(g:=e\cdot(m^\alpha-m^\beta)\) satisfies \(g(x_0)=0\) and \(Dg(x_0)\neq0\). After shrinking \(U\) once more, we may assume that \(Dg\neq0\) throughout \(U\). Thus \(H:=\{g=0\}\) is a \(C^1\) hypersurface containing the complete value-collision set of \(v\).

On \(U\setminus H\), the values of \(v\) split into clusters of multiplicity strictly smaller than \(Q\), and outer stationarity localizes to each value cluster. By the preceding separation of the differential sheets, every differential sub-block of a value cluster belongs to a unique prescribed component \(v^\gamma\); more precisely, \(v^\gamma\) is locally the sum of those differential sub-blocks whose differential sheets lie in \(B_\varrho(A^\gamma)\). The inductive hypothesis applied to the differential sub-blocks of each value cluster, together with the additivity of the outer-variation identity over the sheets, therefore shows that every prescribed block \(v^\gamma\) is outer-stationary in \(U\setminus H\).

It remains to extend the equation across \(H\). Applying the outer-variation identity for \(v^\gamma\) on the two sides of \(H\) with cutoffs vanishing in an \(\varepsilon\)-neighbourhood of \(H\), the boundary contributions converge to the corresponding normal fluxes on \(H\). These cancel because the integrand is continuous and symmetric in the first-jet sheets and the two outward unit normals are opposite. Letting \(\varepsilon\downarrow0\) gives the outer-variation identity for \(v^\gamma\) across \(H\), and hence throughout \(U\).
\end{proof}

\begin{definition}[Almgren's frequency function]
\label{def:frequency-function}
Suppose $v\in W^{1,2}_{\mathrm{loc}}(\Omega,\mathcal A_Q(\mathbb{R}^k))$. Let \(y\in \Omega\) and define $D_{v,y}(r):=r^{-n+2}\int_{B_r(y)}|Dv|^2$ and $H_{v,y}(r):=r^{-n+1} \int_{\partial B_r(y)}|v|^2$.
Whenever \(H_{v,y}(r)>0\), the frequency function is defined as $N_{v,y}(r):=\frac{D_{v,y}(r)}{H_{v,y}(r)}$. Whenever the limit exists, we call $N_v(y):=\lim_{r\downarrow0}N_{v,y}(r)$ the frequency of \(v\) at \(y\). We omit the subscript \(v\) when it is clear from the context.
\end{definition}

\begin{theorem}[Frequency monotonicity]
\label{prop:frequency-monotonicity-final}
Suppose $v\in W^{1,2}_{\mathrm{loc}}(\Omega,\mathcal A_Q(\mathbb R^k))$ is Dirichlet-stationary. Then, for every $y\in\Omega$, \(D_y\) and \(H_y\) are nondecreasing, and \(N_y\) is nondecreasing on every interval contained in \((0,d(y,\partial\Omega))\) on which \(H_y>0\). Moreover, for a.e. such \(r\),
\[
\frac{d}{dr}\log\left(H_y(r)\right)=\frac{2N_y(r)}{r}.
\]

Moreover:
\begin{enumerate}[label=\textup{(\roman*)}]

\item if \(v\) is not identically zero in the connected component of \(\Omega\) containing \(y\), then \(H_y(r)>0\) for every \(0<r<d(y,\partial\Omega)\). Moreover, for every \(0<r\leq R<d(y,\partial\Omega)\), \(\left(\frac Rr\right)^{2N_y(r)}H_y(r)\leq H_y(R)\leq\left(\frac Rr\right)^{2N_y(R)}H_y(r)\).

\item if \(v\) is not identically zero near \(y\) and $\mathcal G \bigl(v(x),Q\llbracket0\rrbracket\bigr) \le C|x-y|^\gamma$ near \(y\), then $N_y(0):=\lim_{r\downarrow0}N_y(r)\ge\gamma$;

\item if \(N_y\equiv\sigma\) on an interval, then \(v\) is homogeneous of
degree \(\sigma\) with centre \(y\) on the corresponding annulus;

\item if \(\Omega=\mathbb R^n\) and \(v\) is nonzero and homogeneous of degree \(\sigma\) with centre \(0\), then, for every \(y\in\mathbb R^n\), \(\lim_{r\to\infty}N_y(r)=\sigma\);

\item if \(\Omega=\mathbb R^n\), \(p\neq0\), \(v\) is homogeneous of degree \(\sigma\) with centre \(0\), and \(N_p\equiv\sigma\), then \(v\) has a representative satisfying \(v(x+tp)=v(x)\) for every \(x\in\mathbb R^n\) and \(t\in\mathbb R\).

\end{enumerate}
\end{theorem}

\begin{proof}
The proof is due to Almgren and can be reconstructed from his arguments in \cite[Chapter~2]{Alm00}. In particular, the monotonicity formula is proved in \cite[Theorem~2.6]{Alm00}. See \cite[Section~3.4]{DLS11} for a modern exposition. The setting there is that of Dirichlet minimizing \(Q\)-valued functions, but the proof only uses Dirichlet stationarity.
\end{proof}

\subsection{\(W^{2,2}\) regularity and compactness}

We briefly recall the following classical extension lemma.

\begin{lemma}[Zero extension]
\label{lem:zero-extension}
Let \(K\subset\Omega\) be relatively closed and let $f\in C^0(\Omega,\mathbb R^N)\cap W^{1,2}_{\mathrm{loc}}(\Omega\setminus K,\mathbb R^N)$.
Suppose that \(f=0\) on \(K\) and that $\int_{\Omega'\setminus K}|Df|^2<\infty$ for every \(\Omega'\Subset\Omega\). Then \(f\in W^{1,2}_{\mathrm{loc}}(\Omega,\mathbb R^N)\), with $Df=\mathbf 1_{\Omega\setminus K}Df$ a.e. in $\Omega$.
In particular, $Df=0$ a.e. on $K$.
\end{lemma}

\begin{proof}
Without loss of generality assume $N=1$. Let $\gamma_\delta:\mathbb{R}\to \mathbb{R}$ be odd, smooth, non-decreasing and convex on $(0,\infty)$, with
\[
\gamma_\delta=0\quad\text{on }[0,\delta],\qquad 0\le\gamma_\delta'\le1,\qquad \gamma_\delta(t)\to t,\qquad \gamma_\delta'(t)\to1
\]
for every $t>0$. Then $\gamma_\delta(f)\in W^{1,2}_{\mathrm{loc}}(\Omega)$ (as can be checked locally around each point). So, if $\varphi\in C^1_c(\Omega)$ is an arbitrary test function, $\int \gamma_\delta(f) D\varphi = -\int D(\gamma_\delta(f)) \varphi$. Expanding the RHS and letting \(\delta\downarrow0\), the dominated convergence theorem yields \(\int_\Omega f\,D\varphi=-\int_{\Omega\setminus\{f=0\}}Df\,\varphi\). Since \(Df=0\) almost everywhere on \(\{f=0\}\setminus K\) by the level-set property for Sobolev functions, the last integral equals \(-\int_{\Omega\setminus K}Df\,\varphi\). This proves the claimed extension formula.
\end{proof}

\begin{proposition}[Inductive Bochner inequality]
\label{prop:inductive-bochner}
Let \(v\in C^1\bigl(\Omega,\mathcal A_q(\mathbb R^k)\bigr)\) be stationary with respect to outer variations. Then \(v\in W^{2,2}_{\mathrm{loc}}(\Omega)\) and $\Delta|Dv|^2\ge 2|D^2v|^2$ in the sense of distributions. Moreover, for every \(B_R\Subset\Omega\),
\[
\sup_{B_{R/2}}|Dv|^2+R^{-n}\int_{B_{R/2}}|Dv|^2+R^{2-n}\int_{B_{R/2}}|D^2v|^2\le CR^{-n-2}\int_{B_R}|v|^2,
\]
where \(C=C(n)\).
\end{proposition}

\begin{proof}
We argue by induction on $q$. For $q=1$, the map is harmonic and the conclusion follows from the classical identity $\Delta |Dv|^2=2|D^2v|^2$.

Assume the result proved for every multiplicity smaller than $q$.
Write $v_a:=\eta\circ v=\frac1q\sum_{i=1}^q v_i$, $v_f:=\sum_{i=1}^q\llbracket v_i-v_a\rrbracket$.
The average $v_a$ is harmonic, while the average-free part $v_f$ is stationary with respect to outer variations and satisfies $\eta\circ v_f=0$.
Moreover, $|v|^2=|v_f|^2+q|v_a|^2$, $|Dv|^2=|Dv_f|^2+q|Dv_a|^2$, and, once the second derivatives have been obtained, $|D^2v|^2=|D^2v_f|^2+q|D^2v_a|^2$.
It is therefore enough to prove the result for $v_f$. Thus, replacing $v$ by $v_f$, we may assume that $v_a=0$.

Let $K_v:=\{x\in\Omega:v(x)=q\llbracket0\rrbracket,\ Dv(x)=q\llbracket0\rrbracket\}$.
Let \(x\in\Omega\setminus K_v\). If the values of \(v(x)\) are not all equal, then, after shrinking to a neighbourhood \(U\) of \(x\), the map decomposes into value-separated clusters $v=\sum_{\beta=1}^M v^\beta$, each of multiplicity strictly smaller than \(q\). Every \(v^\beta\) is outer-stationary: indeed, one may localize any outer variation to a neighbourhood of the graph of \(v^\beta\) which is disjoint from the graphs of the other clusters. Hence the inductive hypothesis applies to each \(v^\beta\).
If instead \(v(x)=q\llbracket0\rrbracket\), then the differential sheets are not all equal, since otherwise the zero-average condition would give \(Dv(x)=q\llbracket0\rrbracket\) and hence \(x\in K_v\). Proposition~\ref{prop:local-splitting} therefore gives a local decomposition $v=\sum_{\alpha=1}^N v^\alpha$ into differential blocks of multiplicity strictly smaller than \(q\). By Lemma~\ref{lem:stationarity-differential-blocks}, every \(v^\alpha\) is outer-stationary, and the inductive hypothesis applies again.
Thus, in either case, \(v\) decomposes locally on \(\Omega\setminus K_v\) into outer-stationary maps of strictly smaller multiplicity. Summing the conclusions for the components gives
\[
v\in W^{2,2}_{\mathrm{loc}}(\Omega\setminus K_v),
\qquad
\Delta|Dv|^2\ge2|D^2v|^2
\quad\text{in }\Omega\setminus K_v.
\]

Let $\gamma_\delta$ be as in the proof of Lemma \ref{lem:zero-extension}. Since $|Dv|^2$ is continuous and vanishes on $K_v$, the support of $\gamma_\delta(|Dv|^2)$ is locally contained in $\Omega\setminus K_v$. Thus, for every $\zeta\in C_c^\infty(\Omega)$,
\begin{equation}
\label{eq:truncated-inequality}
2\int \zeta^2\gamma_\delta'(|Dv|^2)|D^2v|^2\le \int\gamma_\delta(|Dv|^2)\Delta(\zeta^2)\le C\int |Dv|^2 \bigl(|D\zeta|^2+|\zeta||D^2\zeta|\bigr).
\end{equation}
Letting $\delta\downarrow0$ gives
\[
2\int_{\{|Dv|^2>0\}}\zeta^2|D^2v|^2\le C\int |Dv|^2 \bigl(|D\zeta|^2+|\zeta||D^2\zeta|\bigr).
\]

Let $\xi_D$ be an Almgren embedding of $\mathcal A_q\bigl(\operatorname{Hom}(\mathbb R^n,\mathbb R^k)\bigr)$.
Since $v\in W^{2,2}_{\mathrm{loc}}(\Omega\setminus K_v)$, the differential $Dv:\Omega\setminus K_v\to\mathcal A_q\bigl(\operatorname{Hom}(\mathbb R^n,\mathbb R^k)\bigr)$ belongs to $W^{1,2}_{\mathrm{loc}}(\Omega\setminus K_v)$.
The standard level-set property applied to $\xi_D\circ Dv$ gives $D(\xi_D\circ Dv)=0$ a.e. on $\{|Dv|^2=0\}\setminus K_v$.
Since $\xi_D^{-1}$ is Lipschitz on its image, it follows that $D^2v=0$ a.e. on $\{|Dv|^2=0\}\setminus K_v$.
Consequently,
\[
\int_{\Omega\setminus K_v}\zeta^2|D^2v|^2\le C\int |Dv|^2 \bigl(|D\zeta|^2+|\zeta||D^2\zeta|\bigr).
\]

Let \(\Xi\) be an Almgren embedding of \(\mathcal A_q\bigl(\mathbb R^k\times\operatorname{Hom}(\mathbb R^n,\mathbb R^k)\bigr)\). Since \(v\in W^{2,2}_{\mathrm{loc}}(\Omega\setminus K_v)\), one has \(\Xi\circ j_1v\in W^{1,2}_{\mathrm{loc}}(\Omega\setminus K_v)\), with $|D(\Xi\circ j_1v)|^2\leq C\bigl(|Dv|^2+|D^2v|^2\bigr)$ a.e. on $\Omega\setminus K_v$.
The preceding estimate therefore shows that \(\Xi\circ j_1v\) has locally finite Dirichlet energy on \(\Omega\setminus K_v\). Moreover, since \(v\in C^1\), $\Xi\circ j_1v-\Xi\bigl(q\llbracket(0,0)\rrbracket\bigr)$ is continuous on \(\Omega\) and vanishes on \(K_v\). Lemma \ref{lem:zero-extension} gives $\Xi\circ j_1v\in W^{1,2}_{\mathrm{loc}}(\Omega)$, $ D(\Xi\circ j_1v)=0$ a.e. on $K_v$.
Hence \(v\in W^{2,2}_{\mathrm{loc}}(\Omega)\). Since \(\Xi^{-1}\) is Lipschitz on its image, the last identity also implies $D^2v=0$ a.e. on $K_v$.

Finally, the first inequality in \eqref{eq:truncated-inequality}, with $\varphi$ in place of $\zeta^2$, for any nonnegative $\varphi\in C_c^\infty(\Omega)$, gives
\[
2\int\varphi\,\gamma_\delta'(|Dv|^2)|D^2v|^2\le \int\gamma_\delta(|Dv|^2)\Delta\varphi.
\]
Letting $\delta\downarrow0$, and using that $D^2v=0$ almost everywhere on $\{|Dv|^2=0\}$, we obtain
\begin{equation}
\label{eq:bochner}
2\int\varphi|D^2v|^2\le\int |Dv|^2\Delta\varphi,
\end{equation}
which is the desired distributional Bochner inequality.

We now prove the quantitative estimate. The outer-variation identity, tested with $\zeta^2v$ for a cutoff $\zeta$ equal to one on $B_{3R/4}$ and supported in $B_R$, gives
\[
\int_{B_{3R/4}}|Dv|^2\le CR^{-2}\int_{B_R}|v|^2.
\]
Applying \eqref{eq:bochner} with a cutoff $\varphi$ equal to one on $B_{R/2}$ and supported in $B_{3R/4}$, we obtain
\[
\int_{B_{R/2}}|D^2v|^2\le CR^{-2}\int_{B_{3R/4}}|Dv|^2\le CR^{-4}\int_{B_R}|v|^2.
\]
Moreover, since $|Dv|^2$ is subharmonic, for every $x\in B_{R/2}$ the mean-value inequality on $B_{R/4}(x)\subset B_{3R/4}$ gives
\[
|Dv(x)|^2\le CR^{-n}\int_{B_{3R/4}}|Dv|^2\le CR^{-n-2}\int_{B_R}|v|^2.
\]
Therefore,
\[
\sup_{B_{R/2}}|Dv|^2+R^{-n}\int_{B_{R/2}}|Dv|^2+R^{2-n}\int_{B_{R/2}}|D^2v|^2\le CR^{-n-2}\int_{B_R}|v|^2.
\]
\end{proof}

\begin{remark}
In hindsight, the Bochner inequality proved above is actually an equality. Indeed, Theorem~\ref{thm:stationarity-of-differential} applies under the present assumption of outer stationarity and gives that \(Dv\) is Dirichlet-stationary. Testing the outer-variation identity for \(Dv\) with a cutoff approximation of \(\varphi Dv\) then gives $\Delta|Dv|^2=2|D^2v|^2$ in $\mathcal D'(\Omega)$.
\end{remark}

\begin{corollary}[Strong compactness]
\label{cor:strong-compactness}
Let \(v_j\in C^1(\mathbb R^n,\mathcal A_q(\mathbb R^k))\) be Dirichlet-stationary maps, and suppose that \(\sup_j\int_{B_R}|v_j|^2<\infty\) for every \(R<\infty\). Then, after passing to a subsequence, \(v_j\to v\) strongly in \(W^{1,2}_{\mathrm{loc}}(\mathbb R^n)\) and locally uniformly, where \(v\in W^{2,2}_{\mathrm{loc}}(\mathbb R^n)\) is Dirichlet-stationary.
\end{corollary}

\begin{proof}
For every \(R>0\), Proposition~\ref{prop:inductive-bochner} and the assumed \(L^2\)-bounds give uniform bounds for \(v_j\) in \(W^{2,2}(B_R)\) and for \(|Dv_j|\) in \(L^\infty(B_R)\). By Lemma~\ref{lem:rellich-W22} and a diagonal argument, after passing to a subsequence there exists \(v\in W^{2,2}_{\mathrm{loc}}(\mathbb R^n,\mathcal A_q(\mathbb R^k))\) such that \(v_j\to v\) strongly in \(W^{1,2}_{\mathrm{loc}}(\mathbb R^n)\). The uniform local Lipschitz bounds and the strong \(L^2_{\mathrm{loc}}\)-convergence also imply that \(v_j\to v\) locally uniformly.

It remains to verify that \(v\) is Dirichlet-stationary. By the strong convergence of the first jets, after passing to a further subsequence we may assume that \(j_1v_j\to j_1v\) almost everywhere. For fixed test fields, the integrands in the outer- and inner-variation identities are symmetric sums of continuous functions of the first-jet sheets, and therefore converge almost everywhere to the corresponding integrands for \(v\). They are supported in a fixed compact set and bounded in absolute value by \(C(1+|Dv_j|^2)\). The uniform local gradient bounds therefore allow us to pass to the limit in both identities by dominated convergence. Thus \(v\) is Dirichlet-stationary.
\end{proof}

\subsection{Hodge structure and stationarity of the differential}
\label{sec:Hodge-structure}

The purpose of this section is to establish structural properties of the linearised problem that will be used both in the Schauder theory and in the higher-dimensional homogeneity gap proved in the next section. We show that a \(W^{2,2}\) outer-stationary multiple-valued map is Dirichlet-stationary, and that its differential carries a natural Hodge structure and is itself Dirichlet-stationary.

Throughout this section, we use superscripts, rather than subscripts, to label the sheets of multiple-valued maps.

\begin{lemma}[Paired Stokes formula]
\label{lem:paired-stokes}
Let \(u=\sum_{\ell=1}^Q\llbracket(f^\ell,g^\ell)\rrbracket\in W^{1,2}_{\mathrm{loc}}\bigl(\Omega,\mathcal A_Q(\mathbb R^a\times\mathbb R^a)\bigr)\), and set \(f:=\sum_{\ell=1}^Q\llbracket f^\ell\rrbracket\) and \(g:=\sum_{\ell=1}^Q\llbracket g^\ell\rrbracket\). Suppose that \(\operatorname{spt}f\cap\operatorname{spt}g\Subset\Omega\). Then, for every \(i,j\in\{1,\ldots,n\}\),
\begin{equation}
\label{eq:paired-stokes}
\int_\Omega\sum_{\ell=1}^Q D_if^\ell\cdot D_jg^\ell=\int_\Omega\sum_{\ell=1}^Q D_jf^\ell\cdot D_ig^\ell.
\end{equation}
\end{lemma}

\begin{proof}
Write \(J_{ij}(u):=\sum_{\ell=1}^Q(D_if^\ell\cdot D_jg^\ell-D_jf^\ell\cdot D_ig^\ell)\). By Sobolev locality, multiplying both \(f\) and \(g\) by a smooth compactly supported cutoff equal to one near \(\operatorname{spt}f\cap\operatorname{spt}g\) does not change either integral in \eqref{eq:paired-stokes}. We may therefore assume that \(u\) is compactly supported in \(\Omega\).

We first prove the result for Lipschitz maps by induction on \(Q\). The case \(Q=1\) follows by smooth approximation and ordinary integration by parts. Assume that the result holds for every multiplicity smaller than \(Q\). Let \(\bar f:=Q^{-1}\sum_{\ell=1}^Qf^\ell\), \(\bar g:=Q^{-1}\sum_{\ell=1}^Qg^\ell\), and \(u_0:=\sum_{\ell=1}^Q\llbracket(f^\ell-\bar f,g^\ell-\bar g)\rrbracket\). Cancellation of the mixed terms gives
\[
J_{ij}(u)=J_{ij}(u_0)+Q(D_i\bar f\cdot D_j\bar g-D_j\bar f\cdot D_i\bar g).
\]
The last term has zero integral by the case \(Q=1\). Thus we may assume that \(u\) has zero average.

Set \(r(x):=\mathcal G(u(x),Q\llbracket(0,0)\rrbracket)\) and \(C:=\{r=0\}\). Since \(u\) is Lipschitz and has zero average, it splits locally on \(\Omega\setminus C\) into Lipschitz maps of multiplicity smaller than \(Q\). Choose a smooth function \(\chi:[0,\infty)\to[0,1]\) such that \(\chi=0\) on \([0,1]\) and \(\chi=1\) on \([2,\infty)\), and set \(\chi_\varepsilon:=\chi(r/\varepsilon)\). Since \(u\) is compactly supported, \(\operatorname{spt}\chi_\varepsilon\Subset\Omega\setminus C\).

Choose a finite smooth partition of unity \(\{\rho_\beta\}\) on a neighbourhood of \(\operatorname{spt}\chi_\varepsilon\), with each \(\rho_\beta\) compactly supported in an open set on which \(u\) splits into maps of smaller multiplicity. Apply the induction hypothesis to each such map, with its first component multiplied by \(\rho_\beta\chi_\varepsilon\). Summing over the pieces and over \(\beta\), and expanding the derivatives, gives
\[
0=\int_\Omega\chi_\varepsilon J_{ij}(u)+\int_\Omega\sum_{\ell=1}^Q f^\ell\cdot\left(D_i\chi_\varepsilon D_jg^\ell-D_j\chi_\varepsilon D_ig^\ell\right).
\]
Since \(|Dr|\leq|Du|\) and \((\sum_\ell|f^\ell|^2)^{1/2}\leq r\), the absolute value of the second integral is bounded by
\[
C\int_{\{0<r<2\varepsilon\}}\frac{r}{\varepsilon}|Du|^2\leq C\int_{\{0<r<2\varepsilon\}}|Du|^2\longrightarrow0.
\]
Moreover, \(Du=0\) almost everywhere on \(C\), so \(J_{ij}(u)=0\) almost everywhere there. Dominated convergence therefore gives \(\int_\Omega J_{ij}(u)=0\), completing the induction.

Finally, let \(u\) be Sobolev, still with compact support. By the Lipschitz approximation of \cite[Proposition~4.4]{DLS11}\footnote{A simpler extrinsic proof can be obtained using the special Almgren embedding known as Brian White's embedding and usually denoted by \(\xi_{\mathrm{BW}}\), which enjoys the property of being a local isometry in the sense of \cite[Corollary~2.2]{DLS11}. In particular, \(|D(\xi_{\mathrm{BW}}\circ u)|=|Du|\) a.e.; see \cite[Proposition~2.20]{DLS11}.}, followed by multiplication by a fixed spatial cutoff, there exist Lipschitz maps \(u_h\), supported in a fixed compact subset of \(\Omega\), and measurable sets \(E_h\) such that \(u_h=u\) almost everywhere outside \(E_h\) and
\[
\int_{E_h}\left(|Du_h|^2+|Du|^2\right)\longrightarrow0.
\]
By Sobolev locality, \(J_{ij}(u_h)=J_{ij}(u)\) almost everywhere outside \(E_h\). Hence
\[
\|J_{ij}(u_h)-J_{ij}(u)\|_{L^1(\Omega)}\leq C\int_{E_h}\left(|Du_h|^2+|Du|^2\right)\longrightarrow0.
\]
Since \(\int_\Omega J_{ij}(u_h)=0\) by the Lipschitz case, passing to the limit proves the claim.
\end{proof}

\begin{lemma}[Hodge structure of the second jet]
\label{lem:Hodge-second-jet}
Let $v=\sum_{\ell=1}^Q\llbracket v^\ell\rrbracket\in W_{\mathrm{loc}}^{2,2}\bigl(\Omega,\mathcal A_Q(\mathbb R^k)\bigr)$ and let \(j_2v:=\sum_{\ell=1}^Q\llbracket(v^\ell,Dv^\ell,D^2v^\ell)\rrbracket\) be the approximate second jet of \(v\), as in Definition~\ref{def:sobolev-spaces}. Then, for a.e. \(x\in\Omega\) and every \(\ell\), \(D^2v^\ell(x)\) is symmetric.

If \(v\) is also outer stationary, then, for a.e. \(x\in\Omega\) and every \(\ell\), \(D^2v^\ell(x)\) is also trace-free, that is
\[
D^2_{ij}v^\ell(x)=D^2_{ji}v^\ell(x),
\qquad
\sum_{i=1}^n D^2_{ii}v^\ell(x)=0.
\]
\end{lemma}

\begin{proof}
Set \(G:=Dv\in W^{1,2}_{\mathrm{loc}}(\Omega, \mathcal A_Q(\mathbb R^{k\times n}))\) and write $G=\sum_{\ell=1}^Q\llbracket G^\ell\rrbracket$, $G^\ell=(G^\ell_1,\ldots,G^\ell_n)$, $G^\ell_i\in\mathbb R^k$.
Fix $\Psi\in C_c^1(\Omega\times\mathbb R^k;\mathbb R^k)$, $\Psi^\ell(x):=\Psi(x,v^\ell(x))$.
The divergence theorem and Lemma \ref{lem:paired-stokes} (applied to $\sum_{\ell=1}^Q \llbracket (v^\ell, \Psi^\ell)\rrbracket$) give
\[
\begin{aligned}
0 & = \int_\Omega D_i \bigg(\sum_{\ell=1}^Q G_j^\ell \cdot \Psi^\ell\bigg) - D_j \bigg( \sum_{\ell=1}^Q G_i^\ell \cdot \Psi^\ell\bigg)\\
& = \int_\Omega \sum_{\ell=1}^Q (D_iG_j^\ell-D_jG_i^\ell)\cdot\Psi^\ell + \int_\Omega \sum_{\ell=1}^Q \big[(G^\ell_j \cdot D_i\Psi^\ell) - (G_i^\ell \cdot D_j\Psi^\ell)\big]\\
& = \int_\Omega \sum_{\ell=1}^Q (D_iG_j^\ell-D_jG_i^\ell)\cdot\Psi^\ell.
\end{aligned}
\]
Similarly, if \(v\) is also outer stationary, the divergence theorem and outer stationarity give
\[
\begin{aligned}
0 &= \int_\Omega \sum_i D_i\bigg(\sum_{\ell=1}^Q G^\ell_i \cdot \Psi^\ell \bigg)\\
& = \int_\Omega \sum_i \sum_{\ell=1}^Q D_iG_i^\ell \cdot\Psi^\ell + \int_\Omega \sum_i \sum_{\ell=1}^Q G^\ell_i \cdot D_i\Psi^\ell\\
& = \int_\Omega \sum_{\ell=1}^Q \left(\sum_iD_iG_i^\ell\right) \cdot\Psi^\ell=0.
\end{aligned}
\]
Thus, taking \(F^\ell=D_iG_j^\ell-D_jG_i^\ell\), or \(F^\ell=\sum_iD_iG_i^\ell\) when \(v\) is outer stationary, we have \(\int_\Omega\sum_{\ell=1}^QF^\ell(x)\cdot\Psi(x,v^\ell(x))\,dx=0\).
Taking \(\Psi(x,y)=\varphi(x)\eta(y)z\), with \(\varphi\in C_c^\infty(\Omega)\), \(\eta\in C_c^\infty(\mathbb R^k)\), and \(z\in\mathbb R^k\), yields \(\sum_{\ell=1}^Q\eta(v^\ell(x))F^\ell(x)\cdot z=0\) for a.e. \(x\).
By first choosing \(\eta\) in a countable dense family and \(z\) among the vectors of the standard basis of \(\mathbb R^k\), and intersecting the corresponding sets of full measure, a density argument gives a common set of full measure on which \(\sum_{\ell=1}^Q\eta(v^\ell(x))F^\ell(x)=0\) for every \(\eta\in C_c^\infty(\mathbb R^k)\).
Thus, fixing \(x\) in this set and \(p\in\operatorname{spt}v(x)\), we may choose \(\eta\) equal to \(1\) near \(p\) and vanishing near the other points of \(\operatorname{spt}v(x)\), obtaining \(\sum_{\{\ell:v^\ell(x)=p\}}F^\ell(x)=0\).

Finally, as $v\in W^{2,2}$ and the approximate differentials of a multifunction in $W^{1,2}$ agree a.e. when the values are equal, it follows that for a.e. $x\in \Omega$ and for any $\ell,m=1,...,Q$,
\[
v^\ell(x)=v^m(x)
\Longrightarrow
G^\ell(x)=G^m(x)
\Longrightarrow
DG^\ell(x)=DG^m(x).
\]
Hence all the \(F^\ell(x)\) in each coincidence cluster are equal.
Their sum being zero, each of them vanishes. Applying this to the two
choices of \(F^\ell\) proves the claim.
\end{proof}

\begin{theorem}[Dirichlet stationarity from outer stationarity]
\label{thm:stationarity-of-differential}
Let $\Omega\subset\mathbb R^n$ be open. If $v\in W^{2,2}_{\mathrm{loc}}\bigl(\Omega,\mathcal A_Q(\mathbb R^k)\bigr)$ is stationary with respect to outer variations, then both \(v\) and \(Dv\) are Dirichlet-stationary.
\end{theorem}

\begin{proof}
Set
\[
G:=Dv\in W^{1,2}_{\mathrm{loc}}\bigl(\Omega,\mathcal A_Q(\mathbb R^{k\times n})\bigr).
\]
Write $G=\sum_{\ell=1}^Q\llbracket G^\ell\rrbracket$, $G^\ell=(G^\ell_1,\ldots,G^\ell_n)$, with $G^\ell_i\in\mathbb R^k$. By Lemma~\ref{lem:Hodge-second-jet}, almost everywhere,
\begin{equation}
\label{eq:Hodge-field-identities}
D_jG^\ell_i=D_iG^\ell_j,\qquad \sum_{j=1}^nD_jG^\ell_j=0
\end{equation}
for every $\ell=1,\ldots,Q$.

Throughout the proof, in order not to burden the notation, the Euclidean scalar product in $\mathbb R^k$ is left implicit.

We first prove outer stationarity of $G$. Let $\Psi=(\Psi_1,\ldots,\Psi_n)\in C_c^1(\Omega\times\R^{k\times n},\R^{k\times n})$, $H_i^\ell:=\Psi_i(x,G^\ell(x))$, and
\[
H_i:=\sum_{\ell=1}^Q\llbracket H_i^\ell\rrbracket\in W_c^{1,2}\bigl(\Omega,\mathcal A_Q(\R^k)\bigr).
\]
By \eqref{eq:Hodge-field-identities} and Lemma~\ref{lem:paired-stokes}, applied to $\sum_{\ell=1}^Q\llbracket(G^\ell_j,H^\ell_i)\rrbracket\in W^{1,2}_{\mathrm{loc}}\bigl(\Omega,\mathcal A_Q(\R^k\times\R^k)\bigr)$, we have
\[
\begin{aligned}
\int_\Omega\sum_{\ell,j}D_jG^\ell\cdot D_j[\Psi(x,G^\ell)] &= \int_\Omega\sum_{\ell,i,j}D_jG^\ell_i\,D_jH^\ell_i\\
&= \int_\Omega\sum_{\ell,i,j}D_iG^\ell_j\,D_jH^\ell_i\\
&= \int_\Omega\sum_{\ell,i,j}D_jG^\ell_j\,D_iH^\ell_i=0.
\end{aligned}
\]
Thus $G$ is stationary with respect to outer variations.

By a standard cutoff approximation in the $Y$-variable, outer stationarity may also be tested against variations $\Phi$ which are compactly supported in $x$ and satisfy $|D_Y\Phi(x,Y)|\leq C<\infty$ and $|\Phi(x,Y)|+|D_x\Phi(x,Y)|\leq C(1+|Y|)$. Indeed, applying the formula to $\vartheta(|Y|/R)\Phi(x,Y)$, where $\vartheta\in C_c^\infty(\R)$ is a smooth cutoff supported in $(-2,2)$ and such that $\vartheta\equiv1$ on $(-1,1)$, and letting $R\to\infty$, the cutoff error is bounded in absolute value by $C\sum_{\ell=1}^Q|DG^\ell|^2\mathbf 1_{\{R\leq|G^\ell|\leq2R\}}$, and hence converges to zero in $L^1$.

We next prove inner stationarity, first for $v$ and then for $G$. Define, for an arbitrary \(Q\)-valued map \(w\in W^{1,2}_{\mathrm{loc}}\), the symmetric tensor
\[
T(w):=\left[|Dw|^2\delta_{ij}-2\sum_{\ell=1}^QD_iw^\ell D_jw^\ell\right]dx^i\otimes dx^j.
\]
We first show that $T(v)=\left[|G|^2\delta_{ij}-2\sum_{\ell=1}^QG_i^\ell G_j^\ell\right]dx^i\otimes dx^j$ is divergence-free. Indeed, for every $X\in C_c^1(\Omega,\R^n)$, integration by parts and \eqref{eq:Hodge-field-identities} give
\[
\begin{aligned}
\frac12\int_\Omega\sum_{i,j=1}^nT_{ij}(v)D_iX^j &= \int_\Omega\sum_{\ell=1}^Q\left\{\frac12|G^\ell|^2\operatorname{div}X-\sum_{i,j=1}^nG_i^\ell\, G_j^\ell \,D_iX^j\right\}\\
&= -\int_\Omega\sum_{\ell,i,j}G_i^\ell \,D_jG_i^\ell\,X^j+\int_\Omega\sum_{\ell,i,j}\left(D_iG_i^\ell \,G_j^\ell+G_i^\ell\, D_iG_j^\ell\right)X^j=0.
\end{aligned}
\]
Thus, for all $j=1,\ldots,n$,
\begin{equation}
\label{eq:Hodge-field-T-divergence}
\operatorname{div}T_j(v)=0\qquad\text{in }\mathcal D'(\Omega).
\end{equation}
Since $v$ is stationary with respect to outer variations by assumption and \(T(v)\) is the integrand of the inner variation formula for \(v\), it follows that $v$ is Dirichlet-stationary.

Testing the outer-variation formula for $G$ with $\Phi(x,Y)=\varphi(x)Y$, where $\varphi\in C_c^\infty(\Omega)$, gives
\begin{equation}
\label{eq:Hodge-field-first-member-stress}
\frac12\Delta|G|^2=|DG|^2\qquad\text{in }\mathcal D'(\Omega).
\end{equation}
Similarly, testing with $\Phi(x,Y):=\varphi(x)(Y_i\otimes e_j+Y_j\otimes e_i)$ gives
\begin{equation}
\label{eq:Hodge-field-second-member-stress}
\frac12\Delta\left[\sum_{\ell=1}^QG_i^\ell G_j^\ell\right]=\sum_{\ell=1}^Q\sum_{r=1}^nD_rG_i^\ell D_rG_j^\ell\qquad\text{in }\mathcal D'(\Omega).
\end{equation}
Consequently, by \eqref{eq:Hodge-field-first-member-stress} and \eqref{eq:Hodge-field-second-member-stress}, together with the symmetry \(D_iG_r^\ell=D_rG_i^\ell\) in \eqref{eq:Hodge-field-identities}, we get
\begin{equation}
\label{eq:Hodge-field-stress-laplacian}
T(Dv)=\frac12\Delta T(v)\qquad\text{in }\mathcal D'(\Omega).
\end{equation}

Combining \eqref{eq:Hodge-field-T-divergence} and \eqref{eq:Hodge-field-stress-laplacian}, we obtain, for all $j=1,\ldots,n$,
\[
\operatorname{div}T_j(Dv)=\frac12\Delta\bigl(\operatorname{div}T_j(v)\bigr)=0.
\]
This is precisely the inner-variation identity for $Dv$, and hence $Dv$ is Dirichlet-stationary as well.
\end{proof}

\subsection{Homogeneity gap at $1$}

\subsubsection{The case $n=2$}

\begin{theorem}[$2$-dimensional homogeneity gap]
\label{thm:2d-cone-gap}
Let \(w\in W_{\mathrm{loc}}^{2,2}\bigl(\mathbb R^2,\mathcal A_q(\mathbb R^k)\bigr)\) be a nonzero Dirichlet-stationary map, homogeneous of degree \(\sigma>1\). Then \(\sigma\ge1+\frac1q\).
\end{theorem}

\begin{proof}
As \(w\) is homogeneous of degree \(\sigma\), we can fix a representative of \(j_1w\), still denoted by \(j_1w\), satisfying the corresponding pointwise scaling law, with degrees \(\sigma\) and \(\sigma-1\) in its value and differential components, respectively. Then, by slicing and homogeneity, the map \(\theta\mapsto j_1w(\cos\theta,\sin\theta)\) belongs to \(W^{1,2}_{\mathrm{loc}}(\mathbb R)\).

By the one-dimensional Sobolev selection theorem, we can choose \(\phi_\ell\in W^{1,2}_{\mathrm{loc}}(\mathbb R,\mathbb R^k)\) and \(A_\ell\in W^{1,2}_{\mathrm{loc}}\bigl(\mathbb R,\operatorname{Hom}(\mathbb R^2,\mathbb R^k)\bigr)\) such that \(j_1w(\cos\theta,\sin\theta) =\sum_{\ell=1}^q\llbracket(\phi_\ell(\theta),A_\ell(\theta))\rrbracket\) for almost every \(\theta\in\mathbb R\). By homogeneity, \(j_1w(r\cos\theta,r\sin\theta) = \sum_{\ell=1}^q \llbracket(r^\sigma\phi_\ell(\theta),r^{\sigma-1}A_\ell(\theta)) \rrbracket\) for almost every \((r,\theta)\in(0,\infty)\times\mathbb R\).

Let \(I\subset\mathbb R\) be an open interval with \(|I|<2\pi\) and set \(\mathcal C_I:=\{(r\cos\theta,r\sin\theta):r>0,\ \theta\in I\}\). The maps \(w_\ell(r\cos\theta,r\sin\theta):=r^\sigma\phi_\ell(\theta)\) belong to \(W^{1,2}_{\mathrm{loc}}(\mathcal C_I,\mathbb R^k)\) and form a selection of \(w\) on \(\mathcal C_I\). The first-jet identity and the classical chain rule in polar coordinates give \(\phi_\ell'(\theta)=A_\ell(\theta)(-\sin \theta ,\cos \theta )\) almost everywhere on \(I\). Since \(A_\ell\in W^{1,2}_{\mathrm{loc}}(I)\), it follows that \(\phi_\ell\in W^{2,2}_{\mathrm{loc}}(I)\) and hence \(w_\ell\in W^{2,2}_{\mathrm{loc}}(\mathcal C_I,\mathbb R^k)\).

Integration by parts in the outer-variation identity gives \(\int_{\mathcal C_I}\sum_\ell\Delta w_\ell(x)\cdot\Psi(x,w_\ell(x))\,dx=0\) for every \(\Psi\in C_c^1(\mathcal C_I\times\mathbb R^k,\mathbb R^k)\). Testing with \(\Psi(x,Y)=\eta(x)\psi(Y)\), and using a countable uniformly dense family of target tests, shows that for almost every \(x\) and every \(a\in\operatorname{spt}w(x)\), one has \(\sum_{\{\ell:w_\ell(x)=a\}}\Delta w_\ell(x)=0\).

On the other hand, Sobolev locality applied successively to \(w_\ell-w_{\ell'}\) and its gradient gives \(D^2w_\ell=D^2w_{\ell'}\) almost everywhere on \(\{w_\ell=w_{\ell'}\}\). Thus all the Laplacians in each preceding sum coincide, and hence \(\Delta w_\ell=0\) almost everywhere for every \(\ell\). The polar-coordinate formula \(\Delta w_\ell(r\cos\theta,r\sin\theta)=r^{\sigma-2}(\phi_\ell''(\theta)+\sigma^2\phi_\ell(\theta))\) therefore gives \(\phi_\ell''+\sigma^2\phi_\ell=0\) almost everywhere on \(I\). Since \(I\) was arbitrary, this equation holds on \(\mathbb R\), and consequently \(\phi_\ell(\theta)=a_\ell\cos(\sigma\theta)+b_\ell\sin(\sigma\theta)\) for some \(a_\ell,b_\ell\in\mathbb R^k\).

Since \(\theta\mapsto j_1w(\cos\theta,\sin\theta)\) and \((-\sin,\cos)\) are \(2\pi\)-periodic, so is the multiple-valued map \(\theta\mapsto\sum_{\ell=1}^q\llbracket(\phi_\ell(\theta),\phi_\ell'(\theta))\rrbracket\). Hence there exists a permutation \(\pi\) of \(\{1,\ldots,q\}\) such that \((\phi_\ell(2\pi),\phi_\ell'(2\pi))=(\phi_{\pi(\ell)}(0),\phi_{\pi(\ell)}'(0))\). Uniqueness for the Cauchy problem gives \(\phi_\ell(\theta+2\pi)=\phi_{\pi(\ell)}(\theta)\) for every \(\theta\in\mathbb R\).

Since \(w\not\equiv0\), choose \(\ell\) with \(\phi_\ell\not\equiv0\), and let \(m\leq q\) be the length of its cycle under \(\pi\). Then \(\phi_\ell(\theta+2\pi m)=\phi_\ell(\theta)\). The explicit representation of \(\phi_\ell\) implies \(\sigma m\in\mathbb N\). Since \(\sigma>1\), one has \(\sigma m\geq m+1\), and therefore \(\sigma\geq1+\frac1m\geq1+\frac1q\).
\end{proof}

\subsubsection{The case $n\geq 3$}

\begin{lemma}[Monotonicity formula for matrix-valued measures with an aperture bound]
\label{lem:atomic-load}
Let $m\geq 1$ be an integer, $\mathsf S\in \mathcal M\bigl(B^m_1;\operatorname{Sym}^+(\mathbb R^m)\bigr)$ be a finite positive-semidefinite matrix-valued measure, and set $\mu:=\operatorname{tr}\mathsf S$.
Assume that
\begin{equation}
\label{eq:atomic-divergence}
-\operatorname{div}\mathsf S=a\delta_0
\qquad\text{in }B^m_1,
\end{equation}
in the sense that $\int_{B^m_1}D\Psi:d\mathsf S = a\cdot\Psi(0)$ for every $
\Psi\in C_c^1(B^m_1;\mathbb R^m)$.
Suppose moreover that there is \(c\in (0,1]\) such that for every $e\in\mathbb S^{m-1}$
\begin{equation}
\label{eq:aperture-lemma}
e^T\mathsf S e\leq c\mu \qquad\text{as measures}.
\end{equation}
Then $r\mapsto r^{-1/c}\mu(B^m_r)$ is nondecreasing on $(0,1)$. If \(c<1\), then \(a=0\).
\end{lemma}

\begin{proof}
Since \(\mathsf S\) is positive semidefinite and \(\mu=\operatorname{tr}\mathsf S\), one has \(\mathsf S\ll\mu\). We may therefore write
\begin{equation}
\label{eq:polar-decomposition-S}
d\mathsf S(y)=P(y)\,d\mu(y),
\end{equation}
where $P(y)\in\operatorname{Sym}^+(\mathbb R^m)$, $\operatorname{tr}P(y)=1$ for \(\mu\)-almost every \(y\).

Using a countable dense subset of \(\mathbb S^{m-1}\), the measure inequality \eqref{eq:aperture-lemma} implies that, outside a single \(\mu\)-negligible set, \begin{equation}
\label{eq:pointwise-aperture}
e^TP(y)e\leq c \qquad \text{ for every $e\in\mathbb S^{m-1}$.}
\end{equation}

Fix \(\delta\in(0,1)\), and choose a nonincreasing function $\varphi_\delta\in C^1([0,\infty),[0,1])$ such that
\[
\varphi_\delta(t)=1
\quad\text{for }0\leq t\leq1,
\qquad
\varphi_\delta(t)=0
\quad\text{for }t\geq1+\delta.
\]
For $0<\rho<\frac{1}{1+\delta}$, define the smoothed mass $M_\delta(\rho) := \int_{B_1} \varphi_\delta\left(\frac{|y|}{\rho}\right)\,d\mu(y)$.
The function \(M_\delta\) is continuously differentiable on \((0,1/(1+\delta))\), with $M_\delta'(\rho) = \int_{B_1} \frac{\partial}{\partial\rho} \left[ \varphi_\delta\left(\frac{|y|}{\rho}\right) \right]\,d\mu(y)$.
Moreover, $\frac{\partial}{\partial\rho} \left[ \varphi_\delta\left(\frac{|y|}{\rho}\right) \right] = -\frac{|y|}{\rho^2} \varphi_\delta'\left(\frac{|y|}{\rho}\right) \geq0$.

We test \eqref{eq:atomic-divergence} with the radial vector field $\Psi_\rho(y) := \varphi_\delta\left(\frac{|y|}{\rho}\right)y$.
Since \(\Psi_\rho\in C_c^1(B_1,\mathbb R^m)\) and \(\Psi_\rho(0)=0\), we obtain
\begin{equation}
\label{eq:radial-test-zero}
\int_{B_1}D\Psi_\rho:d\mathsf S=0.
\end{equation}

Set \(r:=|y|\), and for \(y\neq0\) let $\widehat y:=\frac{y}{|y|}$. A direct computation gives
\begin{equation}
\label{eq:radial-test-derivative}
D\Psi_\rho(y)
=
\varphi_\delta\left(\frac r\rho\right)I
-
\rho
\frac{\partial}{\partial\rho}
\left[
\varphi_\delta\left(\frac r\rho\right)
\right]
\widehat y\otimes\widehat y.
\end{equation}
Indeed, $D\left[ \varphi_\delta\left(\frac r\rho\right)y \right] = \varphi_\delta\left(\frac r\rho\right)I + \frac r\rho \varphi_\delta'\left(\frac r\rho\right) \widehat y\otimes\widehat y$ and $\frac r\rho \varphi_\delta'\left(\frac r\rho\right) = -\rho \frac{\partial}{\partial\rho}\left[
\varphi_\delta\left(\frac r\rho\right) \right]$.

Substituting \eqref{eq:polar-decomposition-S} and \eqref{eq:radial-test-derivative} into \eqref{eq:radial-test-zero}, and using \(\operatorname{tr}P=1\), yields
$$
M_\delta(\rho)
=
\rho
\int_{B_1}
\frac{\partial}{\partial\rho}
\left[
\varphi_\delta\left(\frac{|y|}{\rho}\right)
\right]
\widehat y^TP(y)\widehat y
\,d\mu(y).
$$

Both factors in the integrand on the right-hand side are
nonnegative. By \eqref{eq:pointwise-aperture}, $\widehat y^TP(y)\widehat y\leq c$ for \(\mu\)-almost every \(y\neq0\). Hence
\begin{equation}
\label{eq:smoothed-differential-inequality}
M_\delta(\rho)
\leq
c\rho
\int_{B_1}
\frac{\partial}{\partial\rho}
\left[
\varphi_\delta\left(\frac{|y|}{\rho}\right)
\right]\,d\mu(y)
=
c\rho M_\delta'(\rho).
\end{equation}

Set $p:=\frac1c\geq 1$.
It follows from \eqref{eq:smoothed-differential-inequality} that $M_\delta'(\rho) - \frac p\rho M_\delta(\rho) \geq0$.
Therefore $\frac{d}{d\rho} \left( \rho^{-p}M_\delta(\rho) \right) = \rho^{-p} \left( M_\delta'(\rho) - \frac p\rho M_\delta(\rho) \right) \geq0$.
Thus \(\rho\mapsto\rho^{-p}M_\delta(\rho)\) is nondecreasing. Letting $\delta\downarrow0$ and then approximating the radii from below,
we conclude that $r\mapsto r^{-p}\mu(B_r)$ is nondecreasing on $(0,1)$.
Consequently, for $0<r<R_0<1$, we have $\mu(B_r) \leq \mu(B_{R_0}) \left(\frac r{R_0}\right)^p$.
In particular, for a suitable constant \(C_0<\infty\),
\begin{equation}
\label{eq:power-growth}
\mu(B_r(0))\leq C_0r^p,
\qquad
p\geq 1.
\end{equation}

Suppose now that \(c<1\), so that \(p>1\), and assume toward a contradiction that \(a\neq0\). Set \(e:=\frac a{|a|}\).
For \(r>0\) sufficiently small, choose \(\chi_r\in C_c^1(B_{2r}(0))\) such that
\[
0\leq\chi_r\leq1,
\qquad
\chi_r=1\text{ on }B_r(0),
\qquad
|D\chi_r|\leq\frac Cr.
\]
Testing \eqref{eq:atomic-divergence} with $\Psi(y):=e\,\chi_r(y)$ gives $|a| = \int_{B_1}D(e\chi_r):d\mathsf S$.
Using \(d\mathsf S=P\,d\mu\), \(\operatorname{tr}P=1\), and $|A:P| \leq |A|\,\operatorname{tr}P$ for $P\in\operatorname{Sym}^+(\mathbb R^m)$, we obtain
\[
|a|
\leq
\int_{B_1}
|D\chi_r|\,d\mu
\leq
\frac Cr
\mu\bigl(B_{2r}(0)\setminus B_r(0)\bigr)
\leq
\frac Cr\mu(B_{2r}(0)).
\]
Hence
\begin{equation}
\label{eq:lower-growth}
\mu(B_{2r}(0))
\geq C^{-1}|a|\,r.
\end{equation}

On the other hand, \eqref{eq:power-growth} gives
\[
\mu(B_{2r}(0))
\leq C_0(2r)^p.
\]
Since \(p>1\), this contradicts
\eqref{eq:lower-growth} as \(r\downarrow0\). Therefore \(a=0\).
\end{proof}

\begin{remark}[Relation with the monotonicity formula for varifolds]
\label{rem:varifold-monotonicity}
The preceding argument is the matrix-valued analogue of the classical
monotonicity argument for stationary varifolds. Indeed, if \(V\) is an
\(n\)-dimensional varifold in \(\mathbb R^m\), consider the
positive-semidefinite matrix-valued measure
\[
\mathsf S_V(A)
:=
\frac 1n \int_{A\times G(n,m)} P_T\,dV(x,T),
\]
where \(P_T\) denotes the orthogonal projection onto \(T\). Then
\[
\operatorname{tr}\mathsf S_V=\|V\|,
\qquad
-\operatorname{div}\mathsf S_V= \frac 1n \delta V,
\]
and hence \(\operatorname{div}\mathsf S_V=0\) whenever \(V\) is
stationary. To verify $-\operatorname{div}\mathsf S_V= \frac 1 n \delta V$ it is enough to note that $P_T: D\Psi= \operatorname{div}_T \Psi$, which in turn follows by diagonalizing $P_T = \sum_{i=1}^n \tau_i\otimes \tau_i$, where $(\tau_i)_{i=1,...,n}$ is an orthonormal basis of $T$ in $\R^m$, so that $P_T :D\Psi = \sum_{i=1}^n \tau_i\otimes \tau_i : D\Psi = \sum_{i=1}^n \tau_i \cdot D\Psi[\tau_i] = \operatorname{div}_T\Psi$.

Moreover, for every \(e\in\mathbb S^{m-1}\), $e^TP_Te=|P_Te|^2\leq 1 =\frac{1}{n}\operatorname{tr}P_T$.
Consequently,
\[
e^T\mathsf S_Ve
\leq
\frac{1}{n}\operatorname{tr}\mathsf S_V,
\]
so the aperture constant is \(c=1/n\). The exponent produced by the
preceding radial argument is therefore $p=\frac{1}{c}=n$, and the corresponding monotone quantity is precisely the usual density
ratio $\rho^{-n}\|V\|(B_\rho)$.
If, instead of using only \(|P_T\widehat x|^2\leq1\), one retains the identity $1-|P_T\widehat x|^2=|P_{T^\perp}\widehat x|^2$, one recovers the standard defect term in the varifold monotonicity formula.
Thus the exponent \(p=1/c\) is the natural generalization of the familiar varifold mechanism.
In particular, if \(\mathsf S\) in Lemma~\ref{lem:atomic-load} were induced by an \(n\)-dimensional varifold \(V\) with $n\geq 2$, the lemma recovers a standard consequence of
the varifold monotonicity formula: the first variation cannot be locally
equal to a nonzero isolated atom.
\end{remark}

\begin{lemma}[Quantitative aperture of trace-free symmetric tensors]
\label{lem:trace-free-aperture}
Let \(E:=\R^{k\times n}\simeq(\R^n)^k\). Suppose that \(P\in\operatorname{Hom}(\R^n,E) \) is such that $Pz=(P^1z,\ldots,P^kz)$, where each $P^a\in \operatorname{Sym}_0(\R^n)$ (the space of symmetric and trace-free homomorphisms).
Then $\|P\|_{\mathrm{op}}^2\leq \frac{n-1}{n} |P|^2$, where $|P|^2:= \operatorname{tr}(PP^*)$.
Equivalently,
\[
PP^*\leq\frac{n-1}{n} |P|^2 I_E
\]
in the sense of quadratic forms.
\end{lemma}

\begin{proof}
We first consider the case \(k=1\). Let \(P\in\operatorname{Sym}_0(\R^n)\), and let \(\lambda_1,\ldots,\lambda_n\) be its eigenvalues. Since \(P\) is trace-free, $\lambda_i=-\sum_{j\neq i}\lambda_j$ for every \(i\). Hence $\lambda_i^2\leq(n-1)\sum_{j\neq i}\lambda_j^2=(n-1)(|P|^2-\lambda_i^2)$, and therefore $\|P\|_{\mathrm{op}}^2=\max_i\lambda_i^2\leq\frac{n-1}{n}|P|^2$.

For general \(k\), for every \(z\in\R^n\),
\[
|Pz|^2 =\sum_{a=1}^k|P^az|^2 \leq\sum_{a=1}^k\|P^a\|_{\mathrm{op}}^2|z|^2 \leq\frac{n-1}{n}\sum_{a=1}^k|P^a|^2|z|^2 =\frac{n-1}{n}|P|^2|z|^2.
\]
Taking the supremum over \(|z|=1\) gives $\|P\|_{\mathrm{op}}^2\leq\frac{n-1}{n}|P|^2$.

To show the final equivalence, let \(e\in E\) and compute
\[
PP^*e\cdot e = |P^*e|^2 \leq \|P^*\|_{\mathrm{op}}^2|e|^2 = \|P\|_{\mathrm{op}}^2|e|^2 \leq \frac{n-1}{n}|P|^2|e|^2.
\]
\end{proof}

\begin{theorem}[Homogeneity gap for $n\geq3$]
\label{thm:homogeneity-one-gap}
Let \(n\geq3\), \(k\geq1\), and \(Q\geq1\). There exists \(\delta=\delta(n,k,Q)>0\) with the following property. Let $v\in W^{2,2}_{\mathrm{loc}}\bigl(\R^n,\mathcal A_Q(\R^k)\bigr)$ be nonzero, Dirichlet-stationary, and homogeneous of degree \(\sigma\geq1\). Then
\[
\sigma=1\quad\text{or}\quad\sigma\geq1+\delta(n,k,Q).
\]
\end{theorem}

\begin{proof}
Suppose by contradiction that the conclusion is false. Then there are nonzero Dirichlet-stationary maps $v_h\in W^{2,2}_{\mathrm{loc}}\bigl(\R^n,\mathcal A_Q(\R^k)\bigr)$ which are homogeneous of degree \(\sigma_h=1+\alpha_h\), where $\alpha_h>0$ and $\alpha_h\longrightarrow0$.
Set $E:=\R^{k\times n}$ and $G_h:=Dv_h$.
By Theorem~\ref{thm:stationarity-of-differential}, \(G_h\) is Dirichlet-stationary and in particular outer-stationary. Moreover, by Lemma~\ref{lem:Hodge-second-jet}, writing $G_h=\sum_{\ell=1}^Q\llbracket G_h^\ell\rrbracket$, $G_h^\ell=(G_{h,1}^\ell,\ldots,G_{h,n}^\ell)$, one has
\begin{equation}
\label{eq:gap-Hodge-identities}
D_iG_{h,j}^\ell=D_jG_{h,i}^\ell,
\qquad
\sum_{i=1}^nD_iG_{h,i}^\ell=0
\end{equation}
almost everywhere.

Since \(v_h\) is homogeneous of degree \(1+\alpha_h\), the map \(G_h\) is homogeneous of degree \(\alpha_h\). Multiplying \(v_h\) by a constant, we may assume that
\begin{equation}
\label{eq:gap-normalization}
\int_{\mathbb S^{n-1}}|G_h|^2=1.
\end{equation}
Since \(G_h\) is Dirichlet-stationary and homogeneous of degree \(\alpha_h\), its frequency centred at the origin is identically equal to \(\alpha_h\). Hence
\begin{equation}
\label{eq:gap-bulk-energy}
\int_{B_1^n}|DG_h|^2
=
\alpha_h\int_{\mathbb S^{n-1}}|G_h|^2
=
\alpha_h.
\end{equation}

\smallskip
\noindent
\emph{Step 1: convergence to a constant \(Q\)-point.}
The decomposition in polar coordinates of the Dirichlet energy of the homogeneous map \(G_h\), together with \eqref{eq:gap-normalization} and \eqref{eq:gap-bulk-energy}, gives
\[
\int_{\mathbb S^{n-1}}|D_{\mathbb S^{n-1}}G_h|^2
=
\alpha_h(\alpha_h+n-2).
\]
By the Poincar\'e inequality for \(Q\)-valued functions on \(\mathbb S^{n-1}\), there exists \(T_h\in\mathcal A_Q(E)\) such that
\[
\int_{\mathbb S^{n-1}}\mathcal G(G_h,T_h)^2
\leq
C(n,Q)\alpha_h(\alpha_h+n-2)
\longrightarrow0.
\]
The normalization \eqref{eq:gap-normalization} implies that the \(T_h\) are bounded. After passing to a subsequence,
\begin{equation}
\label{eq:gap-constant-limit}
T_h\longrightarrow
T=\sum_{\beta=1}^Nq_\beta\llbracket A^\beta\rrbracket,
\qquad
G_h|_{\mathbb S^{n-1}}\longrightarrow T
\quad\text{strongly in }L^2(\mathbb S^{n-1}),
\end{equation}
where the \(A^\beta\in E\) are distinct, \(q_\beta\geq1\), and \(\sum_{\beta=1}^Nq_\beta=Q\). Moreover,
\begin{equation}
\label{eq:gap-nonzero-limit}
|\mathbb S^{n-1}|
\sum_{\beta=1}^Nq_\beta|A^\beta|^2
=
1,
\end{equation}
so at least one \(A^\beta\) is nonzero.

Outer stationarity of \(G_h\), tested with a target-cutoff approximation of \((x,Y)\mapsto\varphi(x)Y\), gives $\Delta|G_h|^2=2|DG_h|^2\geq0$ in \(\mathcal D'( \R^n)\). The mean-value inequality, homogeneity, and \eqref{eq:gap-normalization} therefore give
\begin{equation}
\label{eq:gap-uniform-range}
\operatorname*{ess\,sup}_{B_1^n}|G_h|^2\leq C(n).
\end{equation}

\smallskip
\noindent
\emph{Step 2: the target stress.}
Define the matrix-valued measure \(\mathsf S_h\) by
\begin{equation}
\label{eq:gap-target-stress}
\mathsf S_h(dy)
:=
\frac1{\alpha_h}
\int_{B_1^n}
\sum_{\ell=1}^Q
DG_h^\ell(x)(DG_h^\ell(x))^*
\,\delta_{G_h^\ell(x)}(dy)\,dx.
\end{equation}
Since $\alpha_h=\int_{B_1^n}|DG_h|^2$, the measure \(\mathsf S_h\) is precisely the target-stress measure associated with \(G_h\), normalized by the total Dirichlet energy of \(G_h\) in \(B_1^n\). Thus $\operatorname{tr}\mathsf S_h(E)=1$.
Moreover, \eqref{eq:gap-uniform-range} shows that the measures \(\mathsf S_h\) are supported in a fixed compact subset of \(E\). Then, after passing to a further subsequence,
\[
\mathsf S_h\stackrel{*}{\rightharpoonup}\mathsf S
\qquad\text{in }
\mathcal M\bigl(E;\operatorname{Sym}^+(E)\bigr).
\]

Let \(\Psi\in C_c^1(E;E)\). Testing the outer-stationarity identity for \(G_h\) with a radial cutoff approximation of the map $(x,Y)\longmapsto\mathbf 1_{B_1^n}(x)\Psi(Y)$ gives
\[
\int_{B_1^n}
\sum_{\ell=1}^Q
D\Psi(G_h^\ell):
\bigl(DG_h^\ell(DG_h^\ell)^*\bigr)
=
\int_{\mathbb S^{n-1}}
\sum_{\ell=1}^Q
\Psi(G_h^\ell)\cdot\partial_rG_h^\ell.
\]
Since \(G_h\) is homogeneous of degree \(\alpha_h\), \(\partial_rG_h^\ell=\alpha_hG_h^\ell\) on \(\mathbb S^{n-1}\). Consequently,
\[
\int_E D\Psi:d\mathsf S_h
=
\int_{\mathbb S^{n-1}}
\sum_{\ell=1}^Q
\Psi(G_h^\ell)\cdot G_h^\ell.
\]
Passing to the limit using \eqref{eq:gap-constant-limit} and \eqref{eq:gap-uniform-range}, we obtain
\begin{equation}
\label{eq:gap-atomic-equilibrium}
-\operatorname{div}_Y\mathsf S
=
|\mathbb S^{n-1}|
\sum_{\beta=1}^N
q_\beta A^\beta\,\delta_{A^\beta}.
\end{equation}

\smallskip
\noindent
\emph{Step 3: the aperture estimate and conclusion.}
By \eqref{eq:gap-Hodge-identities}, for almost every \(x\) and every sheet \(\ell\), the linear map \(DG_h^\ell(x)\in\operatorname{Hom}(\mathbb R^n,E)\) satisfies the assumptions of Lemma~\ref{lem:trace-free-aperture}. Hence
\[
DG_h^\ell(DG_h^\ell)^*
\leq
\frac{n-1}{n}|DG_h^\ell|^2I_E.
\]
It follows directly from \eqref{eq:gap-target-stress} that, for every unit vector \(e\in E\),
\[
e^T\mathsf S_he
\leq
\frac{n-1}{n}\operatorname{tr}\mathsf S_h
\qquad\text{as measures}.
\]
Passing to the limit gives
\begin{equation}
\label{eq:gap-limit-aperture}
e^T\mathsf S e
\leq
\frac{n-1}{n}\operatorname{tr}\mathsf S
\qquad\text{as measures}
\end{equation}
for every unit vector \(e\in E\).

Choose \(A^\beta\neq0\), whose existence follows from \eqref{eq:gap-nonzero-limit}, and choose \(R_\beta>0\) so that \(B_{R_\beta}^E(A^\beta)\) contains no other \(A^\gamma\). In this ball, \eqref{eq:gap-atomic-equilibrium} reads
\[
-\operatorname{div}_Y\mathsf S
=
|\mathbb S^{n-1}|q_\beta A^\beta\,\delta_{A^\beta}.
\]
After translating \(A^\beta\) to the origin and rescaling the ball to the unit ball, Lemma~\ref{lem:atomic-load} applies by \eqref{eq:gap-limit-aperture}. Since \((n-1)/n<1\), it gives $|\mathbb S^{n-1}|q_\beta A^\beta=0$, contradicting \(A^\beta\neq0\). This proves the theorem.
\end{proof}

Decreasing the constants if necessary, we assume that $\delta(n,k,Q)\leq1$ and that $Q\mapsto\delta(n,k,Q)$ is nonincreasing.

\subsection{A Liouville theorem}

\begin{theorem}[Liouville theorem]\label{thm:liouville}
Let \(n\geq2\), \(v\in C^{1,\alpha}(\mathbb R^n,\mathcal A_q(\mathbb R^k))\) be Dirichlet-stationary and suppose \([Dv]_{\alpha,\mathbb R^n}<\infty\), where \(0<\alpha<\frac1q\) if \(n=2\), and \(0<\alpha<\delta(n,k,q)\) if \(n\geq3\), with \(\delta(n,k,q)\) as in Theorem~\ref{thm:homogeneity-one-gap}.
Then \(v\) is affine.
\end{theorem}

\begin{proof}
Let \(D_v(r)\), \(H_v(r)\), and \(N_v(r)\) denote the frequency quantities associated with \(v\) and centred at the origin. If \(v\equiv0\), there is nothing to prove. Otherwise, by Theorem~\ref{prop:frequency-monotonicity-final}\textup{(i)}, \(H_v(r)>0\) for every \(r>0\), and hence \(N_v\) is well defined and nondecreasing on \((0,\infty)\).

The global Hölder bound on \(Dv\) gives \(|v(x)|\leq C(1+|x|^{1+\alpha})\), and hence \(H_v(r)\leq C(1+r^{2+2\alpha})\). Therefore \(N_v(\infty):=\lim_{r\to\infty}N_v(r)\leq1+\alpha\). Indeed, if \(N_v(r_0)>1+\alpha\), Theorem~\ref{prop:frequency-monotonicity-final}\textup{(i)} would give \(H_v(r)\geq H_v(r_0)(r/r_0)^{2N_v(r_0)}\) for every \(r\geq r_0\), contradicting the preceding growth bound.

We claim that \(N_v(\infty)\leq1\). Suppose instead that \(1<N_v(\infty)\leq1+\alpha\). Choose \(R_j\to\infty\) and define \(v_j(x):=v(R_jx)/H_v(R_j)^{1/2}\). Then \(H_{v_j}(1)=1\) and \(N_{v_j}(r)=N_v(R_jr)\) for every \(r>0\). Theorem~\ref{prop:frequency-monotonicity-final}\textup{(i)} and the monotonicity of \(H_v\) give \(H_{v_j}(r)\leq\max\{1,r^{2+2\alpha}\}\) for every \(r>0\), and hence \(\sup_j\int_{B_R}|v_j|^2<\infty\) for every \(R<\infty\). By Corollary~\ref{cor:strong-compactness}, after passing to a subsequence, \(v_j\to w\) strongly in \(W^{1,2}_{\mathrm{loc}}(\mathbb R^n)\) and locally uniformly, for some Dirichlet-stationary \(w\in W^{2,2}_{\mathrm{loc}}(\mathbb R^n)\).

By local uniform convergence, \(H_w(1)=1\), so \(w\not\equiv0\). Hence Theorem~\ref{prop:frequency-monotonicity-final}\textup{(i)} gives \(H_w(r)>0\) for every \(r>0\). Strong \(W^{1,2}_{\mathrm{loc}}\)-convergence and local uniform convergence therefore give, for every \(r>0\),
\[
N_w(r)=\frac{D_w(r)}{H_w(r)}=\lim_{j\to\infty}\frac{D_{v_j}(r)}{H_{v_j}(r)}=\lim_{j\to\infty}N_v(R_jr)=N_v(\infty).
\]
Thus \(w\) is nonzero and, by Theorem~\ref{prop:frequency-monotonicity-final}\textup{(iii)}, homogeneous of degree \(N_v(\infty)\).

If \(n=2\), Theorem~\ref{thm:2d-cone-gap} gives \(N_v(\infty)\geq1+\frac1q\), whereas if \(n\geq3\), Theorem~\ref{thm:homogeneity-one-gap} gives \(N_v(\infty)\geq1+\delta(n,k,q)\). Both conclusions contradict \(N_v(\infty)\leq1+\alpha\). Hence \(N_v(\infty)\leq1\).

Since \(N_v(R)\leq N_v(\infty)\leq1\), Theorem~\ref{prop:frequency-monotonicity-final}\textup{(i)} gives \(H_v(R)\leq H_v(1)R^{2N_v(R)}\leq CR^2\) for every \(R\geq1\). Consequently, for every \(R\geq1\),
\[
\int_{B_R}|Dv|^2=R^{n-2}N_v(R)H_v(R)\leq CR^n.
\]
By Proposition~\ref{prop:inductive-bochner}, \(v\in W^{2,2}_{\mathrm{loc}}(\mathbb R^n)\) and \(|Dv|^2\) is subharmonic. For every \(R\geq1\), the mean-value inequality gives
\[
\sup_{B_{R/2}}|Dv|^2\leq\frac{C}{R^n}\int_{B_R}|Dv|^2\leq C.
\]
Letting \(R\to\infty\), we obtain a global upper bound for \(|Dv|^2\), and hence \(Dv\) is bounded.

Set \(G:=Dv\). By Theorem~\ref{thm:stationarity-of-differential}, \(G\) is Dirichlet-stationary. Let \(D_G\), \(H_G\), and \(N_G\) denote its frequency quantities centred at the origin. The boundedness of \(G\) gives \(H_G(R)\leq C\) for every \(R>0\).

If \(G\equiv q\llbracket0\rrbracket\), the conclusion is immediate. Otherwise, by Theorem~\ref{prop:frequency-monotonicity-final}\textup{(i)}, \(H_G(r)>0\) for every \(r>0\). If \(N_G(r_0)=\beta>0\) for some \(r_0>0\), the same theorem gives \(H_G(R)\geq H_G(r_0)(R/r_0)^{2\beta}\) for every \(R\geq r_0\), contradicting the boundedness of \(H_G\). Therefore \(N_G(R)=0\), and hence \(D_G(R)=0\), for every \(R>0\). It follows that \(DG=0\) almost everywhere, so \(G=Dv\) is constant.

Fix a labelling \(j_1v(0)=\sum_{\ell=1}^q\llbracket(a_\ell,A_\ell)\rrbracket\). Applying Lemma~\ref{lem:one-dimensional-jet-selection} along the segment joining \(0\) to an arbitrary point \(x\), the continuous differential labels remain equal to the fixed matrices \(A_\ell\). Integration therefore gives \(v(x)=\sum_{\ell=1}^q\llbracket a_\ell+A_\ell x\rrbracket\). Thus \(v\) is affine.
\end{proof}

\begin{remark}[An alternative conclusion in dimension two]
\label{rem:planar-liouville-alternative}
When \(n=2\), the preceding proof can be concluded without using Theorem~\ref{thm:stationarity-of-differential}. Indeed, once the boundedness of \(Dv\) has been established, \(|Dv|^2\) is a bounded-above subharmonic function on \(\mathbb R^2\), and is therefore constant by the classical Liouville theorem for subharmonic functions in \(\mathbb R^2\). The Bochner inequality then gives \(D^2v=0\) almost everywhere, so \(Dv\) is constant and \(v\) is affine.
\end{remark}

\section{\(Q\)-valued quasilinear systems and Schauder estimates}

\subsection{$C^{1,\alpha}$ Schauder estimates}
\label{sec:first-order-schauder}

We start by introducing the weak divergence-form formulation used in the
first-order Schauder theory.

\begin{definition}[\(Q\)-valued weak solutions]
\label{def:Q-valued-weak-solution}
Let $\mathscr E: \Omega\times\mathbb R^k\times\operatorname{Hom}(\mathbb R^n,\mathbb R^k)\to\operatorname{Hom}(\mathbb R^n,\mathbb R^k)$ and $\mathscr B:\Omega\times\mathbb R^k\times\operatorname{Hom}(\mathbb R^n,\mathbb R^k)\to\mathbb R^k$ be continuous, and let $f\in L^\infty_{\mathrm{loc}}(\Omega,\mathbb R^k)$. Let $u\in C^1\bigl(\Omega,\mathcal A_Q(\mathbb R^k)\bigr)$. We say that \(u\) is a \(Q\)-valued weak solution of
\begin{equation}
\label{eq:perturbative-system}
\Delta u+\operatorname{div}_x\big(\mathscr E(x,u,Du)\big)+\mathscr B(x,u,Du)=f
\end{equation}
if, for every $\Psi\in C_c^1\bigl(\Omega\times\mathbb R^k,\mathbb R^k\bigr)$, one has
\begin{equation}
\label{eq:Q-valued-weak-formulation}
-\int_\Omega\sum_{\ell=1}^Q\left(Du_\ell+\mathscr E(x,u_\ell,Du_\ell)\right):D\bigl[\Psi(x,u_\ell)\bigr]
+
\int_\Omega\sum_{\ell=1}^Q\mathscr B(x,u_\ell,Du_\ell)\cdot\Psi(x,u_\ell)
=
\int_\Omega\sum_{\ell=1}^Q f(x)\cdot\Psi(x,u_\ell),
\end{equation}
where $j_1u=\sum_{\ell=1}^Q\llbracket(u_\ell,Du_\ell)\rrbracket$ is any measurable selection of the first jet and $D\bigl[\Psi(x,u_\ell)\bigr]=D_x\Psi(x,u_\ell)+D_y\Psi(x,u_\ell)Du_\ell$ denotes the total differential with respect to \(x\).
\end{definition}

\begin{remark}[Basic properties and examples]
\label{rem:Q-valued-weak-solutions-examples}

\begin{enumerate}[label=\textup{(\roman*)}]
    \item The expressions in \eqref{eq:Q-valued-weak-formulation} are symmetric in the first-jet sheets. The definition is therefore independent of the measurable selection used to represent \(j_1u\). The forcing term \(f\) is single-valued and acts identically on every sheet.

    \item When \(Q=1\), Definition~\ref{def:Q-valued-weak-solution} is clearly equivalent to the usual weak formulation of \eqref{eq:perturbative-system}.

    \item Stationarity with respect to outer variations for the Dirichlet energy corresponds to $\mathscr E=0$, $\mathscr B=0$, and $f=0$.

    More generally, let $G: \Omega\times\mathbb R^k \times\operatorname{Hom}(\mathbb R^n,\mathbb R^k) \longrightarrow\mathbb R$ be a \(C^1\) function, and consider the functional $u\mapsto \int_\Omega \sum_{\ell=1}^Q \left\{ \frac12|Du_\ell|^2 + G(x,u_\ell,Du_\ell) \right\}$.
    Stationarity with respect to target variations is equivalent to \eqref{eq:Q-valued-weak-formulation} with $\mathscr E=D_PG$, $\mathscr B=-D_yG$, and $f=0$.

    In particular, let $F_{\mathrm{area}}(P) := \sqrt{\det(I+P^TP)}$ and set $\mathscr E_{\mathrm{area}}(P) := DF_{\mathrm{area}}(P)-P$.
    If the graph varifold associated with \(u\) is stationary, then \(u\) is a \(Q\)-valued weak solution of
    \[
    \Delta u + \operatorname{div}_x \big( \mathscr E_{\mathrm{area}}(Du)\big) =0.
    \]
    Moreover, $D_P\mathscr E_{\mathrm{area}}(0) = D^2F_{\mathrm{area}}(0)-\operatorname{Id} = 0$, and hence $\sup_{|P|\leq L} |D_P\mathscr E_{\mathrm{area}}(P)| \longrightarrow0$ as $L\downarrow0$.
    Thus the small-slope area system is a small perturbation of the Dirichlet system in the sense used below.

    \item In the symmetric \(2\)-valued case, \(Du_2=-Du_1\), so there is a single-valued symmetric matrix \(a(x)\in \R^{n\times n}\) such that \(\mathscr E_{\mathrm{area}}^{i}(Du_\ell(x))= \sum_{j=1}^n a^{ij}(x) D_ju_\ell(x)\) for \(\ell=1,2\) and \(i=1,\ldots,n\), where \(\mathscr E_{\mathrm{area}}^i(P):=\mathscr E_{\mathrm{area}}(P)e_i\). Thus the preceding equation becomes the linear system
    \[
    \Delta u+\sum_{i,j=1}^n D_i\bigl(a^{ij}(x)D_j u\bigr)=0.
    \]
    More generally, if \(u\in C^1(\Omega,\mathcal A_2(\R^k))\) is a solution of the area system but does not necessarily have zero average, the function \(u_f:=u-\eta\circ u\) still satisfies a linear system of the same form $\Delta u_f+\sum_{i,j=1}^n D_i\bigl(b^{ij}(x)D_j u_f\bigr)=0$, with single-valued matrix coefficients \(b^{ij}(x)\in\R^{k\times k}\).
    This is why Simon and Wickramasekera in \cite{SW16} only need to consider linear systems for their \(2\)-valued Schauder estimates.
    For general \(Q\), the coefficients obtained by linearisation need not be common to all sheets, so one needs to consider a genuine quasilinear theory to include the area functional.
\end{enumerate}
\end{remark}

\begin{lemma}[Localisation to differential blocks]
\label{lem:weak-solution-differential-blocks}
Let \(U\subset\mathbb R^n\) be a neighbourhood of \(x_0\), and let $u=\sum_{\alpha=1}^N u^\alpha$, $u^\alpha\in C^1\bigl(U,\mathcal A_{Q_\alpha}(\mathbb R^k)\bigr)$, $\sum_{\alpha=1}^N Q_\alpha=Q$, be a \(C^1\) decomposition. Suppose that, for some \(a\in\mathbb R^k\) and pairwise distinct $A^1,\ldots,A^N \in \operatorname{Hom}(\mathbb R^n,\mathbb R^k)$, one has $j_1u^\alpha(x_0) = Q_\alpha\llbracket(a,A^\alpha)\rrbracket$ for every \(\alpha=1,\ldots,N\).

If \(u\) is a \(Q\)-valued weak solution of
\eqref{eq:perturbative-system}, then, after possibly shrinking \(U\) around \(x_0\), every \(u^\alpha\) is a \(Q_\alpha\)-valued weak solution of the same system.
\end{lemma}

\begin{proof}
We argue by induction on \(Q\). The conclusion is immediate if
\(Q=1\) or \(N=1\). Suppose that \(N\geq2\) and that the result has been proved at every multiplicity strictly smaller than \(Q\).

For each \(\alpha\), set \(m^\alpha:=\eta\circ u^\alpha\). Choose \(\varrho>0\) so that the closed balls \(\overline B_\varrho(A^\alpha)\subset\operatorname{Hom}(\mathbb R^n,\mathbb R^k)\) are pairwise disjoint. Since \(Du^\alpha(x_0)=Q_\alpha\llbracket A^\alpha\rrbracket\) and \(Du^\alpha\) is continuous, after possibly shrinking \(U\) we may assume that \(\operatorname{spt}Du^\alpha(x)\subset B_\varrho(A^\alpha)\) for every \(x\in U\) and every \(\alpha\). In particular, \(Dm^\alpha(x)\in B_\varrho(A^\alpha)\) for every \(x\in U\), and hence \(Dm^\alpha(x)\neq Dm^\beta(x)\) whenever \(\alpha\neq\beta\).

Choose two distinct blocks \(u^\alpha\) and \(u^\beta\). Since all the values of \(u\) coincide at \(x_0\), $m^\alpha(x_0)=m^\beta(x_0)$, whereas $Dm^\alpha(x_0)=A^\alpha \neq A^\beta=Dm^\beta(x_0)$.
Hence there exist \(e\in\mathbb R^k\) and \(\nu\in\mathbb R^n\) such that $e\cdot(A^\alpha-A^\beta)\nu\neq0$.
Set $g(x):= e\cdot\bigl(m^\alpha(x)-m^\beta(x)\bigr)$.
Then $g(x_0)=0$, $D_\nu g(x_0)\neq0$.
After shrinking \(U\) again, we may assume that \(Dg\neq0\) throughout \(U\). Thus $H:=\{x\in U:g(x)=0\}$ is a \(C^1\) hypersurface.

As $\left\{ x\in U: u(x)=Q\llbracket b\rrbracket \text{ for some }b\in\mathbb R^k \right\}\subset H$, at every point of \(U\setminus H\), the values of \(u\) split locally into at least two separated clusters. Localising the test variations in the target shows that each value cluster is a weak solution of the same system. Each such cluster has multiplicity strictly smaller than \(Q\). By the preceding separation of the differential sheets, every differential sub-block of a value cluster belongs to a unique prescribed component \(u^\gamma\); more precisely, \(u^\gamma\) is locally the sum of those differential sub-blocks whose differential sheets lie in \(B_\varrho(A^\gamma)\). The inductive hypothesis applied to the differential sub-blocks of each value cluster, together with the additivity of the weak formulation over the sheets, therefore shows that every prescribed block \(u^\gamma\) is a weak solution in \(U\setminus H\).

It remains to extend the equation across \(H\). Fix $\gamma\in\{1,\ldots,N\}$,
and let \(W^\pm\) be the two sides of \(H\) in a sufficiently small neighbourhood $W\Subset U$.
Let \(\rho\in C^1(W)\) be a defining function for \(H\), with $W^\pm=\{\pm\rho>0\}$, and set $\eta_\varepsilon^\pm(x) := \chi\left(\frac{\pm\rho(x)}{\varepsilon}\right)$, where \(\chi\in C^\infty(\mathbb R)\) is nondecreasing, $\chi=0$ on $\left(-\infty,\frac12\right]$, $\chi=1$ on $[1,\infty)$.

Fix $\Psi\in C_c^1\bigl(W\times\mathbb R^k,\mathbb R^k\bigr)$.
Since \(\eta_\varepsilon^\pm\) vanishes near \(H\), the map $(x,y)\longmapsto \eta_\varepsilon^\pm(x)\Psi(x,y)$ is an admissible test map for \(u^\gamma\) over \(W^\pm\). Writing $j_1u^\gamma = \sum_{\ell=1}^{Q_\gamma} \llbracket
(u_\ell^\gamma,Du_\ell^\gamma) \rrbracket$, we obtain
\[
\begin{aligned}
0&=
\int_{W^\pm}
\eta_\varepsilon^\pm
\sum_{\ell=1}^{Q_\gamma}
\Bigl\{
\left(
Du_\ell^\gamma
+
\mathscr E(x,u_\ell^\gamma,Du_\ell^\gamma)
\right)
:
D\bigl[\Psi(x,u_\ell^\gamma)\bigr]
-
\mathscr B(x,u_\ell^\gamma,Du_\ell^\gamma)
\cdot
\Psi(x,u_\ell^\gamma)
+
f(x)\cdot\Psi(x,u_\ell^\gamma)
\Bigr\}
\\
&+
\int_{W^\pm}
\sum_{\ell=1}^{Q_\gamma}
\left(
Du_\ell^\gamma
+
\mathscr E(x,u_\ell^\gamma,Du_\ell^\gamma)
\right)
:
\bigl(
\Psi(x,u_\ell^\gamma)
\otimes D\eta_\varepsilon^\pm
\bigr).
\end{aligned}
\]
The first integral converges by dominated convergence.

Define the single-valued vector field
\[
\mathcal J_\gamma(x)
:=
\sum_{\ell=1}^{Q_\gamma}
\bigg(
Du_\ell^\gamma(x)
+
\mathscr E
\bigl(
x,u_\ell^\gamma(x),Du_\ell^\gamma(x)
\bigr)
\bigg)^*
\Psi\bigl(x,u_\ell^\gamma(x)\bigr)
\in\mathbb R^n.
\]
The expression defining \(\mathcal J_\gamma\) is symmetric in the first-jet sheets. Since \(j_1u^\gamma\) is continuous and \(\mathscr E\) is continuous, \(\mathcal J_\gamma\) is continuous. Moreover, the last integral above is precisely $\int_{W^\pm} \mathcal J_\gamma\cdot D\eta_\varepsilon^\pm$.
Since \(D\rho\neq0\) near \(H\), the nearby level sets of \(\rho\) are
\(C^1\) hypersurfaces. By the coarea formula,
\[
\int_{W^\pm}
\mathcal J_\gamma\cdot D\eta_\varepsilon^\pm
=
\int_{W^\pm}
\frac1\varepsilon
\chi'\left(\frac{\pm\rho}{\varepsilon}\right)
\mathcal J_\gamma\cdot D(\pm\rho)
=
\int_{1/2}^1
\chi'(t)
\int_{\{\pm\rho=\varepsilon t\}}
\mathcal J_\gamma\cdot
\frac{D(\pm\rho)}{|D\rho|}
\,d\mathcal H^{n-1}\,dt.
\]
Since \(\mathcal J_\gamma\) is continuous and compactly supported in \(W\), the \(C^1\) convergence of the level sets \(\{\pm\rho=\varepsilon t\}\) to \(H\), together with dominated convergence in \(t\), gives $\int_{W^\pm} \mathcal J_\gamma\cdot D\eta_\varepsilon^\pm \longrightarrow - \int_H \mathcal J_\gamma\cdot \nu^{\pm} \,d\mathcal H^{n-1}$, where we also used that $\int_{1/2}^1\chi'(t)\,dt=1$, and that, on \(H\), the vector \(D(\pm\rho)/|D\rho|\) points into \(W^\pm\), and is therefore equal to \(-\nu^\pm\), where \(\nu^\pm\) is the outward-pointing unit normal to \(W^\pm\).

Since $\nu^-=-\nu^+$, the two boundary contributions cancel. Summing the identities on \(W^+\) and \(W^-\) and letting \(\varepsilon\downarrow0\), we obtain
\[
0=
\int_W
\sum_{\ell=1}^{Q_\gamma}
\left(
Du_\ell^\gamma
+
\mathscr E(x,u_\ell^\gamma,Du_\ell^\gamma)
\right)
:
D\bigl[\Psi(x,u_\ell^\gamma)\bigr]
-
\int_W
\sum_{\ell=1}^{Q_\gamma}
\mathscr B(x,u_\ell^\gamma,Du_\ell^\gamma)
\cdot
\Psi(x,u_\ell^\gamma)
+
\int_W
\sum_{\ell=1}^{Q_\gamma}
f(x)\cdot
\Psi(x,u_\ell^\gamma).
\]
Thus \(u^\gamma\) is a \(Q_\gamma\)-valued weak solution across \(H\).
Since the equation is local, \(u^\gamma\) is a weak solution throughout \(U\). Since \(\gamma\) was arbitrary, the proof is complete.
\end{proof}

For the remainder of this subsection, hypothesis \((\star)\) in a ball $B$ denotes the following conditions on a pair \((\mathscr E,\mathscr B)\) as in Definition~\ref{def:Q-valued-weak-solution}, with parameters \(0<\beta\leq1\), \(L,\Lambda<\infty\), and \(\varepsilon>0\):
\[
(\star)\qquad
\left\{
\begin{aligned}
&\mathscr E,\mathscr B\text{ are continuous},\\
&\mathscr E(x,a,0)=0,\\
&\left|\mathscr E(x,a,P)-\mathscr E(x',a',P)\right|\leq\Lambda\left(|x-x'|^\beta+|a-a'|\right)|P|,\\
&\left|\mathscr E(x,a,P)-\mathscr E(x,a,P')\right|\leq\varepsilon|P-P'|,\\
&\mathscr B(x,0,0)=0,\\
&\left|\mathscr B(x,a,P)-\mathscr B(x,a',P')\right|\leq\Lambda\left(|a-a'|+|P-P'|\right),
\end{aligned}
\right.
\]
for every \(x,x'\in B\), \(a,a'\in\mathbb R^k\), and \(|P|,|P'|\leq L\).

We next isolate the linearisation step for the normalized blow-up sequence arising in the proof of the Schauder estimate.

\begin{lemma}[Perturbative linearised blow-ups]
\label{lem:perturbative-linearised-blow-ups}
Let \(q\geq1\), \(0<\alpha\leq\beta\leq1\), \(L,\Lambda<\infty\), and \(R>0\).
Let \(r_h\downarrow0\). For every \(h\), let \((\mathscr E_h,\mathscr B_h)\) satisfy hypothesis \((\star)\) on \(B_{r_hR}\) with parameters \(\beta,L,\Lambda,\varepsilon_h\), where $\varepsilon_h\longrightarrow0$.

Let $f_h\in L^\infty(B_{r_hR},\mathbb R^k)$, and let $u_h\in C^1\bigl(B_{r_hR},\mathcal A_q(\mathbb R^k)\bigr)$ be a \(q\)-valued weak solution of
\[
\Delta u_h
+
\operatorname{div}_x
\mathscr E_h(x,u_h,Du_h)
+
\mathscr B_h(x,u_h,Du_h)
=
f_h
\]
in \(B_{r_hR}\), with $\sup_{B_{r_hR}}|Du_h|\leq L$.

Let $a_h\in \R^k$ and $P_h\in \operatorname{Hom}(\R^n, \R^k)$ satisfy $|P_h|\leq L$, and define $\ell_h(x):=a_h+P_hx$. Let \(H_h>0\) satisfy
$$
\frac{|a_h|+|P_h|}{H_h}\longrightarrow0 \qquad \text{and}\qquad
\frac{r_h^{1-\alpha}\|f_h\|_{L^\infty(B_{r_hR})}}{H_h}
\longrightarrow0,
$$
and
\[
w_h(z)
:=
\frac{u_h(r_hz)\ominus\ell_h(r_hz)}{r_h^{1+\alpha}H_h}
\longrightarrow w
\qquad
\text{locally in }C^1(B_R)
\]
for some $w\in C^1\bigl(B_R,\mathcal A_q(\mathbb R^k)\bigr)$.
Then \(w\) is Dirichlet-stationary in \(B_R\).
\end{lemma}

\begin{proof}
Set $v_h(z):=u_h(r_hz)$.
By the definition of \(v_h\), $Dv_h(z)=r_hDu_h(r_hz)$, and a change of variables in the weak formulation shows that \(v_h\) is a \(q\)-valued weak solution of
\[
\Delta v_h
+
r_h\operatorname{div}_z
\mathscr E_h
\left(
r_hz,v_h,\frac{Dv_h}{r_h}
\right)
+
r_h^2
\mathscr B_h
\left(
r_hz,v_h,\frac{Dv_h}{r_h}
\right)
=
r_h^2f_h(r_hz)
\]
in \(B_R\). Moreover, $\sup_{B_R}|Dv_h|/r_h\leq\sup_{B_{r_hR}}|Du_h|\leq L$.

For a measurable selection \(j_1w_h=\sum_{\ell=1}^q\llbracket(w_{h,\ell},Dw_{h,\ell})\rrbracket\), let \(j_1v_h=\sum_{\ell=1}^q\llbracket(v_{h,\ell},Dv_{h,\ell})\rrbracket\) be the corresponding selection of \(j_1v_h\). Then
\begin{equation}
\label{eq:rescaled-first-jet-identities}
v_{h,\ell}=\ell_h(r_hz)+r_h^{1+\alpha}H_hw_{h,\ell}\qquad\text{and}\qquad\frac{Dv_{h,\ell}}{r_h}=P_h+r_h^\alpha H_hDw_{h,\ell}.
\end{equation}

Let $\Psi\in C_c^1(B_R\times\mathbb R^k,\mathbb R^k)$ and define
\[
\Phi_h(z,a)
:=
\Psi\left(
z,
\frac{a-\ell_h(r_hz)}{r_h^{1+\alpha}H_h}
\right).
\]
Then $\Phi_h(z,v_{h,\ell})=\Psi(z,w_{h,\ell})$ and $D\bigl[\Phi_h(z,v_{h,\ell})\bigr]=D\bigl[\Psi(z,w_{h,\ell})\bigr]$.
Using \(\Phi_h\) in the weak formulation for \(v_h\), dividing by
\(r_h^{1+\alpha}H_h\), and observing that $\int_{B_R}r_hP_h:D\left[\sum_{\ell=1}^q\Psi(z,w_{h,\ell})\right]=0$, we obtain
\begin{equation}
\label{eq:PDE-tested}
\begin{aligned}
0={}&
\int_{B_R}
\sum_{\ell=1}^q
Dw_{h,\ell}:
D\bigl[\Psi(z,w_{h,\ell})\bigr]
+
\frac1{r_h^\alpha H_h}
\int_{B_R}
\sum_{\ell=1}^q
\mathscr E_h
\bigl(
r_hz,
v_{h,\ell},
\frac{Dv_{h,\ell}}{r_h}
\bigr)
:
D\bigl[\Psi(z,w_{h,\ell})\bigr]
\\
&-
\frac{r_h}{r_h^\alpha H_h}
\int_{B_R}
\sum_{\ell=1}^q
\mathscr B_h
\bigl(
r_hz,
v_{h,\ell},
\frac{Dv_{h,\ell}}{r_h}
\bigr)
\cdot
\Psi(z,w_{h,\ell})
+
\frac{r_h}{r_h^\alpha H_h}
\int_{B_R}
\sum_{\ell=1}^q
f_h(r_hz)
\cdot
\Psi(z,w_{h,\ell}).
\end{aligned}
\end{equation}

We show that the last three terms tend to zero. Let \(z\in B_R\). By \eqref{eq:rescaled-first-jet-identities} and the conditions on $\mathscr E_h$ in hypothesis $(\star)$, we get
\[
\begin{aligned}
&\frac1{r_h^\alpha H_h}
\left|
\mathscr E_h
\bigl(
r_hz,
v_{h,\ell},
\frac{Dv_{h,\ell}}{r_h}
\bigr)
-
\mathscr E_h
\bigl(
0,a_h,P_h
\bigr)
\right|
\\
&\leq
\frac{1}{r_h^\alpha H_h}
\left|
\mathscr E_h
\left(
r_hz,
v_{h,\ell},
\frac{Dv_{h,\ell}}{r_h}
\right)
-
\mathscr E_h
\left(
r_hz,
v_{h,\ell},
P_h
\right)
\right|+
\frac{1}{r_h^\alpha H_h}
\left|
\mathscr E_h
\left(
r_hz,
v_{h,\ell},
P_h
\right)
-
\mathscr E_h
\left(
0,
a_h,
P_h
\right)
\right|
\\
&\leq
\varepsilon_h|Dw_{h,\ell}|
+
\Lambda\left[
R^\beta r_h^{\beta-\alpha}\frac{|P_h|}{H_h}
+
r_hL|w_{h,\ell}|
+
Rr_h^{1-\alpha}L\frac{|P_h|}{H_h}
\right].
\end{aligned}
\]
The right-hand side converges to zero locally uniformly, since
\(\varepsilon_h\to0\), \((|a_h|+|P_h|)/H_h\to0\), \(r_h\to0\),
\(\beta\geq\alpha\), and \(w_h\to w\) locally in \(C^1(B_R)\).
Since $\mathscr E_h(0,a_h,P_h)$ is constant in \(z\), its contraction with $D\left[\sum_{\ell=1}^q\Psi(z,w_{h,\ell})\right]$ integrates to zero. Thus the term of \eqref{eq:PDE-tested} involving \(\mathscr E_h\) converges to zero.

Similarly, for the lower-order term, \eqref{eq:rescaled-first-jet-identities} and the conditions on \(\mathscr B_h\) in hypothesis $(\star)$ give, on \(B_R\),
\[
\frac{r_h}{r_h^\alpha H_h}
\left|
\mathscr B_h
\bigl(
r_hz,
v_{h,\ell},
\frac{Dv_{h,\ell}}{r_h}
\bigr)
\right|
\leq
\Lambda\left[
r_h^2|w_{h,\ell}|
+
r_h^{1-\alpha}\frac{|a_h|}{H_h}
+
Rr_h^{2-\alpha}\frac{|P_h|}{H_h}
+
r_h|Dw_{h,\ell}|
+
r_h^{1-\alpha}\frac{|P_h|}{H_h}
\right].
\]
The right-hand side converges to zero locally uniformly, since
\(r_h\to0\), \((|a_h|+|P_h|)/H_h\to0\), \(0<\alpha\leq1\), \(R\) is fixed, and
\(w_h\to w\) locally in \(C^1(B_R)\).
Hence the term of \eqref{eq:PDE-tested} involving \(\mathscr B_h\) also converges to zero.

Finally,
\[
\left|
\frac{r_h}{r_h^\alpha H_h}
\int_{B_R}
\sum_{\ell=1}^q
f_h(r_hz)\cdot\Psi(z,w_{h,\ell})
\right|
\leq
C
\frac{r_h^{1-\alpha}\|f_h\|_{L^\infty(B_{r_hR})}}{H_h}
\longrightarrow0.
\]
Thus the term involving \(f_h\) also converges to zero.

Letting \(h\to\infty\), and using the local \(C^1\)-convergence of \(w_h\) to \(w\), we obtain
\[
\int_{B_R}
\sum_{\ell=1}^q
Dw_\ell:
D\bigl[\Psi(z,w_\ell)\bigr]
=0.
\]
Thus \(w\) is stationary with respect to outer variations for the Dirichlet energy in \(B_R\).

Since \(w\in C^1\), Proposition~\ref{prop:inductive-bochner} gives $w\in W^{2,2}_{\mathrm{loc}}\bigl(B_R,\mathcal A_q(\mathbb R^k)\bigr)$, and Theorem~\ref{thm:stationarity-of-differential} implies that \(w\) is Dirichlet-stationary in \(B_R\).
\end{proof}

\begin{theorem}[\(C^{1,\alpha}\) Schauder estimates for \(Q\)-valued quasilinear systems]
\label{thm:quasilinear-Schauder}
Let \(0<\alpha<1/Q\) if \(n=2\), and \(0<\alpha<\delta(n,k,Q)\) if \(n\geq3\). Let \(\alpha\leq\beta\leq1\) and \(L,\Lambda<\infty\).
There exist $\varepsilon_0 = \varepsilon_0(n,k,Q,\alpha,\beta,L,\Lambda)>0$
and $C =C(n,k,Q,\alpha,\beta,L,\Lambda)<\infty$ with the following property.

Let \((\mathscr E,\mathscr B)\) satisfy hypothesis \((\star)\) on \(B^n_1\) with parameters $\beta,L,\Lambda,\varepsilon_0$, let $f\in L^\infty(B^n_1,\mathbb R^k)$, and let $u\in C^{1,\alpha}\bigl(B^n_1,\mathcal A_Q(\mathbb R^k)\bigr)$ be a \(Q\)-valued weak solution of
\[
\Delta u
+
\operatorname{div}_x\mathscr E(x,u,Du)
+
\mathscr B(x,u,Du)
=
f
\]
in \(B^n_1\), with $\sup_{B^n_1}|Du|\leq L$.
Then
\[
\|u\|_{C^{1,\alpha}(B_{1/2}^n)}
\leq
C\left(
\|u\|_{L^2(B_1^n)}
+
\|f\|_{L^\infty(B_1^n)}
\right).
\]
\end{theorem}

\begin{proof}

\smallskip
\noindent
\emph{\underline{Step 1: reduction to a simpler estimate.}}
We first note that hypothesis \((\star)\) is preserved under graphical rescalings. Indeed, if \(0<r\leq1\), \((\mathscr E,\mathscr B)\) satisfies hypothesis \((\star)\) on \(B_r(x_0)\), and \(u\) solves the equation there with right-hand side \(f\), then
\[
u_{x_0,r}(z):=\frac{u(x_0+rz)}{r}
\]
solves an equation of the same form in \(B_1\), with
\[
\mathscr E_r(z,a,P):=\mathscr E(x_0+rz,ra,P),
\qquad
\mathscr B_r(z,a,P):=r\,\mathscr B(x_0+rz,ra,P),
\qquad
f_r(z):=r f(x_0+rz).
\]
Moreover, \((\mathscr E_r,\mathscr B_r)\) satisfies hypothesis \((\star)\) on \(B_1\) with the same parameters, and
\[
\|f_r\|_{L^\infty(B_1)}
\leq
r\|f\|_{L^\infty(B_r(x_0))}.
\]
Hence, by the \(Q\)-valued versions of Simon's interpolation inequality \cite[(1.5)]{Sim97} and the absorbing lemma in \cite[Section~4]{Sim97}, which follow by the same compactness and covering arguments as in the single-valued case, with the required \(Q\)-valued Arzelà--Ascoli compactness provided by Lemma~\ref{lem:C1-compactness}, it is enough to show that, for every \(\delta>0\), there exist \(\varepsilon_\delta>0\) and \(C_\delta<\infty\) such that
\begin{equation}
\label{eq:reduction-estimate}
[Du]_{\alpha;B_{1/2}}
\leq
\delta[Du]_{\alpha;B_1}
+
C_\delta\left(\|u\|_{C^1(B_1)} +\|f\|_{L^\infty(B_1)}\right)
\end{equation}
whenever \((\mathscr E,\mathscr B)\) satisfies hypothesis \((\star)\) on $B_1$ with parameters \(\beta,L,\Lambda,\varepsilon_\delta\), \(f\in L^\infty(B_1,\mathbb R^k)\), and $u\in C^{1,\alpha}\bigl(B_1,\mathcal A_Q(\mathbb R^k)\bigr)$ is a \(Q\)-valued weak solution of
\[
\Delta u+\operatorname{div}_x\mathscr E(x,u,Du)+\mathscr B(x,u,Du)=f
\]
in \(B_1\), with \(\sup_{B_1}|Du|\leq L\).

\smallskip
\noindent
\emph{\underline{Step 2: contradiction sequence and rescaling.}}
Suppose that \eqref{eq:reduction-estimate} fails for some fixed \(\delta>0\). Then there are pairs \((\mathscr E_h,\mathscr B_h)\) satisfying hypothesis \((\star)\) on \(B_1\) with parameters \(\beta,L,\Lambda,h^{-1}\), functions \(f_h\in L^\infty(B_1,\mathbb R^k)\), and \(Q\)-valued weak solutions $u_h\in C^{1,\alpha}\bigl(B_1,\mathcal A_Q(\mathbb R^k)\bigr)$ of
\[
\Delta u_h+\operatorname{div}_x\mathscr E_h(x,u_h,Du_h)+\mathscr B_h(x,u_h,Du_h)=f_h
\]
in \(B_1\), with \(\sup_{B_1}|Du_h|\leq L\), such that
\[
[Du_h]_{\alpha;B_{1/2}}
>
\delta[Du_h]_{\alpha;B_1}
+
h\left(\|u_h\|_{C^1(B_1)} +\|f_h\|_{L^\infty(B_1)}\right).
\]
In particular,
\begin{equation}
\label{eq:contradiction-ratios}
\frac{[Du_h]_{\alpha;B_1}}{[Du_h]_{\alpha;B_{1/2}}}\leq\delta^{-1},
\qquad
\frac{\|u_h\|_{C^1(B_1)} +\|f_h\|_{L^\infty(B_1)}}{[Du_h]_{\alpha;B_{1/2}}}\longrightarrow0.
\end{equation}

Choose \(x_h,y_h\in B_{1/2}\) such that, writing \(\rho_h:=|x_h-y_h|\), \(\xi_h:=\frac{y_h-x_h}{\rho_h}\), we have
\begin{equation}
\label{eq:lower-bound-Holder}
\mathcal G\bigl(Du_h(x_h),Du_h(y_h)\bigr)
\geq
\frac12[Du_h]_{\alpha;B_{1/2}}\rho_h^\alpha.
\end{equation}
Since
\[
\rho_h^\alpha
\leq
\frac{4\sup_{B_1}|Du_h|}{[Du_h]_{\alpha;B_{1/2}}}
\leq
4\frac{\|u_h\|_{C^1(B_1)} +\|f_h\|_{L^\infty(B_1)}}{[Du_h]_{\alpha;B_{1/2}}}
\longrightarrow0,
\]
we have \(\rho_h\to0\). After passing to a subsequence, we may assume that \(\xi_h\to\xi\in\mathbb S^{n-1}\).

Set $\Omega_h:=\{z\in\mathbb R^n:x_h+\rho_hz\in B_1\}(=B_{1/\rho_h}\left(-x_h/\rho_h\right))$ and define
\[
\widetilde w_h(z)
:=
\frac{u_h(x_h+\rho_hz)}
{\rho_h^{1+\alpha}[Du_h]_{\alpha;B_{1/2}}}.
\]
The domains \(\Omega_h\) contain \(B_{1/(2\rho_h)}\) and therefore exhaust \(\mathbb R^n\). Moreover, by \eqref{eq:contradiction-ratios},
\begin{equation}
\label{eq:rescaled-Holder-bound}
[D\widetilde w_h]_{\alpha;\Omega_h}
=
\frac{[Du_h]_{\alpha;B_1}}{[Du_h]_{\alpha;B_{1/2}}}
\leq\delta^{-1},
\end{equation}
while \eqref{eq:lower-bound-Holder} becomes
\begin{equation}
\label{eq:rescaled-gradient-separation}
\mathcal G\bigl(D\widetilde w_h(0),D\widetilde w_h(\xi_h)\bigr)\geq\frac12.
\end{equation}

\smallskip
\noindent
\emph{\underline{Step 3: first-jet clustering.}}
Write $j_1\widetilde w_h(0) = \sum_{i=1}^Q\llbracket(a_{h,i},A_{h,i})\rrbracket$.
Since there are only finitely many pairs of indices, after passing to a subsequence we may assume that $\bigl|(a_{h,i},A_{h,i})-(a_{h,j},A_{h,j})\bigr|$ has a limit in \([0,\infty]\) for every \(i,j\). Declare
\[
i\sim j
\qquad\Longleftrightarrow\qquad
\lim_h\bigl|(a_{h,i},A_{h,i})-(a_{h,j},A_{h,j})\bigr|<\infty.
\]
This is an equivalence relation; denote its equivalence classes by
\[
\{1,\ldots,Q\}=I_1\cup\cdots\cup I_N,
\qquad
q_m:=|I_m|.
\]
Thus the jets belonging to the same class have uniformly bounded mutual distance, whereas, if \(m\neq m'\),
\[
d_h^{m,m'}
:=
\min_{\substack{i\in I_m\\j\in I_{m'}}}
\bigl|(a_{h,i},A_{h,i})-(a_{h,j},A_{h,j})\bigr|
\longrightarrow\infty.
\]

If \(N=1\), we simply set \(\widetilde w_h^{(1)}:=\widetilde w_h\) and choose any \(R_h\to\infty\) such that \(B_{R_h}\subset\Omega_h\).

If \(N\geq2\), set \(d_h:=\min_{m\neq m'}d_h^{m,m'}\). By \eqref{eq:rescaled-Holder-bound}, the seminorms \([D\widetilde w_h]_{\alpha;B_R}\) remain bounded uniformly in \(h\) and \(R\), whenever \(B_R\subset\Omega_h\). Since \(d_h\to\infty\), Corollary~\ref{cor:persistence-jet-clusters} shows that, for every fixed \(R<\infty\) and every sufficiently large \(h\), the prescribed clusters determine a unique decomposition $\widetilde w_h = \sum_{m=1}^N\widetilde w_h^{(m)}$ on $B_R$, where $\widetilde w_h^{(m)} \in C^{1,\alpha}\bigl(B_R,\mathcal A_{q_m}(\mathbb R^k)\bigr)$ and $[D\widetilde w_h^{(m)}]_{\alpha;B_R} \leq C_Q[D\widetilde w_h]_{\alpha;B_R} \leq C_Q\delta^{-1}$. By a diagonal argument and the uniqueness of the decomposition, we may choose \(R_h\to\infty\), with \(B_{R_h}\subset\Omega_h\), so that these decompositions are defined on \(B_{R_h}\).

Let \(u_h^{(m)}\) be the corresponding components of \(u_h\), defined by
\[
u_h^{(m)}(x_h+\rho_hz)
:=
\rho_h^{1+\alpha}[Du_h]_{\alpha;B_{1/2}}
\widetilde w_h^{(m)}(z),
\qquad z\in B_{R_h}.
\]
Then $u_h=\sum_{m=1}^N u_h^{(m)}$ on \(B_{\rho_hR_h}(x_h)\), and each \(u_h^{(m)}\) is a \(q_m\)-valued weak solution of the same system with right-hand side \(f_h\). Indeed, locally one first separates the distinct value clusters and localizes the weak formulation in the target. Within each value cluster, the local first-jet splitting decomposes it into differential blocks, to which Lemma~\ref{lem:weak-solution-differential-blocks} applies. The prescribed component \(u_h^{(m)}\) is locally a sum of such blocks, and hence is itself a weak solution. Since the equation is local, \(u_h^{(m)}\) is therefore a weak solution throughout \(B_{\rho_hR_h}(x_h)\).

\smallskip
\noindent
\emph{\underline{Step 4: extraction of Dirichlet-stationary blow-ups.}}
Set $\tilde a_h^{(m)} := \eta\circ\widetilde w_h^{(m)}(0)$, $\tilde P_h^{(m)} := \eta\circ D\widetilde w_h^{(m)}(0)$, and define
\[
\widehat w_h^{(m)}(z)
:=
\widetilde w_h^{(m)}(z)
\ominus
\bigl(\tilde a_h^{(m)}+\tilde P_h^{(m)}z\bigr).
\]
Since the jets belonging to \(I_m\) have uniformly bounded mutual distance at the origin, $|\widehat w_h^{(m)}(0)| + |D\widehat w_h^{(m)}(0)| \leq C_m$.
Moreover, $[D\widehat w_h^{(m)}]_{\alpha;B_{R_h}} = [D\widetilde w_h^{(m)}]_{\alpha;B_{R_h}} \leq C_Q\delta^{-1}$.
It follows that, for every fixed \(R<\infty\),
\[
\|\widehat w_h^{(m)}\|_{C^{1,\alpha}(B_R)}
\leq C_R
\]
for all sufficiently large \(h\). By Lemma~\ref{lem:C1-compactness}, after passing to a diagonal subsequence,
\[
\widehat w_h^{(m)}
\longrightarrow
w^{(m)}
\qquad\text{locally in }C^1(\mathbb R^n)
\]
for every \(m\), where $[Dw^{(m)}]_{\alpha;\mathbb R^n} \leq C_Q\delta^{-1}$.

We claim that each \(w^{(m)}\) is Dirichlet-stationary. Fix \(R<\infty\). For all sufficiently large \(h\), \(R<R_h\) and \(x_h+B_{\rho_hR}\subset B_1\). Define $u_h^{(m),0}(x):=u_h^{(m)}(x_h+x)$, $\mathscr E_h^0(x,a,P):=\mathscr E_h(x_h+x,a,P)$, $\mathscr B_h^0(x,a,P):=\mathscr B_h(x_h+x,a,P)$, and $f_h^0(x):=f_h(x_h+x)$.
Then \(u_h^{(m),0}\) is a \(q_m\)-valued weak solution on \(B_{\rho_hR}\) with right-hand side \(f_h^0\), and \((\mathscr E_h^0,\mathscr B_h^0)\) satisfies hypothesis \((\star)\) there with parameters \(\beta,L,\Lambda,h^{-1}\).

We apply Lemma~\ref{lem:perturbative-linearised-blow-ups} with $q:=q_m$, $r_h:=\rho_h$, $H_h:=[Du_h]_{\alpha;B_{1/2}}$, $a_h^{(m)}:= \rho_h^{1+\alpha} H_h \tilde a_h^{(m)}$, $P_h^{(m)} := \rho_h^\alpha H_h \tilde P_h^{(m)}$, and $\ell^{(m)}_h(x) := a_h^{(m)} + P_h^{(m)} x$.
By definition, $\ell_h^{(m)}(\rho_hz) = \rho_h^{1+\alpha}H_h \bigl(\tilde a_h^{(m)}+ \tilde P_h^{(m)}z\bigr)$, and hence
\[
\frac{
u_h^{(m),0}(\rho_hz)
\ominus
\ell_h^{(m)}(\rho_hz)
}{
\rho_h^{1+\alpha}H_h
}
=
\widehat w_h^{(m)}(z)
\longrightarrow
w^{(m)}(z)
\qquad
\text{locally in }C^1(B_R).
\]
Moreover, $a_h^{(m)} = \eta\circ u_h^{(m)}(x_h)$, $P_h^{(m)} = \eta\circ Du_h^{(m)}(x_h)$.
Thus
\[
\frac{|a_h^{(m)}|
+
|P_h^{(m)}|}
{H_h}
\leq
\frac{\|u_h\|_{C^1(B_1)}}{H_h}
\longrightarrow0
\]
by \eqref{eq:contradiction-ratios}, while $|P_h^{(m)}|\leq L$.
Moreover, for all sufficiently large \(h\),
\[
\frac{\rho_h^{1-\alpha}\|f_h^0\|_{L^\infty(B_{\rho_hR})}}{H_h}
\leq
\frac{\|f_h\|_{L^\infty(B_1)}}{H_h}
\longrightarrow0
\]
by \eqref{eq:contradiction-ratios}.
Finally, $\sup_{B_{\rho_hR}}|Du_h^{(m),0}|\leq L$.
Lemma~\ref{lem:perturbative-linearised-blow-ups} therefore shows that \(w^{(m)}\) is Dirichlet-stationary in \(B_R\). Since \(R<\infty\) was arbitrary, \(w^{(m)}\) is Dirichlet-stationary in \(\mathbb R^n\).

\smallskip
\noindent
\emph{\underline{Step 5: conclusion via the Liouville theorem.}}
If \(n=2\), then $\alpha<\frac1Q\leq\frac1{q_m}$, while if \(n\geq3\), then $\alpha<\delta(n,k,Q)\leq\delta(n,k,q_m)$ by our convention that \(Q\mapsto\delta(n,k,Q)\) is nonincreasing. Thus Theorem~\ref{thm:liouville} implies that every \(w^{(m)}\) is affine. On the other hand, concatenating optimal matchings within the individual clusters gives
\[
\mathcal G\bigl(D\widetilde w_h(0),D\widetilde w_h(\xi_h)\bigr)^2
\leq
\sum_{m=1}^N
\mathcal G\bigl(D\widetilde w_h^{(m)}(0),D\widetilde w_h^{(m)}(\xi_h)\bigr)^2.
\]
Since the same affine map is subtracted at both points,
\[
\mathcal G\bigl(D\widetilde w_h^{(m)}(0),D\widetilde w_h^{(m)}(\xi_h)\bigr)
=
\mathcal G\bigl(D\widehat w_h^{(m)}(0),D\widehat w_h^{(m)}(\xi_h)\bigr).
\]
The local \(C^1\)-convergence, the fact that \(\xi_h\to\xi\), and the affineness of \(w^{(m)}\) therefore give
\[
\mathcal G\bigl(D\widetilde w_h(0),D\widetilde w_h(\xi_h)\bigr)^2
\longrightarrow0.
\]
This contradicts \eqref{eq:rescaled-gradient-separation}. The interpolation estimate follows, and the proof is complete.
\end{proof}

\subsection{$C^{2,\alpha}$ Schauder estimates}
\label{sec:second-order-schauder}

We now introduce the second-order notation and the basic \(C^2\) theory needed for the strong formulation.
For brevity, set $E_1:=\mathbb R^k\times\operatorname{Hom}(\mathbb R^n,\mathbb R^k)$.

\begin{definition}[\(C^2\) multiple-valued maps and second jets]
\label{def:C2-Q-valued}
We say that \(u\in C^2(\Omega)\) if $u\in C^1\bigl(\Omega,\mathcal A_Q(\mathbb R^k)\bigr)$ and $j_1u\in C^1\bigl(\Omega,\mathcal A_Q(E_1)\bigr)$.
Whenever $j_1u(x) = \sum_{\ell=1}^Q \llbracket(a_\ell,A_\ell)\rrbracket $ and $ j_1(j_1u)(x) = \sum_{\ell=1}^Q \llbracket \bigl( (a_\ell,A_\ell), (B_\ell,S_\ell) \bigr) \rrbracket$, where $B_\ell\in\operatorname{Hom}(\mathbb R^n,\mathbb R^k)$, $S_\ell\in \operatorname{Bil}(\mathbb R^n\times\mathbb R^n,\mathbb R^k)$, we define the second jet and the second differential of \(u\) at \(x\) by
\[
j_2u(x)
:=
\sum_{\ell=1}^Q
\llbracket(a_\ell,A_\ell,S_\ell)\rrbracket,
\qquad
D^2u(x)
:=
\sum_{\ell=1}^Q
\llbracket S_\ell\rrbracket.
\]

We equip the space of second jets with the product norm $|(a,A,S)|^2:=|a|^2+|A|^2+|S|^2$, where \(|A|\) and \(|S|\) denote the corresponding Hilbert--Schmidt norms, and set \(\|u\|_{C^2(\Omega)} := \sup_{\Omega}|u| + \sup_{\Omega}|Du| + \sup_{\Omega}|D^2u|\).
For \(0<\alpha<1\), we say that \(u\in C^{2,\alpha}(\Omega)\) if
\(j_1u\in C^{1,\alpha}\bigl(\Omega,\mathcal A_Q(E_1)\bigr)\). We write
\[
[D^2u]_{\alpha;\Omega}
:=
\sup_{\substack{x,y\in\Omega\\x\neq y}}
\frac{\mathcal G\bigl(D^2u(x),D^2u(y)\bigr)}{|x-y|^\alpha}.
\]
We also set \(\|u\|_{C^{2,\alpha}(\Omega)} := \sup_{\Omega}|u| + \sup_{\Omega}|Du| + \sup_{\Omega}|D^2u| + [D^2u]_{\alpha;\Omega}\).

Finally, we say that \(u_j\to u\) locally in
\(C^2(\Omega,\mathcal A_Q(\mathbb R^k))\) if
\(j_1u_j\to j_1u\) locally in
\(C^1\bigl(\Omega,\mathcal A_Q(E_1)\bigr)\).
\end{definition}

\begin{lemma}[Basic properties of second jets]
\label{lem:second-jet-calculus}
Let \(0<\alpha<1\). The following hold.

\begin{enumerate}[label=\textup{(\roman*)}]
\item
Let \(u\in C^2\bigl(\Omega,\mathcal A_Q(\mathbb R^k)\bigr)\). Then, with the notation of Definition~\ref{def:C2-Q-valued}, $B_\ell=A_\ell$ and \(D^2u(x)\) is symmetric. Thus $j_1(j_1u)(x) = \sum_{\ell=1}^Q \llbracket \bigl( (a_\ell,A_\ell),(A_\ell,S_\ell) \bigr) \rrbracket$ with $S_\ell\in\operatorname{Sym}^2(\mathbb R^n,\mathbb R^k)$.

\item
Let \(u\in C^2\bigl(\Omega,\mathcal A_Q(\mathbb R^k)\bigr)\) and let \(\gamma:I\to\Omega\) be a \(C^1\) curve, where \(I\) is an interval. Then there are $a_i\in C^1(I,\mathbb R^k)$, $A_i\in C^1\bigl(I,\operatorname{Hom}(\mathbb R^n,\mathbb R^k)\bigr)$, $S_i\in C^0\bigl(I,\operatorname{Sym}^2(\mathbb R^n,\mathbb R^k)\bigr)$ such that $j_2u(\gamma(t)) = \sum_{i=1}^Q \llbracket(a_i(t),A_i(t),S_i(t))\rrbracket$ and $\dot a_i(t)=A_i(t)\dot\gamma(t)$, $\dot A_i(t)=S_i(t)\bigl(\dot\gamma(t),\cdot\bigr)$ for every \(t\in I\). If \(j_2u\circ\gamma\) is given as a continuous sum of sub-jets, the selection can be chosen to respect this decomposition.

\item
Let $u_j\in C^{2,\alpha}\bigl(\Omega,\mathcal A_Q(\mathbb R^k)\bigr)$ and suppose that, for every \(\Omega'\Subset\Omega\),
\[
\sup_j\|u_j\|_{C^{2,\alpha}(\Omega')}<\infty.
\]
Then, after passing to a subsequence, there exists $u\in  C^{2,\alpha}_{\mathrm{loc}} \bigl(\Omega,\mathcal A_Q(\mathbb R^k)\bigr)$ such that \(u_j\to u\) locally in \(C^2(\Omega)\).

\item
Let \(u\in C^2\bigl(\Omega,\mathcal A_Q(\mathbb R^k)\bigr)\), and
suppose that, at \(x_0\in\Omega\), $j_2u(x_0) = \sum_{\beta=1}^N Q_\beta \llbracket(a^\beta,A^\beta,S^\beta)\rrbracket$, where the triples \((a^\beta,A^\beta,S^\beta)\) are pairwise distinct and \(\sum_{\beta=1}^NQ_\beta=Q\). Then there are a neighbourhood \(U\) of \(x_0\) and maps $u^\beta\in C^2\bigl(U,\mathcal A_{Q_\beta}(\mathbb R^k)\bigr)$ such that $u=\sum_{\beta=1}^Nu^\beta$, $j_2u^\beta(x_0) = Q_\beta \llbracket(a^\beta,A^\beta,S^\beta)\rrbracket$.
If \(u\in C^{2,\alpha}\), then each \(u^\beta\in C^{2,\alpha}\).

\item
Let \(R<\infty\) and \(u\in C^{2,\alpha}\bigl(B_R,\mathcal A_Q(\mathbb R^k)\bigr)\).
Write $j_2u(0) = \sum_{i=1}^Q\llbracket(a_i,A_i,S_i)\rrbracket$, and let
\[
\{1,\ldots,Q\}=I_1\cup\cdots\cup I_N,
\qquad N\geq2,
\]
be a partition. Set
\[
d
:=
\min_{\substack{\ell\neq m\\i\in I_\ell,\;h\in I_m}}
\bigl|
(a_i,A_i,S_i)-(a_h,A_h,S_h)
\bigr|.
\]
For every \(M<\infty\), there exists $d_0=d_0(Q,\alpha,M,R)>0$ such that, if $[D^2u]_{\alpha;B_R}\leq M$ and $d\geq d_0$, then there are uniquely determined maps $u^\ell\in C^{2,\alpha}\bigl(B_R,\mathcal A_{|I_\ell|}(\mathbb R^k)\bigr)$, $\ell=1,\ldots,N$, such that $u=\sum_{\ell=1}^Nu^\ell$, $j_2u^\ell(0) = \sum_{i\in I_\ell} \llbracket(a_i,A_i,S_i)\rrbracket$,
and $[D^2u^\ell]_{\alpha;B_R} \leq C_Q[D^2u]_{\alpha;B_R}$.
\end{enumerate}
\end{lemma}

\begin{proof}
For \textup{(i)}, Lemma~\ref{lem:elementary-first-jet-calculus}\textup{(iii)}, applied to \(u\) and \(j_1u\), shows that both maps are locally Lipschitz and that their classical first jets agree almost everywhere with their approximate first jets. Thus \(u\in W^{2,2}_{\mathrm{loc}}\bigl(\Omega,\mathcal A_Q(\mathbb R^k)\bigr)\), and Remark~\ref{rem:sobolev-equivalences}\textup{(ii)} gives \(B_\ell=A_\ell\) almost everywhere. Moreover, Lemma~\ref{lem:Hodge-second-jet} gives \(S_\ell\in\operatorname{Sym}^2(\mathbb R^n,\mathbb R^k)\) almost everywhere. Since \(j_1(j_1u)\) is continuous and these conditions define a closed subset of the jet space, both conclusions hold at every point.

For \textup{(ii)}, apply Lemma~\ref{lem:one-dimensional-jet-selection} to the \(C^1\) map \(j_1u\) along \(\gamma\). By \textup{(i)}, the resulting selections give
\[
j_2u(\gamma(t))=\sum_{\ell=1}^Q\llbracket(a_\ell(t),A_\ell(t),S_\ell(t))\rrbracket,\qquad
\dot a_\ell(t)=A_\ell(t)\dot\gamma(t),\qquad
\dot A_\ell(t)=S_\ell(t)(\dot\gamma(t),\cdot),
\]
with \(a_\ell\in C^1(I,\mathbb R^k)\), \(A_\ell\in C^1\bigl(I,\operatorname{Hom}(\mathbb R^n,\mathbb R^k)\bigr)\), and \(S_\ell\in C^0\bigl(I,\operatorname{Sym}^2(\mathbb R^n,\mathbb R^k)\bigr)\). If \(j_2u\circ\gamma\) is given as a continuous sum of sub-jets, then \textup{(i)} induces a corresponding continuous decomposition of \(j_1(j_1u)\circ\gamma\). Applying the same selection lemma with selections respecting this decomposition proves the final assertion.

For \textup{(iii)}, fix \(B_{2r}(x_0)\Subset\Omega\). By assumption there is \(M<\infty\), independent of \(j\), such that \(\|u_j\|_{C^{2,\alpha}(B_{2r}(x_0))}\leq M\).
We show that the \(E_1\)-valued maps \(j_1u_j\) satisfy the hypotheses of Lemma~\ref{lem:C1-compactness}. Let \(x,y\in B_r(x_0)\), set \(h:=y-x\), and apply \textup{(ii)} along \(\gamma(t):=x+th\). We may write \(j_2u_j(\gamma(t))=\sum_{\ell=1}^Q\llbracket(a_\ell(t),A_\ell(t),S_\ell(t))\rrbracket\), with \(\dot A_\ell(t)=S_\ell(t)(h,\cdot)\). The Hölder bound for \(D^2u_j\) and the usual chain-of-balls argument give \(|S_\ell(t)-S_\ell(0)|\leq C_QMt^\alpha|h|^\alpha\), while integration gives \(|A_\ell(t)-A_\ell(0)|\leq Mt|h|\). By \textup{(i)}, \(D(j_1u_j)(\gamma(t))=\sum_{\ell=1}^Q\llbracket(A_\ell(t),S_\ell(t))\rrbracket\), and hence
\[
[D(j_1u_j)]_{\alpha;B_r(x_0)}\leq C(Q,\alpha,r)M.
\]
The assumed bounds also give uniform bounds for \(j_1u_j\) and \(D(j_1u_j)\) on \(B_r(x_0)\). Lemma~\ref{lem:C1-compactness} and a diagonal argument therefore give, after passing to a subsequence, \(j_1u_j\to J\) locally in \(C^1(\Omega)\), for some \(J\in C^{1,\alpha}_{\mathrm{loc}}\bigl(\Omega,\mathcal A_Q(E_1)\bigr)\).

The compatibility in \textup{(i)} is preserved under this convergence. Thus, writing \(J=\sum_{\ell=1}^Q\llbracket(a_\ell,A_\ell)\rrbracket\), we have \(j_1J=\sum_{\ell=1}^Q\llbracket((a_\ell,A_\ell),(A_\ell,S_\ell))\rrbracket\). Let \(\pi:E_1\to\mathbb R^k\) be the projection \(\pi(a,A)=a\), and set \(u:=\pi_\#J\). The chain rule gives
\[
j_1u=\sum_{\ell=1}^Q\llbracket(a_\ell,A_\ell)\rrbracket=J.
\]
Since \(J\in C^{1,\alpha}_{\mathrm{loc}}\), it follows that \(u\in C^{2,\alpha}_{\mathrm{loc}}\). Finally, the convergence \(j_1u_j\to j_1u\) locally in \(C^1\) is precisely the required local \(C^2\)-convergence of \(u_j\) to \(u\).

For \textup{(iv)}, apply Proposition~\ref{prop:local-splitting} to the \(C^1\) map \(j_1u\), with target \(E_1\). By \textup{(i)},
\[
j_1(j_1u)(x_0)=\sum_{\beta=1}^NQ_\beta\llbracket((a^\beta,A^\beta),(A^\beta,S^\beta))\rrbracket,
\]
and these \(N\) atoms are pairwise distinct. We therefore obtain, after shrinking to a neighbourhood \(U\) of \(x_0\), a decomposition \(j_1u=\sum_{\beta=1}^NJ^\beta\), with \(J^\beta\in C^1\bigl(U,\mathcal A_{Q_\beta}(E_1)\bigr)\) and the prescribed first jets at \(x_0\).

By additivity of first jets, \(j_1(j_1u)=\sum_{\beta=1}^Nj_1J^\beta\), so every component \(J^\beta\) inherits the compatibility in \textup{(i)}. Let \(\pi:E_1\to\mathbb R^k\) be the projection \(\pi(a,A)=a\), and set \(u^\beta:=\pi_\#J^\beta\). As in \textup{(iii)}, the chain rule gives \(j_1u^\beta=J^\beta\). Thus \(u=\sum_{\beta=1}^Nu^\beta\), each \(u^\beta\) belongs to \(C^2\), and \(j_2u^\beta(x_0)=Q_\beta\llbracket(a^\beta,A^\beta,S^\beta)\rrbracket\).

If \(u\in C^{2,\alpha}\), shrink \(U\) to a ball. Applying Lemma~\ref{lem:one-dimensional-jet-selection} to \(j_1u\) along segments, with selections respecting the decomposition \(j_1(j_1u)=\sum_{\beta=1}^Nj_1J^\beta\), the usual chain-of-balls argument gives \([DJ^\beta]_{\alpha;U}\leq C_Q[D(j_1u)]_{\alpha;U}\). Hence \(J^\beta\in C^{1,\alpha}\), and therefore \(u^\beta\in C^{2,\alpha}\).

It remains to prove \textup{(v)}. For \(x\in B_R\), let \(K(x)\) be the image of \(j_2u(x)\) under the map
\[
(a,A,S)\longmapsto\left(a-Ax+\frac12S(x,x),A-S(x,\cdot),S\right),
\]
which takes the second jet at \(x\) of a quadratic polynomial to its second jet at the origin. Then \(K\) is continuous and \(K(0)=j_2u(0)\).

Fix \(x\in B_R\). By \textup{(ii)}, we may write
\[
j_2u(tx)=\sum_{\ell=1}^Q\llbracket(a_\ell(t),A_\ell(t),S_\ell(t))\rrbracket,\qquad 
\dot a_\ell(t)=A_\ell(t)x,
\qquad
\dot A_\ell(t)=S_\ell(t)(x,\cdot),
\]
with \((a_\ell(0),A_\ell(0),S_\ell(0))=(a_\ell,A_\ell,S_\ell)\). The Hölder bound for \(D^2u\) and the usual chain-of-balls argument, now based at \(t=1\), give \(|S_\ell(t)-S_\ell(1)|\leq C_Q[D^2u]_{\alpha;B_R}(1-t)^\alpha|x|^\alpha\). Taylor's formula at \(t=1\), applied to \(a_\ell\) to second order and to \(A_\ell\) to first order, therefore yields
\[
\mathcal G(K(x),j_2u(0))\leq C_Q[D^2u]_{\alpha;B_R}|x|^\alpha(1+|x|+|x|^2).
\]

Choose \(d_0>0\) such that \(C_QMR^\alpha(1+R+R^2)<d_0/3\). Since \([D^2u]_{\alpha;B_R}\leq M\) and \(d\geq d_0\), the preceding estimate shows that every atom of \(K(x)\) belongs to the union of the fixed, pairwise disjoint sets
\[
\mathcal N_\gamma:=\bigcup_{\ell\in I_\gamma}B_{d/3}((a_\ell,A_\ell,S_\ell)),\qquad \gamma=1,\ldots,N.
\]
As in the proof of Corollary~\ref{cor:persistence-jet-clusters}, continuity gives a unique assignment to these sets, each containing precisely \(|I_\gamma|\) atoms, counted with multiplicity. Transforming these groups back gives continuous sub-jets \(j_2u=J^1+\cdots+J^N\), with \(J^\gamma(0)=\sum_{\ell\in I_\gamma}\llbracket(a_\ell,A_\ell,S_\ell)\rrbracket\).

Applying \textup{(iv)} locally to the distinct second-jet atoms and grouping the resulting components according to \(J^\gamma\) shows that their value projections \(u^\gamma\) belong to \(C^2(B_R)\) and satisfy \(j_2u^\gamma=J^\gamma\). The decomposition is unique, since along any segment from the origin the transported jet sheets cannot pass between the disjoint sets \(\mathcal N_\gamma\).

Finally, apply \textup{(ii)} along an arbitrary segment joining \(x,y\in B_R\), with selections respecting the sub-jet decomposition. The same chain-of-balls argument gives \(\mathcal G(D^2u^\gamma(x),D^2u^\gamma(y))\leq C_Q[D^2u]_{\alpha;B_R}|x-y|^\alpha\), and hence \([D^2u^\gamma]_{\alpha;B_R}\leq C_Q[D^2u]_{\alpha;B_R}\), completing the proof.
\end{proof}

\begin{definition}[\(Q\)-valued strong solutions]
\label{def:Q-valued-strong-solution}
Let $\mathscr A^{ij}: \Omega\times\mathbb R^k\times \operatorname{Hom}(\mathbb R^n,\mathbb R^k) \longrightarrow \operatorname{End}(\mathbb R^k)$, $i,j=1,\ldots,n$, and $\mathscr C: \Omega\times\mathbb R^k\times \operatorname{Hom}(\mathbb R^n,\mathbb R^k) \longrightarrow \mathbb R^k$ be continuous, with \(\mathscr A^{ij}=\mathscr A^{ji}\). Let \(f\in C^0(\Omega,\mathbb R^k)\).
Let \(u\in C^2\bigl(\Omega,\mathcal A_Q(\mathbb R^k)\bigr)\), and let $j_2u=\sum_{\ell=1}^Q\llbracket(u_\ell,Du_\ell,D^2u_\ell)\rrbracket$ be any measurable selection of the second jet.

We say that $u$ is a \(Q\)-valued strong solution of
\[
\Delta u
+
\sum_{i,j=1}^n
\mathscr A^{ij}(x,u,Du)D_{ij}u
+
\mathscr C(x,u,Du)
=f
\]
in $\Omega$ if
\[
\operatorname{tr}D^2u_\ell(x)
+
\sum_{i,j=1}^n
\mathscr A^{ij}(x,u_\ell(x),Du_\ell(x))
D^2_{ij}u_\ell(x)
+
\mathscr C(x,u_\ell(x),Du_\ell(x))
=f(x)
\]
for every \(x\in\Omega\) and \(\ell=1,\ldots,Q\), where
\(D^2_{ij}u_\ell(x):=D^2u_\ell(x)[e_i,e_j]\).
\end{definition}

For the remainder of this subsection, hypothesis \((\star\star)\) in a ball $B$ denotes the following conditions on a pair \((\mathscr A,\mathscr C)\) as in Definition~\ref{def:Q-valued-strong-solution}, with parameters \(0<\beta\leq1\), \(L,\Lambda<\infty\), and \(\varepsilon>0\):
\[
(\star\star)\qquad
\left\{
\begin{aligned}
&\mathscr A^{ij},\mathscr C \text{ are continuous},\\
&\left|\mathscr A^{ij}(x,a,0)\right|\leq\varepsilon,\\
&\left|
\mathscr A^{ij}(x,a,P)
-
\mathscr A^{ij}(x',a',P)
\right|
\leq
\Lambda\left(
|x-x'|^\beta+|a-a'|
\right),\\
&\left|
\mathscr A^{ij}(x,a,P)
-
\mathscr A^{ij}(x,a,P')
\right|
\leq
\varepsilon|P-P'|,\\
&\mathscr C(x,0,0)=0,\\
&\left|
\mathscr C(x,a,P)
-
\mathscr C(x,a',P')
\right|
\leq
\Lambda\left(
|a-a'|+|P-P'|
\right),\\
&\left|
\mathscr C(x,a,P)
-
\mathscr C(x',a,P)
\right|
\leq
\Lambda |x-x'|^\beta\left(|a|+|P|\right)
\end{aligned}
\right.
\]
for every \(i,j=1,\ldots,n\), \(x,x'\in B\), \(a,a'\in\mathbb R^k\), and \(|P|,|P'|\leq L\).

\begin{lemma}[Perturbative quadratic blow-ups]
\label{lem:perturbative-quadratic-blow-ups}
Let \(q\geq1\), \(0<\alpha\leq\beta\leq1\), \(L,\Lambda<\infty\), and \(R>0\).
Let \(r_h\downarrow0\). For every \(h\), let \((\mathscr A_h,\mathscr C_h)\) satisfy hypothesis \((\star\star)\) on \(B_{r_hR}\) with parameters \(\beta,L,\Lambda,\varepsilon_h\), where $\varepsilon_h\longrightarrow0$.

Let $f_h\in C^{0,\alpha}(B_{r_hR},\mathbb R^k)$, and let $u_h\in C^2\bigl(B_{r_hR},\mathcal A_q(\mathbb R^k)\bigr)$ be a \(q\)-valued strong solution of
\[
\Delta u_h
+
\sum_{i,j=1}^n
\mathscr A_h^{ij}(x,u_h,Du_h)D_{ij}u_h
+
\mathscr C_h(x,u_h,Du_h)
=f_h
\]
in \(B_{r_hR}\), with $\sup_{B_{r_hR}}|Du_h|\leq L$.

Let $a_h\in\mathbb R^k$, $P_h\in\operatorname{Hom}(\mathbb R^n,\mathbb R^k)$, $S_h\in\operatorname{Sym}^2(\mathbb R^n,\mathbb R^k)$ satisfy $(a_h,P_h,S_h)\in\operatorname{spt}j_2u_h(0)$, and define
\[
q_h(x):=a_h+P_hx+\frac12S_h(x,x).
\]
Let \(H_h>0\). Suppose that, for some \(C_0<\infty\),
\[
\frac{|a_h|+|P_h|+|S_h|}{H_h}\longrightarrow0,
\qquad
\frac{[f_h]_{\alpha;B_{r_hR}}}{H_h}
\longrightarrow0,
\qquad
r_h|S_h|
+
r_h^{1-\alpha}\frac{|S_h|^2}{H_h}
\leq C_0,
\]
and that
\[
w_h(z)
:=
\frac{
u_h(r_hz)\ominus q_h(r_hz)
}{
r_h^{2+\alpha}H_h
}
\longrightarrow w
\qquad
\text{locally in }C^2\bigl(B_R,\mathcal A_q(\mathbb R^k)\bigr).
\]

Then \(w\) satisfies $\Delta w=0$ in the strong \(q\)-valued sense in \(B_R\), and \(Dw\) is Dirichlet-stationary in \(B_R\).
\end{lemma}

\begin{proof}
Set \(v_h(z):=u_h(r_hz)\). A change of variables in the strong equation shows that \(v_h\) satisfies
\[
\Delta v_h
+
\sum_{i,j=1}^n
\mathscr A_h^{ij}
\left(
r_hz,
v_h,
\frac{Dv_h}{r_h}
\right)
D_{ij}v_h
+
r_h^2
\mathscr C_h
\left(
r_hz,
v_h,
\frac{Dv_h}{r_h}
\right)
=
r_h^2f_h(r_hz)
\]
in \(B_R\).

For a measurable selection \(j_2w_h=\sum_{\ell=1}^q\llbracket(w_{h,\ell},Dw_{h,\ell},D^2w_{h,\ell})\rrbracket\), let \(j_2v_h=\sum_{\ell=1}^q\llbracket(v_{h,\ell},Dv_{h,\ell},D^2v_{h,\ell})\rrbracket\) be the corresponding selection of $j_2v_h$. Then
\begin{equation}
\label{eq:rescaled-second-jet-identities}
v_{h,\ell}=q_h(r_hz)+r_h^{2+\alpha}H_hw_{h,\ell},\qquad
\frac{Dv_{h,\ell}}{r_h}=P_h+r_hS_h(z,\cdot)+r_h^{1+\alpha}H_hDw_{h,\ell},
\qquad
\frac{D^2v_{h,\ell}}{r_h^2}=S_h+r_h^\alpha H_hD^2w_{h,\ell}.
\end{equation}

Since $(a_h,P_h,S_h)\in\operatorname{spt}j_2u_h(0)$,
the strong equation at \(0\) gives
\[
\operatorname{tr}S_h
+
\sum_{i,j=1}^n
\mathscr A_h^{ij}(0,a_h,P_h)(S_h)_{ij}
+
\mathscr C_h(0,a_h,P_h)
=
f_h(0).
\]
Using \eqref{eq:rescaled-second-jet-identities}, we subtract \(r_h^2\) times this identity from the rescaled strong equation and divide by \(r_h^{2+\alpha}H_h\) to obtain, for every sheet,
\begin{equation}
\label{eq:quadratic-blow-up-equation}
\Delta w_{h,\ell}+E_{h,\ell}
=
\frac{f_h(r_hz)-f_h(0)}{r_h^\alpha H_h},
\end{equation}
where
\[
\begin{aligned}
E_{h,\ell}
={}&
\sum_{i,j=1}^n
\mathscr A_h^{ij}
\left(
r_hz,
v_{h,\ell},
\frac{Dv_{h,\ell}}{r_h}
\right)
D_{ij}w_{h,\ell}
+
\frac{1}{r_h^\alpha H_h}
\sum_{i,j=1}^n
\left[
\mathscr A_h^{ij}
\left(
r_hz,
v_{h,\ell},
\frac{Dv_{h,\ell}}{r_h}
\right)
-
\mathscr A_h^{ij}(0,a_h,P_h)
\right]
(S_h)_{ij}
\\
&
+
\frac{
\mathscr C_h
\left(
r_hz,
v_{h,\ell},
\frac{Dv_{h,\ell}}{r_h}
\right)
-
\mathscr C_h(0,a_h,P_h)
}{
r_h^\alpha H_h
}.
\end{aligned}
\]

We claim that $E_{h,\ell}\longrightarrow0$ locally uniformly in \(B_R\), uniformly over the sheets.

We first observe that, since $(a_h,P_h,S_h)\in\operatorname{spt}j_2u_h(0)$ and $\sup_{B_{r_hR}}|Du_h|\leq L$, we have $|P_h|\leq L$; moreover, $|Dv_{h,\ell}|/r_h\leq L$.
Then, since $\left|\mathscr A_h^{ij}(x,a,0)\right|\leq\varepsilon_h$ and the Lipschitz constant of \(\mathscr A_h^{ij}\) in \(P\) is at most
\(\varepsilon_h\), we have
\[
\begin{aligned}
\left|
\mathscr A_h^{ij}
\left(
r_hz,
v_{h,\ell},
\frac{Dv_{h,\ell}}{r_h}
\right)
\right|
&\leq
\left|\mathscr A_h^{ij}(r_hz,v_{h,\ell},0)\right|
+
\left|
\mathscr A_h^{ij}
\left(
r_hz,
v_{h,\ell},
\frac{Dv_{h,\ell}}{r_h}
\right)
-
\mathscr A_h^{ij}(r_hz,v_{h,\ell},0)
\right|\\
&\leq
\varepsilon_h
+
\varepsilon_h\frac{|Dv_{h,\ell}|}{r_h}
\leq
\varepsilon_h(1+L).
\end{aligned}
\]
The first term in \(E_{h,\ell}\) therefore converges to zero locally
uniformly, since \(w_h\to w\) locally in \(C^2\).

By \eqref{eq:rescaled-second-jet-identities} and the local \(C^2\)-convergence, on every compact subset of \(B_R\), we have
\begin{equation}
\label{eq:quadratic-blow-up-jet-bounds}
|v_{h,\ell}-a_h|
\leq
C\left(
r_h|P_h|
+
r_h^2|S_h|
+
r_h^{2+\alpha}H_h
\right)
\qquad\text{and}\qquad
\left|
\frac{Dv_{h,\ell}}{r_h}-P_h
\right|
\leq
C\left(
r_h|S_h|
+
r_h^{1+\alpha}H_h
\right).
\end{equation}
Using \eqref{eq:quadratic-blow-up-jet-bounds} and the conditions on \(\mathscr A_h^{ij}\) in hypothesis \((\star\star)\), we obtain
\[
\begin{aligned}
&
\frac{|S_h|}{r_h^\alpha H_h}
\left|
\mathscr A_h^{ij}
\left(
r_hz,
v_{h,\ell},
\frac{Dv_{h,\ell}}{r_h}
\right)
-
\mathscr A_h^{ij}(0,a_h,P_h)
\right|
\\
&\leq
C\Bigg[
r_h^{\beta-\alpha}\frac{|S_h|}{H_h}
+
r_h^{1-\alpha}\frac{|S_h|}{H_h}
+
r_h^{2-\alpha}\frac{|S_h|^2}{H_h}
+
r_h^2|S_h|
+
\varepsilon_h
r_h^{1-\alpha}\frac{|S_h|^2}{H_h}
+
\varepsilon_h r_h|S_h|
\Bigg],
\end{aligned}
\]
where \(C<\infty\) is independent of \(h\) and may depend on the fixed parameters and on the compact subset under consideration.
Every term on the right tends to zero by the assumptions of the lemma.

Similarly, \eqref{eq:quadratic-blow-up-jet-bounds} and hypothesis $(\star\star)$ on \(\mathscr C_h\) give
\[
\frac{
\left|
\mathscr C_h
\left(
r_hz,
v_{h,\ell},
\frac{Dv_{h,\ell}}{r_h}
\right)
-
\mathscr C_h(0,a_h,P_h)
\right|
}{
r_h^\alpha H_h
}
\leq
C\left[
r_h^{\beta-\alpha}\frac{|a_h|+|P_h|}{H_h}
+
r_h^{1-\alpha}\frac{|P_h|}{H_h}
+
r_h^{2-\alpha}\frac{|S_h|}{H_h}
+
r_h^2
+
r_h^{1-\alpha}\frac{|S_h|}{H_h}
+
r_h
\right],
\]
which also converges to zero. This proves the claim.

Moreover, on every compact subset of \(B_R\),
\[
\left|
\frac{f_h(r_hz)-f_h(0)}{r_h^\alpha H_h}
\right|
\leq
\frac{[f_h]_{\alpha;B_{r_hR}}}{H_h}|z|^\alpha
\longrightarrow0
\]
uniformly.

Passing to the limit in \eqref{eq:quadratic-blow-up-equation} gives $\Delta w=0$ in the strong \(q\)-valued sense in \(B_R\).
By Lemma~\ref{lem:second-jet-calculus}\textup{(i)}, the atoms of \(D^2w\) are symmetric, while the preceding equation shows that they are trace-free.
Hence \(G:=Dw\) satisfies the Hodge identities \eqref{eq:Hodge-field-identities}. The argument in the proof of Theorem~\ref{thm:stationarity-of-differential} therefore shows that \(Dw\) is Dirichlet-stationary in \(B_R\).
\end{proof}

\begin{theorem}[\(C^{2,\alpha}\) Schauder estimates for \(Q\)-valued quasilinear systems]
\label{thm:quasilinear-C2-Schauder}
Let \(0<\alpha<1/Q\) if \(n=2\), and $0<\alpha<\delta(n,kn,Q)$ if \(n\geq3\). Let \(\alpha\leq\beta\leq1\) and \(L,\Lambda<\infty\).
There exist $\varepsilon_0 =\varepsilon_0(n,k,Q,\alpha,\beta,L,\Lambda)>0$ and $C=C(n,k,Q,\alpha,\beta,L,\Lambda)<\infty$ with the following property.

Let \((\mathscr A,\mathscr C)\) satisfy hypothesis \((\star\star)\) on \(B^n_1\) with parameters \(\beta,L,\Lambda,\varepsilon_0\), let \(f\in C^{0,\alpha}(B^n_1,\mathbb R^k)\), and let $u\in C^{2,\alpha}\bigl(B^n_1,\mathcal A_Q(\mathbb R^k)\bigr)$ be a \(Q\)-valued strong solution of
\[
\Delta u+\sum_{i,j=1}^n\mathscr A^{ij}(x,u,Du)D_{ij}u+\mathscr C(x,u,Du)=f
\]
in \(B^n_1\), with $\sup_{B^n_1}|Du|\leq L$.
Then
\[
\|u\|_{C^{2,\alpha}(B_{1/2}^n)}
\leq
C\left(
\|u\|_{L^2(B_1^n)}
+
\|f\|_{C^{0,\alpha}(B_1^n)}
\right).
\]
\end{theorem}

\begin{proof}

\smallskip
\noindent
\emph{\underline{Step 1: reduction to a simpler estimate.}}
We first observe that hypothesis \((\star\star)\) is preserved under graphical rescalings. Indeed, if \(0<r\leq1\), \((\mathscr A,\mathscr C)\) satisfies hypothesis \((\star\star)\) on \(B_r(x_0)\), and \(u\) solves the equation there with right-hand side \(f\), then
\[
u_{x_0,r}(z):=\frac{u(x_0+rz)}{r}
\]
solves an equation of the same form in \(B_1\), with $\mathscr A_r^{ij}(z,a,P):=\mathscr A^{ij}(x_0+rz,ra,P)$, $\mathscr C_r(z,a,P):=r\,\mathscr C(x_0+rz,ra,P)$, $f_r(z):=r f(x_0+rz)$, and \((\mathscr A_r,\mathscr C_r)\) satisfies hypothesis \((\star\star)\) on \(B_1\) with the same parameters.
Moreover, $\|f_r\|_{C^{0,\alpha}(B_1)} \leq r\|f\|_{C^{0,\alpha}(B_r(x_0))}$.

By the \(Q\)-valued versions of Simon's interpolation inequality \cite[(1.5)]{Sim97} and the absorbing lemma in \cite[Section~4]{Sim97}, applied successively to \(u\) and \(Du\), with the required \(C^2\)-compactness provided by Lemma~\ref{lem:second-jet-calculus}\textup{(iii)}, it is enough to show that, for every \(\delta>0\), there exist \(\varepsilon_\delta>0\) and \(C_\delta<\infty\) such that
\begin{equation}
\label{eq:C2-reduction-estimate}
[D^2u]_{\alpha;B_{1/2}}
\leq
\delta[D^2u]_{\alpha;B_1}
+
C_\delta\left(
\|u\|_{C^2(B_1)}
+
\|f\|_{C^{0,\alpha}(B_1)}
\right),
\end{equation}
whenever \((\mathscr A,\mathscr C)\) satisfies hypothesis \((\star\star)\) on \(B_1\) with parameters \(\beta,L,\Lambda,\varepsilon_\delta\), \(f\in C^{0,\alpha}(B_1,\mathbb R^k)\), and \(u\in C^{2,\alpha}\bigl(B_1,\mathcal A_Q(\mathbb R^k)\bigr)\) is a \(Q\)-valued strong solution of the corresponding equation with right-hand side \(f\), with \(\sup_{B_1}|Du|\leq L\).

\smallskip
\noindent
\emph{\underline{Step 2: contradiction sequence and rescaling.}}
Suppose that \eqref{eq:C2-reduction-estimate} fails for some fixed \(\delta>0\). Then there are pairs \((\mathscr A_h,\mathscr C_h)\) satisfying hypothesis \((\star\star)\) on \(B_1\) with parameters \(\beta,L,\Lambda,h^{-1}\), functions \(f_h\in C^{0,\alpha}(B_1,\mathbb R^k)\), and strong solutions $u_h\in C^{2,\alpha}\bigl(B_1,\mathcal A_Q(\mathbb R^k)\bigr)$ of the corresponding equations with right-hand sides \(f_h\) with \(\sup_{B_1}|Du_h|\leq L\), such that
\[
[D^2u_h]_{\alpha;B_{1/2}}>\delta[D^2u_h]_{\alpha;B_1}+h\bigl(\|u_h\|_{C^2(B_1)}+\|f_h\|_{C^{0,\alpha}(B_1)}\bigr).
\]
Set $H_h:=[D^2u_h]_{\alpha;B_{1/2}}$. Then
\begin{equation}
\label{eq:C2-contradiction-ratios}
\frac{[D^2u_h]_{\alpha;B_1}}{H_h}\leq\delta^{-1},
\qquad
\frac{
\|u_h\|_{C^2(B_1)}
+
\|f_h\|_{C^{0,\alpha}(B_1)}
}{H_h}
\longrightarrow0.
\end{equation}

Choose \(x_h,y_h\in B_{1/2}\) such that, writing $\rho_h:=|x_h-y_h|$, $\xi_h:=\frac{y_h-x_h}{\rho_h}$, one has
\begin{equation}
\label{eq:C2-lower-bound-Holder}
\mathcal G\bigl(D^2u_h(x_h),D^2u_h(y_h)\bigr)\geq\frac12H_h\rho_h^\alpha.
\end{equation}
Since
\[
\rho_h^\alpha\leq\frac{4\sup_{B_1}|D^2u_h|}{H_h}\longrightarrow0,
\]
we have \(\rho_h\to0\). After passing to a subsequence, we may assume that $\xi_h\longrightarrow\xi\in\mathbb S^{n-1}$.

Set $\Omega_h:=\{z\in\mathbb R^n:x_h+\rho_hz\in B_1\}$ and define
\[
\widetilde w_h(z):=\frac{u_h(x_h+\rho_hz)}{\rho_h^{2+\alpha}H_h}.
\]
The domains \(\Omega_h\) exhaust \(\mathbb R^n\). Moreover,
\begin{equation}
\label{eq:C2-rescaled-Holder-bound}
[D^2\widetilde w_h]_{\alpha;\Omega_h}=\frac{[D^2u_h]_{\alpha;B_1}}{H_h}\leq\delta^{-1},
\end{equation}
while \eqref{eq:C2-lower-bound-Holder} becomes
\begin{equation}
\label{eq:C2-rescaled-Hessian-separation}
\mathcal G\bigl(D^2\widetilde w_h(0),D^2\widetilde w_h(\xi_h)\bigr)\geq\frac12.
\end{equation}

\smallskip
\noindent
\emph{\underline{Step 3: second-jet clustering.}}
Write $j_2\widetilde w_h(0)=\sum_{i=1}^Q\llbracket(a_{h,i},A_{h,i},S_{h,i})\rrbracket$.
Since there are only finitely many pairs of indices, after passing to a subsequence we may assume that $|(a_{h,i},A_{h,i},S_{h,i})-(a_{h,j},A_{h,j},S_{h,j})|$ has a limit in \([0,\infty]\) for every \(i,j\). Declare
\[
i\sim j\qquad\Longleftrightarrow\qquad\lim_h\left|(a_{h,i},A_{h,i},S_{h,i})-(a_{h,j},A_{h,j},S_{h,j})\right|<\infty.
\]
Let
\[
\{1,\ldots,Q\}=I_1\cup\cdots\cup I_N,\qquad q_m:=|I_m|,
\]
be the corresponding equivalence classes. Thus the second jets in the same class have uniformly bounded mutual distance, whereas
\[
d_h^{m, m'}:=\min_{\substack{i\in I_m\\j\in I_{m'}}}\left|(a_{h,i},A_{h,i},S_{h,i})-(a_{h,j},A_{h,j},S_{h,j})\right|\longrightarrow\infty
\]
whenever \(m\neq m'\).

If \(N=1\), we simply set \(\widetilde w_h^{(1)}:=\widetilde w_h\) and choose any \(R_h\to\infty\) such that \(B_{R_h}\subset\Omega_h\).

If \(N\geq2\), set $d_h:=\min_{m\neq m'}d_h^{m, m'}$.
For every fixed \(R<\infty\), the bound \eqref{eq:C2-rescaled-Holder-bound}, together with \(d_h\to\infty\), allows us to apply Lemma~\ref{lem:second-jet-calculus}\textup{(v)} for all sufficiently large \(h\). We obtain a unique decomposition $\widetilde w_h=\sum_{m=1}^N\widetilde w_h^{(m)}$ on \(B_R\), where $\widetilde w_h^{(m)}\in C^{2,\alpha}\bigl(B_R,\mathcal A_{q_m}(\mathbb R^k)\bigr)$ and $[D^2\widetilde w_h^{(m)}]_{\alpha;B_R}\leq C_Q[D^2\widetilde w_h]_{\alpha;B_R}\leq C_Q\delta^{-1}$.
By a diagonal argument and uniqueness of the decomposition, we may choose \(R_h\to\infty\), with \(B_{R_h}\subset\Omega_h\), so that these decompositions are defined on \(B_{R_h}\).

Define the corresponding components of \(u_h\) by
\[
u_h^{(m)}(x_h+\rho_hz):=\rho_h^{2+\alpha}H_h\widetilde w_h^{(m)}(z),\qquad z\in B_{R_h}.
\]
Then $u_h=\sum_{m=1}^N u_h^{(m)}$ on \(B_{\rho_hR_h}(x_h)\). Since the strong equation is imposed sheetwise on the second jet, every \(u_h^{(m)}\) is a \(q_m\)-valued strong solution of the same system with right-hand side \(f_h\).

\smallskip
\noindent
\emph{\underline{Step 4: extraction of the limiting quadratic blow-ups.}}
For each \(m\), choose one index \(i_m\in I_m\), and let $(\tilde a_h^{(m)},\tilde P_h^{(m)},\tilde S_h^{(m)}):=(a_{h,i_m},A_{h,i_m},S_{h,i_m})$.
Define
\[
\widehat w_h^{(m)}(z):=\widetilde w_h^{(m)}(z)\ominus\left(\tilde a_h^{(m)}+\tilde P_h^{(m)}z+\frac12\tilde S_h^{(m)}(z,z)\right).
\]
Since the second jets belonging to \(I_m\) have uniformly bounded mutual distance at the origin, $|\widehat w_h^{(m)}(0)|+|D\widehat w_h^{(m)}(0)|+|D^2\widehat w_h^{(m)}(0)|\leq C_m$. Moreover, $[D^2\widehat w_h^{(m)}]_{\alpha;B_{R_h}}=[D^2\widetilde w_h^{(m)}]_{\alpha;B_{R_h}}\leq C_Q\delta^{-1}$.
Using Lemma~\ref{lem:second-jet-calculus}\textup{(ii)} along segments, together with the usual chain-of-balls argument and integration once and twice, we obtain, for every fixed \(R<\infty\),
\[
\|\widehat w_h^{(m)}\|_{C^{2,\alpha}(B_R)}\leq C_R
\]
for all sufficiently large \(h\).

Lemma~\ref{lem:second-jet-calculus}\textup{(iii)} and a diagonal argument therefore give, after passing to a subsequence,
\[
\widehat w_h^{(m)}\longrightarrow w^{(m)}\qquad\text{locally in }C^2(\mathbb R^n)
\]
for every \(m\), where $[D^2w^{(m)}]_{\alpha;\mathbb R^n}\leq C_Q\delta^{-1}$.

We now want to apply Lemma~\ref{lem:perturbative-quadratic-blow-ups}. Fix \(R<\infty\). For all sufficiently large \(h\), \(R<R_h\) and \(x_h+B_{\rho_hR}\subset B_1\). Define $u_h^{(m),0}(x):=u_h^{(m)}(x_h+x)$, $(\mathscr A_h^0)^{ij}(x,a,P):=\mathscr A_h^{ij}(x_h+x,a,P)$, $\mathscr C_h^0(x,a,P):=\mathscr C_h(x_h+x,a,P)$, and $f_h^0(x):=f_h(x_h+x)$.
Then \(u_h^{(m),0}\) is a \(q_m\)-valued strong solution on \(B_{\rho_hR}\) with right-hand side \(f_h^0\), and \((\mathscr A_h^0,\mathscr C_h^0)\) satisfies hypothesis \((\star\star)\) there with parameters \(\beta,L,\Lambda,h^{-1}\). Moreover, $\sup_{B_{\rho_hR}}|Du_h^{(m),0}|\leq L$.

We apply Lemma~\ref{lem:perturbative-quadratic-blow-ups} with $q=q_m$, $r_h:=\rho_h$, and $a_h^{(m)}:=\rho_h^{2+\alpha}H_h\tilde a_h^{(m)}$, $P_h^{(m)}:=\rho_h^{1+\alpha}H_h\tilde P_h^{(m)}$, $S_h^{(m)}:=\rho_h^\alpha H_h\tilde S_h^{(m)}$. Define also $q_h^{(m)}(x):=a_h^{(m)}+P_h^{(m)}x+\frac12S_h^{(m)}(x,x)$.
Then $(a_h^{(m)},P_h^{(m)},S_h^{(m)})\in\operatorname{spt}j_2u_h^{(m),0}(0)$ and, by definition,
\[
q_h^{(m)}(\rho_hz)=\rho_h^{2+\alpha}H_h\left(\tilde a_h^{(m)}+\tilde P_h^{(m)}z+\frac12\tilde S_h^{(m)}(z,z)\right).
\]
Hence
\[
\frac{u_h^{(m),0}(\rho_hz)\ominus q_h^{(m)}(\rho_hz)}{\rho_h^{2+\alpha}H_h}=\widehat w_h^{(m)}(z)\longrightarrow w^{(m)}\qquad\text{locally in }C^2(B_R).
\]
Moreover,
\[
\frac{|a_h^{(m)}|+|P_h^{(m)}|+|S_h^{(m)}|}{H_h}\leq\frac{\|u_h\|_{C^2(B_1)}}{H_h}\longrightarrow0
\]
by \eqref{eq:C2-contradiction-ratios}. 
Moreover,
\[
\frac{[f_h^0]_{\alpha;B_{\rho_hR}}}{H_h}
\leq
\frac{[f_h]_{\alpha;B_1}}{H_h}
\longrightarrow0
\]
by \eqref{eq:C2-contradiction-ratios}.
Hence, in order to apply Lemma~\ref{lem:perturbative-quadratic-blow-ups}, it only remains to verify that
\begin{equation}
\label{eq:C2-quadratic-jet-scale-bound}
\rho_h|S_h^{(m)}|+\rho_h^{1-\alpha}\frac{|S_h^{(m)}|^2}{H_h}\leq C_0
\end{equation}
for some $C_0<\infty$.

By the interpolation inequality applied to \(Du_h\),
\[
\sup_{B_{3/4}}|D^2u_h|\leq C\left([D^2u_h]_{\alpha;B_1}^{\frac1{1+\alpha}}\left(\sup_{B_1}|Du_h|\right)^{\frac{\alpha}{1+\alpha}}+\sup_{B_1}|Du_h|\right).
\]
Indeed, fix \(x\in B_{3/4}\), an atom \(S_0\in\operatorname{spt}D^2u_h(x)\), a unit vector \(e\), and \(0<t\leq1/8\). By Lemma~\ref{lem:second-jet-calculus}\textup{(ii)} and the usual chain-of-balls argument, there are corresponding atoms \(P_0\in\operatorname{spt}Du_h(x)\) and \(P_t\in\operatorname{spt}Du_h(x+te)\) such that $|P_t-P_0-tS_0(e,\cdot)|\leq C[D^2u_h]_{\alpha;B_1}t^{1+\alpha}$.
Hence
\[
|S_0(e,\cdot)|\leq\frac{|P_t-P_0|}{t}+C[D^2u_h]_{\alpha;B_1}t^\alpha\leq C\left(\frac{\sup_{B_1}|Du_h|}{t}+[D^2u_h]_{\alpha;B_1}t^\alpha\right).
\]
Taking the supremum and optimizing in \(t\) gives the asserted estimate.
Since \eqref{eq:C2-contradiction-ratios} gives $\frac{[D^2u_h]_{\alpha;B_1}}{H_h}\leq\delta^{-1}$, $\frac{\sup_{B_1}|Du_h|}{H_h}\longrightarrow0$, and \(\sup_{B_1}|Du_h|\leq L\), it follows that
$$
\frac{\left(\sup_{B_{3/4}}|D^2u_h|\right)^{1+\alpha}}{H_h}\leq C.
$$
On the other hand, the choice of \(x_h,y_h\) gives $H_h\rho_h^\alpha\leq4\sup_{B_{3/4}}|D^2u_h|$. Hence
\[
\rho_h\sup_{B_{3/4}}|D^2u_h|
\leq
4\rho_h^{1-\alpha}\frac{\left(\sup_{B_{3/4}}|D^2u_h|\right)^2}{H_h}
\leq
C (\rho_h\sup_{B_{3/4}}|D^2u_h|)^{1-\alpha}.
\]
Consequently,
\[
\rho_h\sup_{B_{3/4}}|D^2u_h|+\rho_h^{1-\alpha}\frac{\left(\sup_{B_{3/4}}|D^2u_h|\right)^2}{H_h}\leq C.
\]
Since $|S_h^{(m)}|\leq\sup_{B_{3/4}}|D^2u_h|$, the required bound \eqref{eq:C2-quadratic-jet-scale-bound} follows.

Lemma~\ref{lem:perturbative-quadratic-blow-ups} therefore gives
\[
\Delta w^{(m)}=0\qquad\text{in }B_R,
\]
and \(Dw^{(m)}\) is Dirichlet-stationary in \(B_R\). Since \(R<\infty\) was arbitrary, \(\Delta w^{(m)}=0\) in \(\mathbb R^n\) and \(Dw^{(m)}\) is Dirichlet-stationary in \(\mathbb R^n\). Moreover,
\[
[D(Dw^{(m)})]_{\alpha;\mathbb R^n}=[D^2w^{(m)}]_{\alpha;\mathbb R^n}\leq C_Q\delta^{-1}.
\]

\smallskip
\noindent
\emph{\underline{Step 5: conclusion via the Liouville theorem.}}
The map $Dw^{(m)}:\mathbb R^n\longrightarrow\mathcal A_{q_m}\bigl(\operatorname{Hom}(\mathbb R^n,\mathbb R^k)\bigr)$ takes values in a Euclidean space of dimension \(kn\).

If \(n=2\), then $\alpha<\frac1Q\leq\frac1{q_m}$.
If \(n\geq3\), then $\alpha<\delta(n,kn,Q)\leq\delta(n,kn,q_m)$.
Theorem~\ref{thm:liouville}, applied to \(Dw^{(m)}\), therefore shows that \(Dw^{(m)}\) is affine. Hence $D^2w^{(m)}$ is constant.

Concatenating optimal matchings within the individual second-jet clusters gives
\[
\mathcal G\bigl(D^2\widetilde w_h(0),D^2\widetilde w_h(\xi_h)\bigr)^2\leq\sum_{m=1}^N\mathcal G\bigl(D^2\widetilde w_h^{(m)}(0),D^2\widetilde w_h^{(m)}(\xi_h)\bigr)^2.
\]
Since the same quadratic polynomial is subtracted at both points,
\[
\mathcal G\bigl(D^2\widetilde w_h^{(m)}(0),D^2\widetilde w_h^{(m)}(\xi_h)\bigr)=\mathcal G\bigl(D^2\widehat w_h^{(m)}(0),D^2\widehat w_h^{(m)}(\xi_h)\bigr).
\]
The local \(C^2\)-convergence, the fact that \(\xi_h\to\xi\), and the constancy of \(D^2w^{(m)}\) therefore give
\[
\mathcal G\bigl(D^2\widetilde w_h(0),D^2\widetilde w_h(\xi_h)\bigr)\longrightarrow0.
\]
This contradicts \eqref{eq:C2-rescaled-Hessian-separation}. Hence \eqref{eq:C2-reduction-estimate} holds. The interpolation and absorption argument from Step~1 then gives the asserted Schauder estimate.
\end{proof}

\section{Applications to area-stationary maps}
\subsection{Small-slope Schauder estimates and Bernstein theorem}

\begin{corollary}[Schauder estimates for stationary \(Q\)-valued graphs]
\label{cor:stationary-graph-Schauder}
Let \(0<\alpha<1/Q\) if \(n=2\), and \(0<\alpha<\delta(n,k,Q)\) if \(n\geq3\). There exist constants $\varepsilon_0=\varepsilon_0(n,k,Q,\alpha)>0$, $C=C(n,k,Q,\alpha)<\infty$, such that, if $u\in C^{1,\alpha}\bigl(B^n_1,\mathcal A_Q(\mathbb R^k)\bigr)$, $\sup_{B^n_1}|Du|\leq\varepsilon_0$, and the graph varifold associated with \(u\) is stationary, then
\[
\|u\|_{C^{1,\alpha}(B_{1/2}^n)}
\leq
C\|u\|_{L^2(B_1^n)}.
\]
\end{corollary}

\begin{proof}
As observed in Remark~\ref{rem:Q-valued-weak-solutions-examples}\textup{(iii)}, stationarity of the graph implies that \(u\) is a \(Q\)-valued weak solution of \(\Delta u+\operatorname{div}\mathscr E_{\mathrm{area}}(Du)=0\), where \(\mathscr E_{\mathrm{area}}(P):=D_P\sqrt{\det(I+P^TP)}-P\).

Let \(\varepsilon_*\) be the perturbation threshold in Theorem~\ref{thm:quasilinear-Schauder}, with \(\beta=L=\Lambda=1\). Since \(D_P\mathscr E_{\mathrm{area}}(0)=0\), we may choose \(0<\rho\leq1\) such that \(\sup_{|P|\leq1}|D_P\mathscr E_{\mathrm{area}}(\rho P)|\leq\varepsilon_*\).

Suppose that \(\sup_{B_1}|Du|\leq\rho\), and set \(\widetilde u:=\rho^{-1}u\) and \(\mathscr E_\rho(P):=\rho^{-1}\mathscr E_{\mathrm{area}}(\rho P)\). Then \(\sup_{B_1}|D\widetilde u|\leq1\), and \(\widetilde u\) is a \(Q\)-valued weak solution of \(\Delta\widetilde u+\operatorname{div}\mathscr E_\rho(D\widetilde u)=0\). Moreover, \(D_P\mathscr E_\rho(P)=D_P\mathscr E_{\mathrm{area}}(\rho P)\), so the mean value theorem gives \(|\mathscr E_\rho(P)-\mathscr E_\rho(P')|\leq\varepsilon_*|P-P'|\) whenever \(|P|,|P'|\leq1\). Since \(\mathscr E_\rho\) is independent of \(x\) and \(y\), \(\mathscr E_\rho(0)=0\), and the lower-order term vanishes, the pair \((\mathscr E_\rho,0)\) satisfies hypothesis \((\star)\) with parameters \(\beta=L=\Lambda=1\) and perturbation parameter \(\varepsilon_*\).

The conclusion, with \(\varepsilon_0:=\rho\), follows from Theorem~\ref{thm:quasilinear-Schauder} applied to \(\widetilde u\), after multiplying the resulting estimate by \(\rho\).
\end{proof}

\begin{corollary}[Small-slope Bernstein theorem]
\label{cor:Bernstein}
Let \(n\geq2\), \(k\geq1\), and \(Q\geq1\). Suppose that \(0<\alpha<1/Q\) if \(n=2\), and that \(0<\alpha<\delta(n,k,Q)\) if \(n\geq3\). There exists \(\varepsilon_0=\varepsilon_0(n,k,Q,\alpha)>0\) with the following property. If $u\in C^{1,\alpha}_{\mathrm{loc}}\bigl(\mathbb R^n,\mathcal A_Q(\mathbb R^k)\bigr)$, $\sup_{\mathbb R^n}|Du|\leq\varepsilon_0$, and the graph varifold associated with \(u\) is stationary in \(\mathbb R^{n+k}\), then \(u\) is affine.
\end{corollary}

\begin{proof}
For \(R\geq1\), set \(u_R(x):=R^{-1}u(Rx)\). The graph of \(u_R\) is stationary and $\sup_{B_1}|Du_R| = \sup_{B_R}|Du| \leq\varepsilon_0$.
Applying Corollary~\ref{cor:stationary-graph-Schauder} to \(u_R\) and then scaling back, we obtain $[Du]_{\alpha;B_{R/2}} \leq C R^{-n/2-1-\alpha}\|u\|_{L^2(B_R)}$.
The one-dimensional first-jet selection along segments gives $|u(x)|\leq |u(0)|+\varepsilon_0|x|$, and therefore $\|u\|_{L^2(B_R)} \leq C\bigl(|u(0)|R^{n/2}+\varepsilon_0R^{n/2+1}\bigr)$.
Consequently,
\[
[Du]_{\alpha;B_{R/2}}
\leq
C\bigl(|u(0)|R^{-1-\alpha}+\varepsilon_0R^{-\alpha}\bigr)
\longrightarrow0
\]
as \(R\to\infty\). It follows that \(Du\) is constant. Lemma~\ref{lem:one-dimensional-jet-selection}, applied along segments, then shows that \(u\) is affine.
\end{proof}

\subsection{The dimension bound for the branch set of area-stationary maps}
\label{sec:branch-dimension}

\begin{lemma}[Stationarity of graphical differential blocks]
\label{lem:area-stationary-differential-blocks}
Let \(U\subset\mathbb R^n\) be a neighbourhood of \(x_0\), and let $u=\sum_{\gamma=1}^N u^\gamma$, $u^\gamma\in C^1\bigl(U,\mathcal A_{Q_\gamma}(\mathbb R^k)\bigr)$, $\sum_{\gamma=1}^NQ_\gamma=Q$, be a \(C^1\) decomposition. Suppose that, for some \(a\in\mathbb R^k\) and pairwise distinct $A^1,\ldots,A^N \in\operatorname{Hom}(\mathbb R^n,\mathbb R^k)$, one has $j_1u^\gamma(x_0) = Q_\gamma\llbracket(a,A^\gamma)\rrbracket$ for every \(\gamma=1,\ldots,N\).

If the graph varifold \(V:=\mathbf v(u)\) is stationary, then, after
possibly shrinking \(U\) around \(x_0\), every graph varifold $V^{\gamma}:=\mathbf v(u^\gamma)$ is stationary.
\end{lemma}

\begin{proof}
We argue by induction on \(Q\).
The conclusion is immediate if \(Q=1\), or if \(N=1\).
Suppose \(N\geq2\) and that the result holds for every multiplicity strictly smaller than \(Q\).

For each \(\alpha\), set \(m^\alpha:=\eta\circ u^\alpha\). Choose \(\varrho>0\) so that the closed balls \(\overline B_\varrho(A^\alpha)\subset\operatorname{Hom}(\mathbb R^n,\mathbb R^k)\) are pairwise disjoint. Since \(Du^\alpha(x_0)=Q_\alpha\llbracket A^\alpha\rrbracket\) and \(Du^\alpha\) is continuous, after possibly shrinking \(U\) we may assume that \(\operatorname{spt}Du^\alpha(x)\subset B_\varrho(A^\alpha)\) for every \(x\in U\) and every \(\alpha\). In particular, \(Dm^\alpha(x)\in B_\varrho(A^\alpha)\) for every \(x\in U\), and hence \(Dm^\alpha(x)\neq Dm^\beta(x)\) whenever \(\alpha\neq\beta\).

Choose two distinct differential blocks \(u^\alpha\) and \(u^\beta\). Since all the values of \(u\) coincide at \(x_0\), one has $m^\alpha(x_0)=m^\beta(x_0)$, whereas $Dm^\alpha(x_0)=A^\alpha\neq A^\beta=Dm^\beta(x_0)$.
Hence there exist \(e\in\mathbb R^k\) and \(\nu\in\mathbb R^n\) such that $e\cdot(A^\alpha-A^\beta)\nu\neq0$.
Set $f(x):=e\cdot\bigl(m^\alpha(x)-m^\beta(x)\bigr)$.
Then $f(x_0)=0$, $D_\nu f(x_0)=e\cdot(A^\alpha-A^\beta)\nu\neq0$.
After shrinking \(U\) if necessary, we may therefore assume that
\(Df\neq0\) throughout \(U\), and hence $H:=\{x\in U:f(x)=0\}$ is a \(C^1\) hypersurface.

Let $C_u:=
\left\{
x\in U:
u(x)=Q\llbracket a\rrbracket
\text{ for some }a\in\mathbb R^k
\right\}$.
Then \(C_u\subset H\), since at every point of \(C_u\) all the value
blocks, and therefore their averages, coincide.

Consequently, at every point of \(U\setminus H\), the values of \(u\)
split into at least two disjoint clusters. The clusters remain uniformly
separated in a sufficiently small neighbourhood, and, by choosing
variations supported near the graph of only one cluster, stationarity
localizes to each cluster separately. Since every value cluster has multiplicity strictly smaller than \(Q\), the inductive hypothesis applies to its differential sub-blocks. By the preceding separation of the differential sheets, every such sub-block belongs to a unique prescribed component \(u^\gamma\); more precisely, \(u^\gamma\) is locally the sum of those differential sub-blocks whose differential sheets lie in \(B_\varrho(A^\gamma)\). Since first variation is additive over the corresponding graph varifolds, summing these stationary sub-blocks shows that every \(V^\gamma\) is stationary in \(U\setminus H\).

It is then enough to show that stationarity extends across \(H\). Fix \(\gamma\in\{1,\ldots,N\}\). Let \(W^\pm\) be the two sides of \(H\) in a sufficiently small neighbourhood \(W\Subset U\), and fix \(Y\in C_c^1(W\times\mathbb R^k,\mathbb R^{n+k})\). Since \(W^\pm\subset U\setminus H\), the preceding argument shows that \(V^\gamma\) is stationary over both \(W^+\) and \(W^-\).

Let \(\rho\in C^1(W)\) be a defining function for \(H\), with \(W^\pm=\{\pm\rho>0\}\), and set $\eta_\varepsilon^\pm(x):=\chi\left(\frac{\pm\rho(x)}{\varepsilon}\right)$, where \(\chi\in C^\infty(\mathbb R)\) is nondecreasing, \(\chi=0\) on \((-\infty,\frac12]\), and \(\chi=1\) on \([1,\infty)\).
Since \(\eta_\varepsilon^\pm\) vanishes near \(H\), \((\eta_\varepsilon^\pm\circ\pi)Y\) is an admissible variation over \(W^\pm\), where \(\pi(x,y)=x\). Thus
\begin{equation}
\label{eq:stationarity-identity}
\begin{aligned}
0&=\int_{\pi^{-1}(W^\pm)} \operatorname{div}_S\big((\eta_\varepsilon^\pm\circ\pi)Y\big)\,dV^\gamma\\
&= \int_{\pi^{-1}(W^\pm)}(\eta_\varepsilon^\pm\circ\pi)\operatorname{div}_S Y\,dV^\gamma+\int_{\pi^{-1}(W^\pm)}Y\cdot\nabla_S(\eta_\varepsilon^\pm\circ\pi)\,dV^\gamma.
\end{aligned}
\end{equation}
Indeed, write pointwise $j_1u^\gamma(x) = \sum_{\ell=1}^{Q_\gamma} \llbracket(u_\ell(x),Du_\ell(x))\rrbracket$ and set \(S_\ell:=(\operatorname{Id},Du_\ell(x))[\mathbb R^n]\), \(\tau_{\ell,i}:=(e_i,D_i u_\ell(x))\), and \(g_{\ell,ij}:=\tau_{\ell,i}\cdot\tau_{\ell,j}\), denoting by \(g_\ell^{ij}\) the entries of the inverse matrix.
Since $\nabla_{S_\ell}(\eta_\varepsilon^\pm\circ\pi)(x,u_\ell(x))=\sum_{i,j=1}^n g_\ell^{ij}D_i\eta_\varepsilon^\pm(x)\,\tau_{\ell,j}$, the graphical representation of \(V^\gamma\) gives
\[
\int_{\pi^{-1}(W^\pm)}Y\cdot\nabla_S(\eta_\varepsilon^\pm\circ\pi)\,dV^\gamma=\int_{W^\pm}\sum_{\ell=1}^{Q_\gamma}\sum_{i,j=1}^n \sqrt{\det(g_{\ell,ij})}\,g_\ell^{ij}D_i\eta_\varepsilon^\pm\,Y(x,u_\ell(x))\cdot\tau_{\ell,j}\,dx.
\]
Moreover, by the coarea formula,
\[
\int_W\varphi\cdot D\eta_\varepsilon^\pm\,dx=\int_{1/2}^1\chi'(t)\int_{\{\pm\rho=\varepsilon t\}}\varphi\cdot\frac{D(\pm\rho)}{|D\rho|}\,d\mathcal H^{n-1}\,dt\longrightarrow-\int_H\varphi\cdot\nu^\pm\,d\mathcal H^{n-1}
\]
for every \(\varphi\in C_c(W,\mathbb R^n)\), where \(\nu^\pm\) is the outward-pointing normal to \(W^\pm\). Applying this to the preceding continuous symmetric expression in \(j_1u^\gamma\), we obtain
\[
\int_{\pi^{-1}(W^\pm)}Y\cdot\nabla_S(\eta_\varepsilon^\pm\circ\pi)\,dV^\gamma\longrightarrow-\int_H\sum_{\ell=1}^{Q_\gamma}\sum_{i,j=1}^n \sqrt{\det(g_{\ell,ij})}\,g_\ell^{ij}(\nu^\pm)_iY(x,u_\ell(x))\cdot\tau_{\ell,j}\,d\mathcal H^{n-1}.
\]
Since \(\eta_\varepsilon^\pm\to1\) pointwise in \(W^\pm\), the first term on the right-hand side in the stationarity identity \eqref{eq:stationarity-identity} converges by dominated convergence. Hence, since \(\nu^-=-\nu^+\),
\[
\delta V^\gamma(Y)=\int_{\pi^{-1}(W^+)}\operatorname{div}_{S}Y\,dV^\gamma+\int_{\pi^{-1}(W^-)}\operatorname{div}_{S}Y\,dV^\gamma=0.
\]
Since \(W\) can be chosen around any point of \(H\), and stationarity is a local property, this shows that \(V^\gamma\) is stationary across \(H\). Together with its stationarity in \(U\setminus H\), it follows that \(V^\gamma\) is stationary throughout \(U\). Since \(\gamma\) was arbitrary, this proves the lemma.

\end{proof}

\begin{lemma}[Pointwise transformation of first jets under graphical reparametrization]
\label{lem:pointwise-first-jet-reparametrization}
Let \(u\in C^1\bigl(B_2^n,\mathcal A_Q(\mathbb R^k)\bigr)\), let \(R=\begin{pmatrix}R_{11}&R_{12}\\R_{21}&R_{22}\end{pmatrix}\in O(n+k)\), and let \(\widetilde u\in\operatorname{Lip}\bigl(B_1^n,\mathcal A_Q(\mathbb R^k)\bigr)\) satisfy
\begin{equation}
\label{eq:pointwise-reparametrization-graphical-identity}
\mathbf v(\widetilde u)
=
\bigl(R_\#\mathbf v(u)\bigr)
\mres
\bigl(B_1^n\times\mathbb R^k\bigr).
\end{equation}
Suppose that there exist \(0<c\leq C<\infty\) such that, for every \(z\in B_2^n\), every differential atom \(B\) of \(Du(z)\), and every \(h\in\mathbb R^n\),
\begin{equation}
\label{eq:pointwise-reparametrization-transversality}
c|h|
\leq
|(R_{11}+R_{12}B)h|
\leq
C|h|.
\end{equation}

Let \(\widetilde x_0\in B_1^n\), and write \(\widetilde u(\widetilde x_0)=\sum_{\beta=1}^M q_\beta\llbracket\widetilde a_0^\beta\rrbracket\), where the \(\widetilde a_0^\beta\) are pairwise distinct. For each \(\beta\), set \(\begin{pmatrix}x_0^\beta\\a^\beta\end{pmatrix}:=R^{-1}\begin{pmatrix}\widetilde x_0\\\widetilde a_0^\beta\end{pmatrix}\). Assume that \(x_0^\beta\in B_2^n\) for every \(\beta\)\footnote{Note that the varifold identity \eqref{eq:pointwise-reparametrization-graphical-identity} alone only implies that \(x_0^\beta\in\overline B_2^n\).}. For each \(\beta\), write the part of \(j_1u(x_0^\beta)\) whose value component is \(a^\beta\) as \(\sum_{\gamma=1}^{N_\beta}q_{\beta,\gamma}\llbracket(a^\beta,A^{\beta,\gamma})\rrbracket\), where the \(A^{\beta,\gamma}\) are pairwise distinct.
Then, for every \(\beta\), \(\sum_{\gamma=1}^{N_\beta}q_{\beta,\gamma}=q_\beta\), \(\widetilde u\) is differentiable at \(\widetilde x_0\), and
\begin{equation}
\label{eq:reparametrized-first-jet-transformation}
j_1\widetilde u(\widetilde x_0)
=
\sum_{\beta=1}^M
\sum_{\gamma=1}^{N_\beta}
q_{\beta,\gamma}
\llbracket
\left(
\widetilde a_0^\beta,
(R_{21}+R_{22}A^{\beta,\gamma})
(R_{11}+R_{12}A^{\beta,\gamma})^{-1}
\right)
\rrbracket.
\end{equation}
\end{lemma}

\begin{proof}
Let \(\pi:\mathbb R^{n+k}\to\mathbb R^n\) denote the horizontal projection. Following \cite[Definition~1.10]{DLS15}, we denote by \(\mathbf G_w\) the graph current associated with a multiple-valued map \(w\), with its canonical graphical orientation.

Since \(\widetilde u\) is Lipschitz, the usual value-splitting argument for \(Q\)-valued Lipschitz maps \cite[Proposition~1.6]{DLS11} gives, after shrinking to a neighbourhood of \(\widetilde x_0\), a local decomposition \(\widetilde u=\sum_{\beta=1}^M\widetilde u^\beta\), with \(\widetilde u^\beta(\widetilde x_0)=q_\beta\llbracket\widetilde a_0^\beta\rrbracket\). By assumption, \(x_0^\beta\in B_2^n\). The assumed varifold identity \eqref{eq:pointwise-reparametrization-graphical-identity} and the continuity of \(u\) then give \(a^\beta\in\spt u(x_0^\beta)\).
By Proposition~\ref{prop:local-splitting}, for each \(\beta\) there exists \(r>0\) such that, in \(B_r^n(x_0^\beta)\), the part of \(u\) corresponding to \(a^\beta\) decomposes as \(\sum_{\gamma=1}^{N_\beta}u^{\beta,\gamma}\), where each \(u^{\beta,\gamma}:B_r^n(x_0^\beta)\to\mathcal A_{q_{\beta,\gamma}}(\mathbb R^k)\) satisfies
\(j_1u^{\beta,\gamma}(x_0^\beta) =q_{\beta,\gamma} \llbracket (a^\beta,A^{\beta,\gamma}) \rrbracket\).

By \eqref{eq:pointwise-reparametrization-transversality}, the matrices \(R_{11}+R_{12}B\) are uniformly invertible, and hence the rotated planes \(R(P_B)\) are graphs over the horizontal plane, with locally uniformly bounded slopes. Set \(L_{\beta,\gamma}:=R_{11}+R_{12}A^{\beta,\gamma}\) and \(\sigma_{\beta,\gamma}:=\operatorname{sign}\det L_{\beta,\gamma}\).

After decreasing \(r\) if necessary, differentiability at \(x_0^\beta\)  gives \(\lvert b-a^\beta-A^{\beta,\gamma}h\rvert\leq\frac c2|h|\) whenever \(h\in B_r^n\) and \(b\in\spt u^{\beta,\gamma}(x_0^\beta+h)\). Set \(T_r^{\beta,\gamma}:=\mathbf G_{u^{\beta,\gamma}}\) and \(\widetilde T_r^{\beta,\gamma}:=R_\#T_r^{\beta,\gamma}\). By the boundary formula for multiple-valued graph currents \cite[Theorem~2.1]{DLS15}, \(\spt\partial T_r^{\beta,\gamma}\subset\{(x_0^\beta+h,b):|h|=r,\ b\in\spt u^{\beta,\gamma}(x_0^\beta+h)\}\).
Since
\(\pi(R(x_0^\beta+h,b))-\widetilde x_0=L_{\beta,\gamma}h+R_{12}(b-a^\beta-A^{\beta,\gamma}h)\),
\eqref{eq:pointwise-reparametrization-transversality} gives
\(\bigl|\pi(R(x_0^\beta+h,b))-\widetilde x_0\bigr|\geq\frac c2|h|\).
Consequently,
\(\spt\partial\widetilde T_r^{\beta,\gamma}\cap\bigl(B_{cr/2}^n(\widetilde x_0)\times\mathbb R^k\bigr)=\varnothing\),
and hence \(\partial\widetilde T_r^{\beta,\gamma}=0\) in \(B_{cr/2}^n(\widetilde x_0)\times\mathbb R^k\). The same estimate also gives
\(\spt\widetilde T_r^{\beta,\gamma}\cap\pi^{-1}(\widetilde x_0)=\{(\widetilde x_0,\widetilde a_0^\beta)\}\).

Together with orthogonality of \(R\), the preceding estimates give
\(\bigl|R(x_0^\beta+h,b)-(\widetilde x_0,\widetilde a_0^\beta)\bigr|\leq C_0\bigl|\pi(R(x_0^\beta+h,b))-\widetilde x_0\bigr|\)
on \(\spt\widetilde T_r^{\beta,\gamma}\), for some \(C_0<\infty\). The constancy theorem therefore gives
\((\pi_\#\widetilde T_r^{\beta,\gamma})\mres B_{cr/2}^n(\widetilde x_0)=m_{\beta,\gamma}\llbracket B_{cr/2}^n(\widetilde x_0)\rrbracket\)
for some \(m_{\beta,\gamma}\in\mathbb Z\). The tangent current to \(\widetilde T_r^{\beta,\gamma}\) at \((\widetilde x_0,\widetilde a_0^\beta)\) is \(q_{\beta,\gamma}R_\#\llbracket P_{A^{\beta,\gamma}}\rrbracket\), whose horizontal push-forward is \(\sigma_{\beta,\gamma}q_{\beta,\gamma}\llbracket\mathbb R^n\rrbracket\). Comparison with the preceding identity gives \(m_{\beta,\gamma}=\sigma_{\beta,\gamma}q_{\beta,\gamma}\).

By continuity of the first jet, the radius \(r\) may also be chosen so that
\(\operatorname{sign}\det(R_{11}+R_{12}B)=\sigma_{\beta,\gamma}\)
for every \(z\in B_r^n(x_0^\beta)\) and every differential atom \(B\) of \(Du^{\beta,\gamma}(z)\). Therefore, after replacing \(\widetilde T_r^{\beta,\gamma}\) by \(\sigma_{\beta,\gamma}\widetilde T_r^{\beta,\gamma}\), the horizontal projection preserves orientation on all its approximate tangent planes and its horizontal push-forward has multiplicity \(q_{\beta,\gamma}\) (and obviously this replacement does not change the associated varifold). Moreover, the uniform slope bound and the area formula give the mass bounds required to apply \cite[Lemma~5.5]{DLS15} to the one-dimensional slices. Together with the preceding boundary and multiplicity conclusions, these facts allow us to apply the existence argument in Part~I of the proof of \cite[Theorem~5.1]{DLS15}, obtaining \(\rho>0\) and Lipschitz maps \(\widetilde u^{\beta,\gamma}:B_\rho^n(\widetilde x_0)\to\mathcal A_{q_{\beta,\gamma}}(\mathbb R^k)\), with \(\widetilde u^{\beta,\gamma}(\widetilde x_0)=q_{\beta,\gamma}\llbracket\widetilde a_0^\beta\rrbracket\), whose graph currents satisfy
\(\mathbf G_{\widetilde u^{\beta,\gamma}}=\sigma_{\beta,\gamma}\widetilde T_r^{\beta,\gamma}\)
in \(B_\rho^n(\widetilde x_0)\times\mathbb R^k\). In particular,
\begin{equation}
\label{eq:local-reparametrization-graphical-identity}
\mathbf v(\widetilde u^{\beta,\gamma})
=
\bigl(R_\#\mathbf v(u^{\beta,\gamma})\bigr)
\mres
\bigl(B_\rho^n(\widetilde x_0)\times\mathbb R^k\bigr).
\end{equation}

Since there are only finitely many pairs \((\beta,\gamma)\), after restricting the maps if necessary, the same radius \(\rho>0\) may be used for all of them. After decreasing \(\rho\) further if necessary, \eqref{eq:pointwise-reparametrization-graphical-identity} and \eqref{eq:local-reparametrization-graphical-identity}, together with uniqueness of the fibrewise graphical representation, show that \(\widetilde u^\beta=\sum_{\gamma=1}^{N_\beta}\widetilde u^{\beta,\gamma}\) almost everywhere in \(B_\rho^n(\widetilde x_0)\), and hence throughout \(B_\rho^n(\widetilde x_0)\) by continuity. Consequently, \(\widetilde u=\sum_{\beta=1}^M\sum_{\gamma=1}^{N_\beta}\widetilde u^{\beta,\gamma}\) throughout \(B_\rho^n(\widetilde x_0)\). In particular, \(\sum_{\gamma=1}^{N_\beta}q_{\beta,\gamma}=q_\beta\).

We now compute the first jet of each \(\widetilde u^{\beta,\gamma}\). Let \(\widetilde h\to0\), and let \(\widetilde b\in\spt\widetilde u^{\beta,\gamma}(\widetilde x_0+\widetilde h)\). By \eqref{eq:local-reparametrization-graphical-identity}, there exist \(h\in\mathbb R^n\) and \(b\in\spt u^{\beta,\gamma}(x_0^\beta+h)\) such that
\(\begin{pmatrix}\widetilde x_0+\widetilde h\\\widetilde b\end{pmatrix}=R\begin{pmatrix}x_0^\beta+h\\b\end{pmatrix}\).
Subtracting
\(\begin{pmatrix}\widetilde x_0\\\widetilde a_0^\beta\end{pmatrix}=R\begin{pmatrix}x_0^\beta\\a^\beta\end{pmatrix}\)
gives \(\widetilde h=R_{11}h+R_{12}(b-a^\beta)\) and \(\widetilde b-\widetilde a_0^\beta=R_{21}h+R_{22}(b-a^\beta)\).

Differentiability of \(u^{\beta,\gamma}\) at \(x_0^\beta\) gives, uniformly over \(b\in\spt u^{\beta,\gamma}(x_0^\beta+h)\), \(b-a^\beta=A^{\beta,\gamma}h+o(|h|)\). Consequently, \(\widetilde h=(R_{11}+R_{12}A^{\beta,\gamma})h+o(|h|)\) and \(\widetilde b-\widetilde a_0^\beta=(R_{21}+R_{22}A^{\beta,\gamma})h+o(|h|)\). By \eqref{eq:pointwise-reparametrization-transversality}, after decreasing \(\rho\) if necessary,
\begin{equation}
\label{eq:reparametrization-bilipschitz}
\frac c2|h|
\leq
|\widetilde h|
\leq
2C|h|.
\end{equation}
It follows that \(h=(R_{11}+R_{12}A^{\beta,\gamma})^{-1}\widetilde h+o(|\widetilde h|)\), and therefore
\(\widetilde b-\widetilde a_0^\beta=(R_{21}+R_{22}A^{\beta,\gamma})(R_{11}+R_{12}A^{\beta,\gamma})^{-1}\widetilde h+o(|\widetilde h|)\).
Since this holds for all \(q_{\beta,\gamma}\) values of \(\widetilde u^{\beta,\gamma}(\widetilde x_0+\widetilde h)\), counted with multiplicity,
\(\mathcal G\bigl(\widetilde u^{\beta,\gamma}(\widetilde x_0+\widetilde h),q_{\beta,\gamma}\llbracket\widetilde a_0^\beta+(R_{21}+R_{22}A^{\beta,\gamma})(R_{11}+R_{12}A^{\beta,\gamma})^{-1}\widetilde h\rrbracket\bigr)=o(|\widetilde h|)\).
Thus \(\widetilde u^{\beta,\gamma}\) is differentiable at \(\widetilde x_0\), with
\(j_1\widetilde u^{\beta,\gamma}(\widetilde x_0)=q_{\beta,\gamma}\llbracket\bigl(\widetilde a_0^\beta,(R_{21}+R_{22}A^{\beta,\gamma})(R_{11}+R_{12}A^{\beta,\gamma})^{-1}\bigr)\rrbracket\).

Summing over \(\beta,\gamma\), we conclude that \(\widetilde u\) is differentiable at \(\widetilde x_0\) and that \eqref{eq:reparametrized-first-jet-transformation} holds.
\end{proof}

\begin{lemma}[\(C^{1,\alpha}\) reparametrization over nearby planes]
\label{lem:C1alpha-reparametrization}
There exist \(\varepsilon=\varepsilon(n,k,Q)>0\) and \(C=C(n,k,Q)<\infty\) with the following property.

Let \(u\in C^1\bigl(B_2^n,\mathcal A_Q(\mathbb R^k)\bigr)\) and \(R\in O(n+k)\) satisfy \(\|u\|_{C^0(B^n_2)}+\|Du\|_{C^0(B^n_2)}\leq\varepsilon\) and \(\|R-I\|\leq\varepsilon\). Then there exists a map \(\widetilde u\in C^1\bigl(B_1^n,\mathcal A_Q(\mathbb R^k)\bigr)\) such that \(\mathbf v(\widetilde u)=\bigl(R_\#\mathbf v(u)\bigr)\mres\bigl(B_1^n\times\mathbb R^k\bigr)\) and \(\|D\widetilde u\|_{C^0(B^n_1)}\leq C\bigl(\|R-I\|+\|Du\|_{C^0(B^n_2)}\bigr)\).

Moreover, if \(u\in C^{1,\alpha}\bigl(B_2^n,\mathcal A_Q(\mathbb R^k)\bigr)\) for some \(0<\alpha<1\), then \(\widetilde u\in C^{1,\alpha}\bigl(B_1^n,\mathcal A_Q(\mathbb R^k)\bigr)\) and \([D\widetilde u]_{\alpha;B^n_1}\leq C[Du]_{\alpha;B^n_2}\).
\end{lemma}

\begin{proof}
Lemma~\ref{lem:one-dimensional-jet-selection} gives \(\operatorname{Lip}(u)\leq\|Du\|_{C^0(B_2)}\). Set \(u_2(z):=\frac12u(2z)\), \(z\in B_1^n\). Since the ambient dilation commutes with \(R\), it is enough to apply \cite[Theorem~5.1]{DLS15} to \(u_2\), with fixed radii \(s:=\frac12\) and \(r:=\frac34\), and with \(\varphi:B_s^n\to\mathbb R^k\) linear such that \(\Sigma:=R^{-1}\bigl(\mathbb R^n\times\{0\}\bigr)\cap(B_s^n\times\mathbb R^k)=\operatorname{graph}\varphi\). For \(\varepsilon\) sufficiently small depending only on \(n,k,Q\), the assumed smallness of \(u\), \(Du\), and \(R-I\) gives the hypotheses of that theorem. Scaling the resulting reparametrization back, applying \(R\), and restricting to \(B_1^n\), we obtain a Lipschitz \(Q\)-valued map \(\widetilde u:B_1^n\to\mathcal A_Q(\mathbb R^k)\) such that
\begin{equation}
\label{eq:reparametrization-global-graphical-identity}
\mathbf v(\widetilde u)
=
\bigl(R_\#\mathbf v(u)\bigr)
\mres
\bigl(B_1^n\times\mathbb R^k\bigr).
\end{equation}

Since \(u\) is Lipschitz, it extends continuously to \(\overline B_2^n\). For \(x\in\partial B_2^n\) and \(a\in\spt u(x)\), the smallness assumptions give \(|R_{11}x+R_{12}a|\geq|x|-\|R_{11}-I\||x|-\|R_{12}\||a|\geq2-\varepsilon(2+\varepsilon)>1\), after decreasing \(\varepsilon\) if necessary. Consequently, \eqref{eq:reparametrization-global-graphical-identity} implies that every preimage under \(R\) of a point of the graph of \(\widetilde u\) lies in \(B_2^n\times\mathbb R^k\).
If \(B\) is a differential atom of \(Du(z)\), then \(|B|\leq\varepsilon\) and \(\bigl|(R_{11}+R_{12}B)h-h\bigr|\leq\bigl(\|R_{11}-I\|+\|R_{12}\||B|\bigr)|h|\leq2\varepsilon|h|\). Thus, after decreasing \(\varepsilon\) if necessary, \(\frac12|h|\leq|(R_{11}+R_{12}B)h|\leq2|h|\) for every \(h\in\mathbb R^n\). Lemma~\ref{lem:pointwise-first-jet-reparametrization} therefore shows that \(\widetilde u\) is differentiable at every point of \(B_1^n\) and that the first-jet transformation formula \eqref{eq:reparametrized-first-jet-transformation} holds.

We next prove that \(j_1\widetilde u\) is continuous. Fix \(\widetilde x_0\in B_1^n\), and retain the local submaps \(u^{\beta,\gamma}\) and \(\widetilde u^{\beta,\gamma}\) constructed in the proof of Lemma~\ref{lem:pointwise-first-jet-reparametrization}. Thus, for some \(\rho>0\), \(\widetilde u=\sum_{\beta,\gamma}\widetilde u^{\beta,\gamma}\) in \(B_\rho^n(\widetilde x_0)\), and each \(\widetilde u^{\beta,\gamma}\) has a collapsed first jet at \(\widetilde x_0\).

Let \(\widetilde x\to\widetilde x_0\), and let \((\widetilde a,\widetilde A)\) be an atom of \(j_1\widetilde u^{\beta,\gamma}(\widetilde x)\). By the local graphical identity \eqref{eq:local-reparametrization-graphical-identity}, there exist \(x\in B_2^n\) and \(a\in\spt u^{\beta,\gamma}(x)\) such that \(\begin{pmatrix}\widetilde x\\\widetilde a\end{pmatrix}=R\begin{pmatrix}x\\a\end{pmatrix}\). Applying \eqref{eq:reparametrized-first-jet-transformation} to the corresponding local pair, there is an atom \(A\) of \(Du^{\beta,\gamma}(x)\) such that \(\widetilde A=(R_{21}+R_{22}A)(R_{11}+R_{12}A)^{-1}\). Moreover, \eqref{eq:reparametrization-bilipschitz} gives \(|x-x_0^\beta|\leq C|\widetilde x-\widetilde x_0|\).

Since \(j_1u^{\beta,\gamma}\) is continuous and \(j_1u^{\beta,\gamma}(x_0^\beta)=q_{\beta,\gamma}\llbracket(a^\beta,A^{\beta,\gamma})\rrbracket\), every such atom \((a,A)\) converges to \((a^\beta,A^{\beta,\gamma})\) as \(\widetilde x\to\widetilde x_0\). The map \(A\mapsto(R_{21}+R_{22}A)(R_{11}+R_{12}A)^{-1}\) is smooth in the range under consideration. Hence every atom of \(j_1\widetilde u^{\beta,\gamma}(\widetilde x)\) converges to \(\bigl(\widetilde a_0^\beta,(R_{21}+R_{22}A^{\beta,\gamma})(R_{11}+R_{12}A^{\beta,\gamma})^{-1}\bigr)\). Thus \(j_1\widetilde u^{\beta,\gamma}\) is continuous at \(\widetilde x_0\). Summing the local components and using that \(\widetilde x_0\) was arbitrary, we conclude that \(\widetilde u\in C^1(B_1^n,\mathcal A_Q(\mathbb R^k))\).

For every differential atom \(A\) of \(Du\), the smallness assumptions imply \(\bigl|(R_{21}+R_{22}A)(R_{11}+R_{12}A)^{-1}\bigr|\leq C\bigl(\|R-I\|+|A|\bigr)\). Together with \eqref{eq:reparametrized-first-jet-transformation}, this gives
\begin{equation}
\label{eq:reparametrization-gradient-bound}
\|D\widetilde u\|_{C^0(B_1)}
\leq
C\bigl(\|R-I\|+\|Du\|_{C^0(B_2)}\bigr).
\end{equation}

Assume now that \(u\in C^{1,\alpha}(B_2^n,\mathcal A_Q(\mathbb R^k))\). Let \(\widetilde x,\widetilde y\in B_1^n\), and let \(\sigma(t):=(1-t)\widetilde x+t\widetilde y\). By Lemma~\ref{lem:one-dimensional-jet-selection}, the restriction of \(j_1\widetilde u\) to \(\sigma\) admits continuous selections \(j_1\widetilde u(\sigma(t))=\sum_{i=1}^Q\llbracket(\widetilde a_i(t),\widetilde A_i(t))\rrbracket\), where each \(\widetilde a_i\) is \(C^1\) and \(\dot{\widetilde a}_i(t)=\widetilde A_i(t)(\widetilde y-\widetilde x)\).

Define \(\begin{pmatrix}x_i(t)\\a_i(t)\end{pmatrix}:=R^{-1}\begin{pmatrix}\sigma(t)\\\widetilde a_i(t)\end{pmatrix}\). By \eqref{eq:reparametrization-global-graphical-identity}, \((x_i(t),a_i(t))\) belongs to the graph of \(u\), and \(x_i(t)\in B_2^n\). Set \(A_i(t):=(R_{12}^{T}+R_{22}^{T}\widetilde A_i(t))(R_{11}^{T}+R_{21}^{T}\widetilde A_i(t))^{-1}\). For \(\varepsilon\) sufficiently small, the inverse is well defined, and \eqref{eq:reparametrized-first-jet-transformation} shows that \((a_i(t),A_i(t))\) is an atom of \(j_1u(x_i(t))\). Moreover, \(x_i'(t)=(R_{11}^{T}+R_{21}^{T}\widetilde A_i(t))(\widetilde y-\widetilde x)\). By \eqref{eq:reparametrization-gradient-bound}, the paths \(x_i\) therefore satisfy \(|x_i(t)-x_i(s)|\leq C|t-s|\,|\widetilde y-\widetilde x|\).

The Hölder continuity of \(Du\) gives \(\mathcal G(Du(x_i(t)),Du(x_i(s)))\leq C[Du]_{\alpha;B_2}|t-s|^\alpha|\widetilde y-\widetilde x|^\alpha\). Since \(A_i(t)\) is a continuous selection of the differential atoms of \(Du(x_i(t))\), the usual chain-of-balls argument yields \(|A_i(1)-A_i(0)|\leq C[Du]_{\alpha;B_2}|\widetilde y-\widetilde x|^\alpha\). Using once more the smooth dependence of \(A\mapsto(R_{21}+R_{22}A)(R_{11}+R_{12}A)^{-1}\), we obtain \(|\widetilde A_i(1)-\widetilde A_i(0)|\leq C[Du]_{\alpha;B_2}|\widetilde y-\widetilde x|^\alpha\). The selections at \(t=0\) and \(t=1\) provide a matching between the atoms of \(D\widetilde u(\widetilde x)\) and \(D\widetilde u(\widetilde y)\); hence \([D\widetilde u]_{\alpha;B_1}\leq C[Du]_{\alpha;B_2}\).
\end{proof}

\begin{lemma}[Reparametrization at a collapsed first jet]
\label{lem:collapsed-jet-reparametrization}
Let \(0<\alpha<1\), let \(u\in C^{1,\alpha}_{\mathrm{loc}}\bigl(B_2^n, \mathcal A_Q(\mathbb R^k)\bigr)\), and suppose that \(j_1u(0)=Q\llbracket(0,A)\rrbracket\). Let \(R\in O(n+k)\) satisfy \(R(P_A)=\mathbb R^n\times\{0\}\), where \(P_A:=\{(x,Ax):x\in\mathbb R^n\}\). Then there exist \(\rho>0\) and \(\widetilde u\in C^{1,\alpha}\bigl(B_{2\rho}^n,\mathcal A_Q(\mathbb R^k)\bigr)\) such that \(\mathbf v(\widetilde u)\mres B_\rho^{n+k}=\bigl(R_\#\mathbf v(u)\bigr)\mres B_\rho^{n+k}\), and \(\widetilde u(0)=Q\llbracket0\rrbracket\), \(D\widetilde u(0)=Q\llbracket0\rrbracket\).
\end{lemma}

\begin{proof}
We follow the current-theoretic reparametrization scheme of De Lellis--Spadaro \cite{DLS15}, with some simplifications afforded by the \(C^1\) assumption.

Let \(\pi:\mathbb R^{n+k}\to\mathbb R^n\) denote the horizontal projection, write \(R=\begin{pmatrix}R_{11}&R_{12}\\R_{21}&R_{22}\end{pmatrix}\), and set \(L:=R_{11}+R_{12}A:\mathbb R^n\to\mathbb R^n\). Since \(R(x,Ax)=(Lx,0)\), the matrix \(L\) is invertible and \(R_{21}+R_{22}A=0\). Moreover, since \(R\) is orthogonal,
\(|Lh|=|(h,Ah)|\geq|h|\) for every \(h\in\mathbb R^n\).

If \(B\) is a differential atom of \(Du(z)\), then, under the graphical parametrization \(h\mapsto R(h,Bh)\) of \(R(P_B)\), the horizontal projection is \(h\mapsto(R_{11}+R_{12}B)h\). By continuity of the first jet, we may choose \(s\in(0,1)\) sufficiently small and \(C<\infty\) such that, for every differential atom \(B\) of \(Du(z)\), \(z\in B_s^n\), and every \(h\in\mathbb R^n\),
\begin{equation}
\label{eq:collapsed-jet-horizontal-transversality}
C^{-1}|h|
\leq
|(R_{11}+R_{12}B)h|
\leq
C|h|.
\end{equation}
Since \(R(P_B)=\{\bigl((R_{11}+R_{12}B)x,(R_{21}+R_{22}B)x\bigr):x\in\mathbb R^n\}\), it follows, after increasing \(C\) if necessary, that
\begin{equation}
\label{eq:collapsed-jet-slope-representation}
R(P_B)
=
\{\bigl(y,(R_{21}+R_{22}B)(R_{11}+R_{12}B)^{-1}y\bigr):y\in\mathbb R^n\},
\qquad
\bigl|(R_{21}+R_{22}B)(R_{11}+R_{12}B)^{-1}\bigr|
\leq
C.
\end{equation}

After decreasing \(s\) if necessary, we may apply the current-theoretic construction in the proof of Lemma~\ref{lem:pointwise-first-jet-reparametrization} to \(u\) at the origin, with \(q=Q\), \(a=0\), and \(B=A\). Its transversality hypothesis \eqref{eq:pointwise-reparametrization-transversality} follows from \eqref{eq:collapsed-jet-horizontal-transversality}. More precisely, since \(|Lh|\geq|h|\) and \(\|R_{12}\|\leq1\), continuity of the first jet allows us to assume that \(|B-A|\leq\frac12\), and hence
\(\lvert(R_{11}+R_{12}B)h\rvert\geq|Lh|-\lvert R_{12}(B-A)h\rvert\geq\frac12|h|\).
Thus the constant \(c\) in \eqref{eq:pointwise-reparametrization-transversality} may be taken to be \(\frac12\). The construction therefore yields \(\rho>0\), with \(B_{2\rho}^n\Subset B_{s/2}^n\), and a Lipschitz map \(\widetilde u:B_{2\rho}^n\to\mathcal A_Q(\mathbb R^k)\), with \(\widetilde u(0)=Q\llbracket0\rrbracket\), such that
\begin{equation}
\label{eq:collapsed-jet-graphical-identity}
\mathbf v(\widetilde u)
=
\bigl(R_\#\mathbf v(u|_{B_s^n})\bigr)
\mres
\bigl(B_{2\rho}^n\times\mathbb R^k\bigr).
\end{equation}

The bound \(|B-A|\leq\frac12\) and Lemma~\ref{lem:one-dimensional-jet-selection} imply that \(|b-Ax|\leq\frac12|x|\) for every \(x\in\overline B_s^n\) and \(b\in\spt u(x)\). Hence, for \(|x|=s\), we have \(|R_{11}x+R_{12}b|=|Lx+R_{12}(b-Ax)|\geq|Lx|-\|R_{12}\||b-Ax|\geq s/2>2\rho\). Together with \eqref{eq:collapsed-jet-graphical-identity}, this shows that every preimage under \(R\) of a point of the graph of \(\widetilde u\) lies in \(B_s^n\times\mathbb R^k\).
After localization and rescaling, Lemma~\ref{lem:pointwise-first-jet-reparametrization} applies at every point of \(B_{2\rho}^n\) by \eqref{eq:collapsed-jet-horizontal-transversality} and \eqref{eq:collapsed-jet-graphical-identity}. It follows that \(\widetilde u\) is differentiable there and that the first-jet transformation formula \eqref{eq:reparametrized-first-jet-transformation} holds.

The continuity and Hölder arguments in the proof of Lemma~\ref{lem:C1alpha-reparametrization} now carry over without any smallness assumption on \(R-I\). Indeed, the required local bilipschitz estimates follow from \eqref{eq:collapsed-jet-horizontal-transversality}, while the map \(B\mapsto(R_{21}+R_{22}B)(R_{11}+R_{12}B)^{-1}\) is smooth, with uniformly bounded derivative, in the range under consideration. Hence, after decreasing \(\rho\) if necessary, \(\widetilde u\in C^{1,\alpha}\bigl(B_{2\rho}^n,\mathcal A_Q(\mathbb R^k)\bigr)\). Since every differential atom of \(Du(0)\) equals \(A\) and \(R_{21}+R_{22}A=0\), the transformation formula gives \(D\widetilde u(0)=Q\llbracket0\rrbracket\). Finally, \(B_{2\rho}^n\Subset B_{s/2}^n\) implies \(\rho<s\), so \eqref{eq:collapsed-jet-graphical-identity} yields
\(\mathbf v(\widetilde u)\mres B_\rho^{n+k}
=\bigl(R_\#\mathbf v(u)\bigr)\mres B_\rho^{n+k}\), as required.
\end{proof}

\begin{theorem}[Dimension bound for the branch set]
\label{thm:branch-dimension}
Let \(n\geq2\), \(k\geq1\), \(Q\geq2\), and let \(\Omega\subset\mathbb R^n\) be open. Suppose that $u\in C^{1,\alpha}_{\mathrm{loc}} \bigl(\Omega,\mathcal A_Q(\mathbb R^k)\bigr)$ for some \(\alpha>0\), and that the graph varifold \(\mathbf v(u)\) is stationary in \(\Omega\times\mathbb R^k\). Then
\[
\dim_{\mathcal H}\mathcal B_u\leq n-2.
\]
\end{theorem}

\begin{proof}
The assertion is local, and we argue by induction on \(Q\).
The case \(Q=1\) is immediate. Fix \(Q\geq2\), assume that the conclusion holds at every strictly smaller multiplicity, and choose \(0<\beta<\alpha\) so that the Schauder estimate for the area system applies simultaneously for every \(q\leq Q\). Thus, we require $\beta<\frac1Q$ if \(n=2\), and $\beta<
\min_{2\leq q\leq Q}\delta(n,k,q)$ if \(n\geq3\).

For \(L>0\), let
\[
\mathcal V_L
:=
\left\{
(\Gamma_\#\mathbf v(w))\mres B_2^{n+k}
\,\middle|\,
\begin{array}{l}
1\leq q\leq Q,\quad \Gamma\in O(n+k),\\
w\in
C^{1,\beta}_{\mathrm{loc}}
\bigl(B_2^n,\mathcal A_q(\mathbb R^k)\bigr),\\
\|Dw\|_{C^0(B_2^n)}\leq L,\\
\mathbf v(w)
\text{ is stationary in }
B_2^n\times\mathbb R^k
\end{array}
\right\}
\cup\{0\}.
\]

We verify that there exists \(L=L(n,k,Q,\beta)>0\) such that the class \(\mathcal V_L\) satisfies the hypotheses needed to apply \cite[Theorem~A]{KMW26} up to multiplicity \(Q\).

The closure under rotations and homotheties required in \cite[Section~1.1]{KMW26} is immediate from the definition of \(\mathcal V_L\).

We now verify the required weak compactness. Let \(V_h\in\mathcal V_L\) satisfy \(\limsup_{h\to\infty}\|V_h\|(B_2^{n+k})<\infty\). Since the \(V_h\) are stationary integral varifolds, after passing to a subsequence they converge as varifolds to some \(V\) in \(B_2^{n+k}\). It remains to show that \(V\in\mathcal V_L\). If \(V_h=0\) for infinitely many \(h\), then \(V=0\in\mathcal V_L\), so we may assume that \(V_h\neq0\) for every \(h\). Write \(V_h=(\Gamma_{h\#}\mathbf v(w_h))\mres B_2^{n+k}\), where \(w_h\) has multiplicity \(q_h\leq Q\). After passing to a subsequence, we may assume that \(q_h=q\) is constant and, by the compactness of \(O(n+k)\), that \(\Gamma_h\to\Gamma\in O(n+k)\). Since \(B_2^{n+k}\) is invariant under rotations and push-forward is continuous under convergence of the rotations, the varifolds \(\widetilde V_h:=(\Gamma_h^{-1})_\#V_h=\mathbf v(w_h)\mres B_2^{n+k}\) converge to \(\widetilde V:=(\Gamma^{-1})_\#V\). It is therefore enough to prove that \(\widetilde V=\mathbf v(w)\mres B_2^{n+k}\) for some admissible \(w\); rotating back then gives \(V=(\Gamma_\#\mathbf v(w))\mres B_2^{n+k}\in\mathcal V_L\). Replacing \(V_h\) and \(V\) by \(\widetilde V_h\) and \(\widetilde V\), respectively, and relabelling, we may then assume that \(\Gamma_h=\operatorname{Id}\) for every \(h\).
Choose a labelling \(w_h(0)=\sum_{\ell=1}^q\llbracket a_{h,\ell}\rrbracket\). After passing to a further subsequence and relabelling, there is \(0\leq q'\leq q\) such that \(a_{h,\ell}\) remains bounded for \(1\leq\ell\leq q'\), while \(|a_{h,\ell}|\to\infty\) for \(q'<\ell\leq q\). The Lipschitz decomposition of \cite[Proposition~1.6]{DLS11} separates these two groups into submaps \(w_h'\) and \(w_h''\). The slope bound implies that the values of \(w_h''\) diverge uniformly on \(B_2^n\), and hence \(\mathbf v(w_h'')\mres B_2^{n+k}=0\) for all sufficiently large \(h\). Moreover, \(\mathbf v(w_h')\) is stationary by localization of the first variation.
If \(q'=0\), then \(V=0\). Otherwise, \(w_h'(0)\) is uniformly bounded and \(\sup_{B_2^n}|Dw_h'|\leq L\), so \(w_h'\) is uniformly bounded on compact subsets of \(B_2^n\).
Taking \(L\leq\min_{1\leq p\leq Q}\varepsilon_0(n,k,p,\beta)\), Corollary~\ref{cor:stationary-graph-Schauder} gives uniform \(C^{1,\beta}\) bounds on every compact subset of \(B_2^n\). Lemma~\ref{lem:C1-compactness} and a diagonal argument yield a map \(w\in C^{1,\beta}_{\mathrm{loc}}\bigl(B_2^n,\mathcal A_{q'}(\mathbb R^k)\bigr)\) such that \(w_h'\to w\) locally in \(C^1\). The slope bound and stationarity pass to the limit, and the varifold convergence gives \(V=\mathbf v(w)\mres B_2^{n+k}\). Undoing the original rotation, the original limit belongs to \(\mathcal V_L\).

Now we verify that the local disjoint-decomposition property is also satisfied. Suppose without loss of generality that \(V=\mathbf v(w)\mres B_2^{n+k}\in\mathcal V_L\), with \(w\in C^{1,\beta}_{\mathrm{loc}}\bigl(B_2^n,\mathcal A_q(\mathbb R^k)\bigr)\), \(\|Dw\|_{C^0(B_2^n)}\leq L\), and \(\mathbf v(w)\) stationary in \(B_2^n\times\mathbb R^k\), and that \(V=V_1+V_2\) in \(B_2^{n+k}\), where \(V_1\) and \(V_2\) have disjoint supports. We want to show that \((\eta_{0,1/2})_\#(V_i\mres B_1^{n+k})\in\mathcal V_L\) for \(i=1,2\). If \(V_i\mres B_1^{n+k}=0\) for some \(i\), the claim follows immediately, so we may assume that both restrictions are nonzero.
Iterating the Lipschitz decomposition of \cite[Proposition~1.6]{DLS11}, we may write \(w=\sum_{\lambda=1}^N w^\lambda\), where the \(w^\lambda\) have fixed multiplicities, \(\|Dw^\lambda\|_{C^0(B^n_2)}\leq L\), and \(\operatorname{diam}\bigl(\bigcup_{x\in B_2^n}\operatorname{spt}w^\lambda(x)\bigr)\leq C(n,k,Q)L\). Let \(\Lambda_*:=\{\lambda:\operatorname{spt}\mathbf v(w^\lambda)\cap B_1^{n+k}\neq\varnothing\}\) and set \(w^*:=\sum_{\lambda\in\Lambda_*}w^\lambda\). After decreasing \(L\) if necessary, for every \(\lambda\in\Lambda_*\), $\spt \mathbf{v}(w^\lambda|_{B_1^n})$ is contained in \(B_2^{n+k}\), and hence \(V\mres B_1^{n+k}=\mathbf v(w^*)\mres B_1^{n+k}\).
For \(i=1,2\) and \(x\in B_1^n\), choose any labelling \(w^*(x)=\sum_{\ell=1}^{q_*}\llbracket w^{*,\ell}(x)\rrbracket\) and set $w^i(x):=\sum_{\substack{1\leq\ell\leq q_*, (x,w^{*,\ell}(x))\in\operatorname{spt}V_i}}\llbracket w^{*,\ell}(x)\rrbracket$.
This definition is independent of the labelling. Since \(\operatorname{spt}V_1\) and \(\operatorname{spt}V_2\) are relatively closed in \(B_2^{n+k}\) and disjoint, and every point of \(\operatorname{spt}\mathbf v(w^*|_{B_1^n})\) lies in their union, the local value-splitting property of \(w^*\) shows that the multiplicity of each \(w^i\) is locally constant. Hence, since \(B_1^n\) is connected and \(V_i\mres B_1^{n+k}\neq0\), \(w^i\in C^{1,\beta}\bigl(B_1^n,\mathcal A_{q_i}(\mathbb R^k)\bigr)\) for some \(1\leq q_i\leq q\), with \(\|Dw^i\|_{C^0(B^n_1)}\leq L\), and \(V_i\mres B_1^{n+k}=\mathbf v(w^i)\mres B_1^{n+k}\).
By localization of the first variation, \(\mathbf v(w^i)\) is stationary in \(B_1^n\times\mathbb R^k\). Setting \(\widetilde w^i(z):=2w^i(z/2)\), we therefore have \(\mathbf v(\widetilde w^i)\mres B_2^{n+k}=(\eta_{0,1/2})_\#(V_i\mres B_1^{n+k})\in\mathcal V_L\), as required.

It remains to verify the \(\varepsilon\)-regularity property. After decreasing \(L=L(n,k,Q,\beta)>0\) if necessary, we shall show that there exists \(\varepsilon=\varepsilon(n,k,Q,\beta)>0\) such that the conclusions of \cite[Definition~1.2]{KMW26} hold whenever \(V=(\Gamma_\#\mathbf v(w))\mres B_2^{n+k}\in\mathcal V_L\setminus\{0\}\) and an \(n\)-plane \(P\) satisfy \((\omega_n2^n)^{-1}\|V\|(B_2^{n+k})<Q+\frac12\), \(\widehat E_{V,P}<\varepsilon\), where \(\widehat E_{V,P}^2:=\int_{\mathcal C_1(P)}\operatorname{dist}^2(X,P)\,d\|V\|(X)\) and \(\mathcal C_r(P):=\{X\in\mathbb R^{n+k}:|\pi_P(X)|<r\}\), together with the additional lower mass bound \(\omega_n^{-1}\|V\|(\mathcal C_1(P))\geq\frac12\).
Since the hypotheses and conclusions of the \(\varepsilon\)-regularity property are invariant under rotations, after applying \(\Gamma^{-1}\) and replacing \(P\) by \(\Gamma^{-1}P\), we may assume that \(\Gamma=\operatorname{Id}\). Thus \(V=\mathbf v(w)\mres B_2^{n+k}\) is represented as an \(L\)-Lipschitz graph over \(\mathbb R^n\times\{0\}\), whereas the excess is measured relative to the possibly different plane \(P\).

We first select the components of the graph which contribute to \(V\mres B_{15/8}^{n+k}\). As in the verification of the local disjoint-decomposition property, iterating \cite[Proposition~1.6]{DLS11} gives \(w=\sum_{\lambda=1}^N w^\lambda\), where the \(w^\lambda\) are locally value-separated \(C^{1,\beta}\) submaps, each satisfying \(\|Dw^\lambda\|_{C^0(B^n_2)}\leq L\), and \(\operatorname{diam}\left(\bigcup_{x\in B_2^n}\operatorname{spt}w^\lambda(x)\right)\leq C(n,k,Q)L\).
By localization of the first variation, each \(\mathbf{v}(w^\lambda)\) is stationary in \(B_2^n\times\mathbb R^k\).
Let \(\Lambda_*:=\left\{\lambda\in \{1,...,N\}:\spt \mathbf{v}(w^\lambda) \cap B_{15/8}^{n+k}\neq\varnothing\right\}\), \(w^*:=\sum_{\lambda\in\Lambda_*}w^\lambda\).
The lower mass bound and the smallness of \(\widehat E_{V,P}\) imply that \(\Lambda_*\neq\varnothing\) (indeed, otherwise \(V\mres B_{15/8}^{n+k}=0\), and hence \(\widehat E_{V,P}^2\geq \left(\left(\frac{15}{8}\right)^2-1\right)\|V\|(\mathcal C_1(P))\), contrary to the lower mass bound when \(\widehat E_{V,P}^2\) is sufficiently small). Thus \(w^*\in C^{1,\beta}_{\mathrm{loc}}\bigl(B_2^n,\mathcal A_{q^*}(\mathbb R^k)\bigr)\) for some \(1\leq q^*\leq q\), with \(\|Dw^*\|_{C^0(B^n_2)}\leq L\), and \(\mathbf v(w^*)\) is stationary in \(B_2^n\times\mathbb R^k\).

We next claim that \(P\) can be made arbitrarily close to \(\mathbb R^n\times\{0\}\) by choosing \(L\) and \(\varepsilon\) sufficiently small.
Indeed, otherwise there would exist \(L_h\downarrow0\), varifolds \(V_h=\mathbf v(w_h)\mres B_2^{n+k}\), and planes \(P_h\) such that \(\|Dw_h\|_{C^0(B^n_2)}\leq L_h\), \(\widehat E_{V_h,P_h}\to0\), and \(\|\pi_{P_h}-\pi_{\mathbb R^n\times\{0\}}\|\geq\eta\) for some fixed \(\eta>0\). After passing to a subsequence, \(P_h\to P\) and \(V_h\) converges locally as varifolds to a stationary integral varifold \(V_\infty\). The lower mass bound and \(\widehat E_{V_h,P_h}\to0\) ensure that \(V_\infty\neq0\). Since the slopes of the graphs tend to zero, the tangent plane to \(V_\infty\) is \(\mathbb R^n\times\{0\}\) at \(\|V_\infty\|\)-almost every point. On the other hand, \(\widehat E_{V_h,P_h}\to0\) implies that \(V_\infty\) is supported in \(P\) in the corresponding cylinder. The constancy theorem therefore shows that \(V_\infty\) is a nonzero integer multiple of \(|P|\). Hence \(P=\mathbb R^n\times\{0\}\), contradicting \(\|\pi_P-\pi_{\mathbb R^n\times\{0\}}\|\geq\eta\) and proving the claim.

Hence, after decreasing \(L\) and \(\varepsilon\) if necessary, we may assume that \(P\) is sufficiently close to \(\mathbb R^n\times\{0\}\) that there exists \(R\in O(n+k)\), as close to the identity as required in Lemma~\ref{lem:C1alpha-reparametrization} and in the argument below, such that \(R(P)=\mathbb R^n\times\{0\}\).
Set \(\mathcal C_r:=B_r^n\times\mathbb R^k\) and \(\overline V:=R_\#V\).

With this choice of \(R\), we claim that, after decreasing \(L\) and \(\varepsilon\) if necessary, \(\|w^*\|_{C^0(B^n_2)}\) is as small as required in Lemma~\ref{lem:C1alpha-reparametrization}. Indeed, otherwise there would exist sequences \(L_h\to0\), \(V_h=\mathbf v(w_h)\mres B_2^{n+k}\), \(P_h\), \(R_h\to I\) with \(R_h(P_h)=\mathbb R^n\times\{0\}\), and corresponding decompositions \(w_h=\sum_{\lambda=1}^{N_h}w_h^\lambda\), such that \(\widehat E_{V_h,P_h}\to0\) but \(\|w_h^*\|_{C^0(B^n_2)}\) stays bounded away from zero. By the component-diameter estimate, after passing to a subsequence there are \(\lambda_h\in\Lambda_{*,h}\) and \(a_h\in\operatorname{spt}w_h^{\lambda_h}(0)\) such that \(|a_h|\) stays bounded away from zero. Since \(\lambda_h\in\Lambda_{*,h}\), we may choose \((x_h,b_h)\in\operatorname{spt}\mathbf v(w_h^{\lambda_h})\cap B_{15/8}^{n+k}\), and the component-diameter estimate gives \(|a_h-b_h|\leq CL_h\). Hence a fixed portion of the graph of \(w_h^{\lambda_h}\) near \(0\) remains a fixed positive distance from \(\mathbb R^n\times\{0\}\), while, for \(h\) sufficiently large, its image under \(R_h\) lies in \(B_2^{n+k}\cap\mathcal C_1\). This gives a fixed positive lower bound for \(\widehat E_{V_h,P_h}\), a contradiction. This contradiction proves the claim.

Consequently, after restricting to \(B_{7/4}^n\) and making a fixed graphical rescaling, Lemma~\ref{lem:C1alpha-reparametrization}, with \(\alpha=\beta\), applies to \(w^*\) and \(R\). We obtain a map $\widetilde w\in C^{1,\beta}\bigl(B_{7/8}^n,\mathcal A_{q^*}(\mathbb R^k)\bigr)$ whose graph agrees with the corresponding portion of \(R\big(\spt \mathbf{v}(w^*)\big)\), and whose slope lies in the range of applicability of the Schauder estimate for the area system. Then $\mathbf{v}(\widetilde w)$ is stationary in \(\mathcal C_{7/8}\). By decreasing \(L\) and \(\varepsilon\) once more if necessary, \(\spt \mathbf{v}(\widetilde w|_{B^n_{7/8}})\) is contained in \(B_{15/8}^{n+k}\). Therefore, by the definition of \(\Lambda_*\), $\bigl(\overline V\mres B_{15/8}^{n+k}\bigr)\mres\mathcal C_{3/4} = \mathbf v\bigl(\widetilde w|_{B^n_{3/4}}\bigr)$.
Moreover, the definition of the graph varifold gives $\int_{B^n_{7/8}}|\widetilde w|^2 \leq \int_{\mathcal C_1}\operatorname{dist}^2(Z,\R^n\times \{0\})\,d\|\overline V\|(Z) = \widehat E^2_{V,P}$.
The interior Schauder estimate for the area system consequently yields
\[
\|\widetilde w\|_{C^{1,\beta}(B^n_{3/4})}
\leq
C \widehat E_{V,P}.
\]

The preceding varifold identity and estimate give condition \emph{(A)} of \cite[Definition~1.2]{KMW26}, with multiplicity \(1\leq q^*\leq q\); the \(C^{1,\beta}\) graphical structure gives condition \emph{(B)}; and the Hölder continuity of the tangent planes, together with the corresponding Taylor estimate, gives the uniqueness of the tangent cones and the two quantitative estimates in condition \emph{(C)}.

We have thus verified that \(\mathcal V_L\) satisfies all the hypotheses of \cite[Theorem~A]{KMW26}.
We now apply \cite[Theorem~A]{KMW26} near the points of \(K_u\). Let \(x\in K_u\), write $j_1u(x) = Q\llbracket(a,A)\rrbracket$, and set $X:=(x,a)$.
By Lemma~\ref{lem:collapsed-jet-reparametrization}, after translating \(X\) to the origin and applying an ambient rotation sending \(P_A\) onto \(\mathbb R^n\times\{0\}\), the graph varifold of \(u\) is locally represented by a map $w\in C^{1,\alpha} \bigl(B^n_\rho,\mathcal A_Q(\mathbb R^k)\bigr)$ for some \(\rho>0\), satisfying $w(0)=Q\llbracket0\rrbracket$, $Dw(0)=Q\llbracket0\rrbracket$.
After a fixed rescaling we may assume that \(w\) is defined on \(B^n_2\).
For $w_r(z):=r^{-1}w(rz)$,  the \(C^{1,\alpha}\) continuity of the first jet and the identity \(Dw(0)=Q\llbracket0\rrbracket\) give $\|Dw_r\|_{C^0(B^n_2)} \leq Cr^\alpha \longrightarrow0$ as $r\downarrow0$.
Thus, for \(r\) sufficiently small, the corresponding translated,
rotated, and rescaled graph belongs to \(\mathcal V_L\).
Now let $x\in\mathcal B_u\cap K_u$.
At \(X=(x,a)\) one has $\Theta_{\mathbf v(u)}(X)=Q$, and the tangent cone is $Q|P_A|$.
Denote by \(\mathcal B_{\mathbf v(u)}\) the branch set of the varifold as in \cite{KMW26}, and let $\pi:\mathbb R^{n+k}\longrightarrow\mathbb R^n$ be the orthogonal projection.
We claim that \(X\in\mathcal B_{\mathbf v(u)}\). Indeed, \(Q|P_A|\) is a tangent cone to \(\mathbf v(u)\) at \(X\), and hence, by the definition of the branch set for varifolds, it is enough to show that \(X\) is singular. If \(X\) were regular, then, since \(\Theta_{\mathbf v(u)}(X)=Q\), the varifold would agree near \(X\) with a smooth embedded submanifold of multiplicity \(Q\). Its tangent plane \(P_A\) is graphical over \(\mathbb R^n\), so the inverse function theorem would represent this submanifold locally as the graph of a smooth single-valued map \(u_0\). The graphical identity would then give \(u=Q\llbracket u_0\rrbracket\) locally, contrary to \(x\in\mathcal B_u\).
Consequently, $\mathcal B_u\cap K_u \subset \pi\left( \mathcal B_{\mathbf v(u)} \cap \{\Theta_{\mathbf v(u)}=Q\} \right)$.
Theorem~A of \cite{KMW26}, applied in the normalized neighbourhoods constructed above, together with the fact that \(\pi\) is Lipschitz, gives $\dim_{\mathcal H} (\mathcal B_u\cap K_u) \leq n-2$.
Finally, let $x\in\mathcal B_u\setminus K_u$.
Since \(x\notin K_u\), the first jet of \(u\) at \(x\) has at least two distinct atoms. By the local first-jet splitting, after shrinking to a neighbourhood \(W_x\) of \(x\), $u = \sum_{\gamma=1}^N u^\gamma$, $u^\gamma \in C^{1,\alpha} \bigl(W_x,\mathcal A_{Q_\gamma}(\mathbb R^k)\bigr)$, $Q_\gamma<Q$.
To see that the graph of every \(u^\gamma\) is stationary, fix \(y\in W_x\). Near \(y\), first split \(u\) into its distinct value clusters, to which stationarity localizes. Within each value cluster, use the local first-jet splitting to separate the differential blocks and apply Lemma~\ref{lem:area-stationary-differential-blocks}. Summing the blocks belonging to each \(u^\gamma\), we conclude that \(\mathbf v(u^\gamma)\) is stationary near \(y\). Since \(y\in W_x\) was arbitrary, \(\mathbf v(u^\gamma)\) is stationary in \(W_x\times\mathbb R^k\). Moreover, $\mathcal B_u\cap W_x \subset \bigcup_{\gamma=1}^N \mathcal B_{u^\gamma}$.
The inductive hypothesis therefore gives $\dim_{\mathcal H} (\mathcal B_u\cap W_x) \leq n-2$.
Taking countable subcovers of \(\mathcal B_u\cap K_u\) and \(\mathcal B_u\setminus K_u\) completes the proof.
\end{proof}

\bigskip

\noindent
Department of Pure Mathematics and Mathematical Statistics, University of Cambridge\\
\textit{Email}: \texttt{ml2003@cam.ac.uk}

\end{document}